\documentclass[a4paper,reqno,english]{smfart}
\usepackage{color}
\usepackage{amscd,amsfonts,amssymb,amscd,amsmath,latexsym}
\usepackage[hypertex]{hyperref}
\usepackage[all]{xy}
\makeatletter
\renewcommand{\subsubsection}{\@startsection
{subsubsection}
{1}
{0mm}
{0mm}
{0mm}
{\normalfont\normalsize\itshape}}
\makeatother

\usepackage{mathrsfs}

\author{Ulrich Bunke}\address{NWF I - Mathematik,
Johannes-Kepler-Universit{\"a}t Regensburg,
93040 Regensburg,
GERMANY}\email{ulrich.bunke@mathematik.uni-regensburg.de} 
\author{Thomas Schick}
\address{Mathematisches Institut, Georg-August-Universit{\"a}t G{\"o}ttingen, 
Bunsenstrasse 3, 37073 G{\"o}ttingen,
GERMANY}\email{schick@uni-math.gwdg.de}

 \dedicatory{Dedicated to Jean-Michel Bismut on the occasion
of his 60th birthday}

\title{Smooth K-Theory}

\alttitle{K-theorie differentiable}

\newtheorem{theorem}{Theorem}[section] 
\newtheorem{prop}[theorem]{Proposition}
\newtheorem{lem}[theorem]{Lemma}
\newtheorem{ddd}[theorem]{Definition}
\newtheorem{kor}[theorem]{Corollary}

\newcommand{\PD}{{\tt PD}}

\newcommand{\hocolim}{{\tt hocolim\:}}
\newcommand{\Fr}{{\tt Fr}}
\newcommand{\bQ}{\mathbf{Q}}
\renewcommand{\lim}{{\tt lim\:}}

\renewcommand{\P}{{\mathbb{P}}}

\newcommand{\Z}{\mathbb{Z}}

\newcommand{\bE}{{\bf E}}

\newcommand{\R}{\mathbb{R}}

\newcommand{\Rep}{{\tt Rep}}

\newcommand{\Q}{\mathbb{Q}}

\newcommand{\tR}{{\tt R}}

\newcommand{\bK}{{\bf K}}

\renewcommand{\det}{{\tt det}}
\renewcommand{\sinh}{\mathrm{sinh}}
\newcommand{\K}{\mathbb{K}}

\newcommand{\bB}{\mathbf{B}}

\newcommand{\bL}{{\bf L}}

\newcommand{\C}{\mathbb{C}}
\renewcommand{\Pr}{{\tt Pr}}

\newcommand{\Aut}{{\tt Aut}}

\newcommand{\cM}{\mathcal{M}}

\newcommand{\cE}{\mathcal{E}}

\newcommand{\cV}{\mathcal{V}}

\newcommand{\cW}{\mathcal{W}}

\newcommand{\cK}{\mathcal{K}}

\newcommand{\Hom}{{\tt Hom}}
\newcommand{\vol}{{\tt vol}}
\newcommand{\cO}{\mathcal{O}}
\newcommand{\End}{{\tt End}}

\newcommand{\im}{{\tt im}}

\newcommand{\cF}{\mathcal{F}}

\newcommand{\ee}{{\tt e}}

\newcommand{\id}{{\tt id}}

\newcommand{\nat}{\mathbb{N}}

\def\imath{{i}}

\def\hB{\hspace*{\fill}$\Box$ \newline\noindent}

\newcommand{\ind}{{\tt index}}

\newcommand{\cS}{\mathcal{S}}

\def\hB{\hspace*{\fill}$\Box$ \\[0.5cm]\noindent}
\newcommand{\cL}{\mathcal{L}}
 \newcommand{\cG}{\mathcal{G}}

\newcommand{\cQ}{\mathcal{Q}}

\newcommand{\bW}{\mathbf{W}}
\newcommand{\pr}{{\tt pr}}

\newcommand{\ch}{{\mathbf{ch}}}

\newcommand{\bV}{\mathbf{V}}

\newcommand{\hA}{\hat{\mathbf{A}}}

\begin{document}
\frontmatter
 
\begin{abstract}
In this paper we consider smooth extensions of cohomology theories. In particular we
 construct an analytic multiplicative model of  smooth $K$-theory. We further
 introduce the notion of a smooth $K$-orientation of a proper submersion
 $p\colon W\to B$ and define the associated push-forward $\hat p_!:\hat K(W)\to \hat K(B)$. We show that the push-forward has the expected properties as
 functoriality, compatibility with pull-back diagrams, projection formula and a bordism formula.

We construct a multiplicative lift of the Chern character $\hat \ch:\hat K(B)\to \hat H(B,\Q)$, where $\hat H(B,\Q)$ denotes the smooth extension of rational cohomology, and we show that $\hat \ch$ 
 induces  a rational isomorphism.
 
If $p\colon W\to B$ is a proper submersion with a smooth $K$-orientation, then we define a class 
$A(p)\in \hat H^{ev}(W,\Q)$ 
(see Lemma \ref{weldioe})
and the modified push-forward
$\hat p_!^A:=\hat p_!(A(p)\cup\dots):\hat H(W,\Q)\to \hat H(B,\Q).$
One of our main results lifts the cohomological version of the Atiyah-Singer index theorem to 
smooth cohomology. It states that $\hat p^A_!\circ \hat \ch=\hat \ch\circ \hat
p_!$.
\end{abstract}

\begin{altabstract}
  Nous considerons les extensions differentiables des theories de
  cohomology. En particulier, nous construisons un mod\`ele analytique et
  avec multiplication de la K-theorie differentiable. Nous
  introduisons le concept d'une K-orientation differentiable d'une
  submersion propre $p\colon W\to B$. Nous contruisons une application
  d'integration associ\'e $\hat p_!:\hat K(W)\to \hat K(B)$; et nous
  demontrons les propri\'et\'es attendues comme 
 functorialit\'e, compatibilit\'e avec pull-back, formules de projection
  et de bordism. 

Nous construisons une version differentiable du charact\`ere de Chern
$\hat \ch:\hat K(B)\to \hat H(B,\Q)$, o\`u $\hat H(B,\Q)$ est une
extension differentiable de la
cohomologie rationelle, et nous
demontrons que  $\hat \ch$  
 induit un isomorphisme rationel.   

Si $p\colon W\to B$ est une submersion propre avec une $K$-orientation
differentiable, nous definissons une classe 
$A(p)\in \hat H^{ev}(W,\Q)$ 
(compare Lemma \ref{weldioe})
et une application d'integration modifi\'e
$\hat p_!^A:=\hat p_!(A(p)\cup\dots):\hat H(W,\Q)\to \hat H(B,\Q).$
Un de nos resultats principales est une version en cohomologie
differentiable du theor\`eme d'indice de Atiyah-Singer. Cette version
dits que $\hat p^A_!\circ \hat \ch=\hat \ch\circ \hat 
p_!$.
\end{altabstract}


\thanks{Thomas Schick was funded by Courant Research Center G\"ottingen ``Higher order structures in 
   mathematics'' via the German Initiative of Excellence}
\maketitle

\tableofcontents

\mainmatter
\section{Introduction}
\subsection{The main results}

\subsubsection{}

In this paper we construct a model of a smooth extension of the generalized
 cohomology theory $K$, complex $K$-theory. Historically, the concept of
 smooth extensions of a cohomology theory started with smooth integral
 cohomology \cite{MR827262}, also called real Deligne cohomology, see
 \cite{MR1197353}. A second, geometric model of smooth integral cohomology is
 given in \cite{MR827262}, where the smooth integral cohomology classes
 were called differential characters.
 One important motivation of its definition was that one can associate natural
 differential characters to hermitean vector bundles with connection which
 refine the Chern classes. The differential character in degree two even
 classifies hermitean line bundles with connection up to isomorphism.
The multiplicative structure of smooth integral cohomology also encodes cohomology operations, see \cite{math.AT/0411043}.

 The holomorphic counterpart of the theory became an important ingredient of arithmetic geometry.

\subsubsection{}

Motivated by the problem of setting up  lagrangians for  quantum field theories with differential form field strength  it was argued in \cite{MR1769477}, \cite{hep-th/0011220} that one may need  smooth extensions of other generalized cohomology theories. The choice of the generalized cohomology theory is here dictated by a charge quantization condition, which mathematically is reflected by a lattice in real cohomology. Let $N$
be a graded real vector space such that the field strength lives in $\Omega_{d=0}(B)\otimes N$, the closed forms on the manifold $B$  with coefficients in $N$.
Let $L(B)\subset H(B,N)$ be the lattice given by the charge quantization condition on $B$.
Then one looks for a generalized cohomology theory $h$ and a natural transformation
$c:h(B)\to H(B,N)$ such that $c(h(B))=L(B)$. It was argued in \cite{MR1769477}, \cite{hep-th/0011220} that the fields of the theory should be considered as cycles for a smooth extension $\hat h$ of the pair $(h,c)$.
For example, if $N=\R$ and the charge quantization leads to $L(B)=\im(H(B,\Z)\to H(B,\R))$, then the relevant smooth extension could be the smooth integral cohomology theory of \cite{MR827262}.

In Subsection \ref{intrtrr} we will introduce the notion of a smooth extension in an axiomatic way.

\subsubsection{}

\cite{hep-th/0011220} proposes in particular to consider smooth extensions of complex and real versions of $K$-theory. In that paper it was furthermore indicated how cycle models of such smooth extensions could look like. The goal of the present paper is to carry through this program in the case of complex $K$-theory.

\subsubsection{}

In the remainder of the present subsection we describe, expanding the abstract,  our main results. 
The main ingredient is a construction of an analytic model of  smooth $K$-theory\footnote{or differentiable $K$-theory in the language of other authors} using cycles and relations.
\subsubsection{}

Our philosophy for the construction of smooth $K$-theory is that a vector bundle with connection or a family of Dirac operators with some additional geometry should represent a smooth $K$-theory class tautologically. In this way we follow the outline   in \cite{hep-th/0011220}. Our class of cycles is quite big. This makes the construction of smooth $K$-theory classes or transformations to smooth $K$-theory easy, but it complicates the verification that certain cycle level constructions out of smooth $K$-theory are well-defined. The great advantage of our choice is that the constructions of the product and the push-forward on the level of cycles are of differential geometric nature.

More precisely we use the notion of a geometric family which was introduced in \cite{math.DG/0201112}
in order to subsume all geometric data needed to define a Bismut super-connection in one notion.
A cycle of the smooth $K$-theory $\hat K(B)$ of a compact manifold $B$ is a pair $(\cE,\rho)$ of a geometric family $\cE$ and an element
$\rho\in \Omega(B)/\im(d)$, see Section \ref{tztewqze}. Therefore, cycles are differential geometric objects. Secondary spectral
invariants from local index theory, namely $\eta$-forms, enter the definition of the relations (see Definition \ref{uuu1}).
The first main result is that our construction really yields a smooth extension in the sense of Definition 
\ref{ddd556}.

\subsubsection{}

Our smooth $K$-theory $\hat K(B)$ is a contravariant functor on the category of compact smooth manifolds
(possibly with boundary)
with values in the category of  $\Z/2\Z$-graded rings. This multiplicative structure is expected since $K$-theory is a multiplicative generalized cohomology theory, and the Chern character is multiplicative, too.
As said above, the construction of the product on the level of cycles (Definition \ref{proddefin}) is of differential-geometric nature. Analysis enters the verification of well-definedness. The main result is here that our construction produces a multiplicative smooth extension in the sense of Definition \ref{multdef1}. 

\subsubsection{}

Let us consider a proper submersion $p\colon W\to B$ with closed fibres which has a topological $K$-orientation.
Then we have a push-forward $p_!\colon K(W)\to K(B)$, and it is an important part of the theory to extend
this push-forward to the smooth extension.

For this purpose one needs a smooth refinement of the notion of a $K$-orientation which we introduce in 
\ref{smmmmzuz}. We then define the associated push-forward $\hat p_!\colon
\hat K(W)\to \hat K(B)$, again by a differential-geometric construction on the
level of cycles (\ref{eq300}).
 We show that the push-forward has the expected properties: functoriality, compatibility with pull-back diagrams,
projection formula, bordism formula.

\subsubsection{}

Let $\bV=(V,h^V,\nabla^V)$ be a hermitean vector bundle with connection. In
 \cite{MR827262} a smooth refinement $\hat \ch(\bV)\in  \hat H(B,\Q)$  of the  Chern character  
was constructed. In the present paper
we construct a lift of the Chern character $\ch\colon K(B)\to H(B,\Q)$ 
to a multiplicative natural transformation of smooth cohomology theories (see (\ref{eq5001}))
 $$\hat \ch\colon \hat K(B)\to \hat H(B,\Q)$$
such that
$\hat \ch(\bV)=\hat \ch([\cV,0])$, where $\cV$ is the geometric family determined by $\bV$.
The Chern character
induces a natural  isomorphism of $\Z/2\Z$-graded rings 
$$\hat K(B)\otimes\Q\stackrel{\sim}{\to} \hat H(B,\Q)$$
(Proposition  \ref{qiso}).

\subsubsection{}\label{uiefwefwef}

If $p\colon W\to B$ is a proper submersion with a smooth $K$-orientation, then we define a class (see Lemma \ref{weldioe})
$A(p)\in \hat H^{ev}(W,\Q)$
and the modified push-forward
$$\hat p_!^A:=\hat p_!(A(p)\cup\dots)\colon \hat H(W,\Q)\to \hat H(B,\Q)\ .$$
Our index  theorem \ref{main} lifts the characteristic class version of the Atiyah-Singer index theorem to 
smooth cohomology. It states that the diagram
$$\xymatrix{\hat K(W)\ar[d]^{\hat p_!}\ar[r]^{\hat \ch}&\hat
  H(W,\Q)\ar[d]^{\hat p^A_! }\\\hat K(B)\ar[r]^{\hat \ch}&\hat H(B,\Q)}$$
commutes.

%

\subsubsection{}

In Subsection \ref{intrtrr} we present a short introduction to  the theory of smooth extensions of generalized cohomology theories. In Subsection \ref{gabgh} we review in some detail the literature about variants of smooth $K$-theory and associated index theorems. 
In Section \ref{tztewqze} we present the  cycle model of smooth
$K$-theory. The main result is the verification that our construction
satisfies the axioms given below.
Section \ref{idwiqdqwwd} is devoted to the push-forward. We introduce the notion of a smooth $K$-orientation, and we  construct the push-forward on the cycle level. The main results are that the push-forward descends to smooth $K$-theory, and the verification of  its functorial  properties. In Section \ref{jkjdkdqwdqdwd} we discuss the ring structure in smooth $K$-theory and its compatibility with the push-forward.
Section \ref{sec5} presents a collection of natural constructions of smooth $K$-theory classes. 
In Section \ref{sqsjhss} we construct the Chern character and prove the smooth index theorem.

\subsection{A short introduction to smooth cohomology theories}\label{intrtrr}

\subsubsection{}\label{cher}

The first example of a smooth cohomology theory appeared under the name
Cheeger-Simons differential characters in \cite{MR827262}. Given a discrete
subring $\tR\subset \R$ we have a functor\footnote{In the literature, this group
  is sometimes denoted by $\hat H(B,\R/\tR)$, possibly with a degree-shift by
  one.} 
$B\mapsto \hat H(B,\tR)$ from smooth manifolds to $\Z$-graded rings. It comes
with natural transformations
\begin{enumerate}
\item $R\colon  \hat H(B,\tR)\to \Omega_{d=0}(B)\quad $ (curvature)
\item $I\colon  \hat H(B,\tR)\to  H(B,\tR)\quad $ (forget smooth data)
\item $a\colon \Omega(B)/\im(d)\to  \hat H(B,\tR)\quad $ (action of forms).
\end{enumerate}
Here $\Omega(B)$ and $\Omega_{d=0}(B)$ denote the space of smooth real differential forms and its subspace of closed forms.
The map $a$ is of degree $1$. Furthermore, one has the following properties, all shown in \cite{MR827262}.
\begin{enumerate}
\item The following diagram commutes
$$\xymatrix{\hat H(B,\tR)\ar[d]^R\ar[r]^I& H(B,\tR)\ar[d]^{\tR\to \R}\\
\Omega_{d=0}(B)\ar[r]^{dR}&H(B,\R)}\ ,$$
where $dR$ is the de Rham homomorphism.
\item $R$ and $I$ are ring homomorphisms.
\item $R\circ a=d$,
\item $a(\omega)\cup x=a(\omega\wedge R(x))$, $\forall x\in \hat H(B,\tR)$,
  $\forall \omega\in\Omega(B)/\im(d)$,
\item The sequence
\begin{equation}\label{eq9000}H(B,\tR)\stackrel{}{\to} \Omega(B)/\im(d)\stackrel{a}{\to} \hat H(B,\tR)\stackrel{I}{\to} H(B,\tR)\to 0\end{equation}
is exact.
\end{enumerate}

\subsubsection{}\label{axi1}

Cheeger-Simons differential characters are the first example of a more general structure which is described for instance in the first section of  \cite{hep-th/0011220}. In view of our
 constructions of examples for this structure in the case of bordism theories and $K$-theory, and the presence of completely different  pictures like \cite{MR2192936} we think that an axiomatic description of smooth cohomology theories
is useful.

  Let $N$ be a $\Z$-graded vector space over $\R$. We consider  a generalized cohomology theory $h$ with a  natural transformation of cohomology theories
$c\colon h(B)\to H(B,N)$. The natural universal example is given by $N:=h^*\otimes \R$, where $c$ is the canonical transformation. Let $\Omega(B,N):=\Omega(B)\otimes_\R N$.
To a pair $(h,c)$ we associate the notion of a smooth extension $\hat h$. 
Note that manifolds in the present paper may have boundaries. 
\begin{ddd}\label{ddd556}
A smooth extension of the pair $(h,c)$ is a functor $B\to\hat h(B)$ from the category of compact smooth manifolds
to $\Z$-graded groups together with natural transformations
\begin{enumerate}
\item $R\colon  \hat h(B)\to \Omega_{d=0}(B,N)$ (curvature)
\item $I\colon  \hat h(B)\to  h(B)$ (forget smooth data)
\item $a\colon \Omega(B,N)/\im(d)\to  \hat h(B)$ (action of forms)\ .
\end{enumerate}
These transformations are required to satisfy the following axioms:
 \begin{enumerate}
\item The following diagram commutes
$$\xymatrix{\hat h(B)\ar[d]^R\ar[r]^I& h(B)\ar[d]^{c}\\
\Omega_{d=0}(B,N)\ar[r]^{dR}&H(B,N)}\ .$$
\item \begin{equation}\label{drgl}
R\circ a=d\ .
\end{equation}
\item $a$ is of degree $1$.
\item The sequence
\begin{equation}\label{exax}
h(B)\stackrel{c}{\to} \Omega(B,N)/\im(d)\stackrel{a}{\to} \hat h(B)\stackrel{I}{\to} h(B)\to 0\ .
\end{equation}
is exact.
\end{enumerate}
\end{ddd}
The Cheeger-Simons smooth cohomology $B\mapsto \hat H(B,\tR)$ considered in \ref{cher} is the smooth extension of the pair $(H(\dots,\tR),i)$, where
$i\colon H(B,\tR)\to H(B,\R)$ is induced by the inclusion $\tR\to \R$. The main object of the present paper, smooth $K$-theory, is a smooth extension of the pair $(K,\ch_\R)$, and we actually work with the obvious $\Z/2\Z$-graded version of these axioms.

\subsubsection{}

If $h$ is a multiplicative cohomology theory, then one can consider a $\Z$-graded ring $N$ over $\R$ and a  multiplicative transformation $c\colon h(B)\to H(B,N)$. In this case is makes sense to talk about a multiplicative smooth extension $\hat h$ of  $(h,c)$.
\begin{ddd}\label{multdef1}
A smooth extension $\hat h$ of $(h,c)$ is called multiplicative, if $\hat h$ together with the transformations $R,I,a$ is a smooth extension of $(h,c)$, and in addition
\begin{enumerate}
\item $\hat h$ is a functor to $\Z$-graded rings,
\item $R$ and $I$ are multiplicative, 
\item $a(\omega)\cup x=a(\omega\wedge R(x))$ for $x\in \hat h(B)$ and $\omega\in\Omega(B,N)/\im(d)$. 
\end{enumerate}
\end{ddd}
The smooth extension $\hat H(\dots ,\tR)$ of ordinary cohomology $H(\dots ,\tR)$ with coefficients in a subring $\tR\subset \R$ considered in \ref{cher} is multiplicative.  The smooth extension $\hat K$ of $K$-theory which we construct in the present paper is multiplicative, too.

\subsubsection{}

Consider two pairs $(h_i,c_i)$, $i=0,1$ as in \ref{axi1} and  a transformation of generalized cohomology theories $u\colon h_0\to h_1$  such that $c_1\circ h=c_0$. Then we define the notion of a natural transformation of smooth cohomology theories which refines $u$.
\begin{ddd}\label{natdef12}
A natural transformation of smooth extensions $\hat u\colon \hat h_0\to \hat h_1$
which refines $u$ is a natural transformation $\hat u\colon \hat h_0(B)\to \hat h_1(B)$ such that
the following diagram commutes:
$$
\xymatrix{\Omega(B,N)/\im(d)\ar[r]^a\ar@{=}[d]&\hat h_0(B)\ar@/^1cm/[rr]^R\ar[r]^I\ar[d]^{\hat u}&h_0(B)\ar[d]^{u}&\Omega_{d=0}(B,N)\ar@{=}[d]\\\Omega(B,N)/\im(d)\ar[r]^a&\hat h_1(B)\ar[r]^I\ar@/_1cm/[rr]^R&h_1(B)&\Omega_{d=0}(B,N)}\   .
$$
\end{ddd}

Our main example is the Chern character $$\hat \ch\colon \hat K(B)\to \hat H(B,\Q)$$ which refines the ordinary Chern character $\ch \colon K(B)\to H(B,\Q)$. The Chern character and its smooth refinements are actually multiplicative.

\subsubsection{}

One can show that two smooth extensions of $(H(\dots, R),i)$ are canonically
isomorphic (see \cite{math.AT/0701077} and {\cite[Section 4]{bs2009}}). There is no
uniqueness result for arbitrary pairs $(h,c)$.  {Appropriate examples in the
  case of $K$-theory are presented in \cite[Section 6]{bs2009}}. In order to fix the uniqueness
problem one has to require more conditions,
which are all quite natural. 

The projection $\pr_2\colon S^1\times B\to B$ has a   canonical smooth $K$-orientation
(see \ref{pullpush1} for details). 
Hence we have a push-forward $(\hat\pr_2)_!\colon \hat K(S^1\times B)\to \hat K(B)$ (see Definition \ref{ddd1}). This map plays the role of the suspension 
for the smooth extension. It is natural in $B$, and the following diagram commutes (see Proposition \ref{mainprop})
\begin{equation}\label{eq30001}\xymatrix{\Omega(S^1\times B)/\im(d)\ar[d]^{\int_{S^1\times B/B}}\ar[r]^a&\hat K(S^1\times B)\ar@/^1cm/[rr]^R\ar[d]^{(\hat\pr_2)_!}\ar[r]^I&K(B)\ar[d]^{(\pr_2)_!}&\Omega(S^1\times B)\ar[d]^{\int_{S^1\times B/B}}\\\Omega(B)/\im(d)\ar[r]^a&\hat K( B)\ar@/_1cm/[rr]^R\ar[r]^I&K(B)&\Omega(B)}\ .
\end{equation}
Furthermore, it satisfies (see \ref{pullpush0})
\begin{equation}\label{nullww}
(\hat\pr_2)_!\circ \pr_2^*=0\ .
\end{equation}
\textcolor{black}{We have the following theorem, also discovered by Wiethaup.}
\begin{theorem}[\textcolor{black}{\cite[Section 3, Section 4]{bs2009}}]\label{wieth}
There is a unique (up to isomorphism)
smooth extension of the pair $(K,\ch_\R)$ for which in addition the push-forward along $\pr_2\colon S^1\times B\to B$ is defined, is natural in $B$, satisfies (\ref{nullww}), and is such that  (\ref{eq30001}) commutes. If we require the isomorphism to preserve $(\hat\pr_2)_!$, then
it is also unique.
\end{theorem}

\subsubsection{}

The theory of \cite{MR2192936} gives the following general existence result.
\begin{theorem}[\cite{MR2192936}]
For every pair $(h,c)$ of a generalized cohomology theory and a natural transformation $h\to HN$ there exists a smooth extension $\hat h$ in the sense of Definition \ref{ddd556}.
\end{theorem}

A similar general result about multiplicative extensions is not known. Besides smooth extensions of ordinary
cohomology and $K$-theory we have a collection of multiplicative extensions of
bordism theories, again by an an explicit construction in a cycle model. The
details \textcolor{black}{can be found in {\cite{BSSW07}}}.

\subsubsection{}

Let us now assume that $(h,c)$ is multiplicative, and that
 $\hat h$ is a multiplicative smooth extension of the  pair $(h,c)$. Let $p\colon W\to B$ be a proper submersion with closed fibres. An $h$-orientation of
$p$ is given by a collection of compatible choices of $h$-Thom classes on representatives of the stable normal bundle of $p$.
Equivalently, we can fix a Thom class on the vertical tangent bundle, and we will adopt this point of view in the present paper. If $p$ is $h$-oriented, then we have a push-forward
$$p_!\colon h(W)\to h(B)\ .$$ It is an inportant question for applications and calculations   how one can lift the push-forward
to the smooth extensions. 

In the case of smooth ordinary cohomology with coefficients in $R$ it turns
out that an ordinary orientation of $p$ suffices in order to define $\hat p_!\colon \hat H(W,R)\to \hat H(B,R)$. This push-forward has been considered e.g.~in
\cite{MR1197353}, \cite{MR2179587}, \cite{Koehler}. We refer to \ref{ratdel}
for more details.

A push-forward for more general pairs $(h,c)$ has been considered in 
\cite{MR2192936} {without a discussion of  functorial properties}.

\subsubsection{}

The philosophy in the present paper is that the push-forward in $K$-theory is realized analytically using families
of fibre-wise Dirac operators. Therefore, in the present paper a smooth $K$-orientation is given by a collection
of geometric data which allows to define the push-forward on the level of
cycles, which are given by families of Dirac type operators. We add a differential form to the data in order to capture the behaviour under deformations.

\subsubsection{}

We have cycle models of multiplicative smooth extensions of bordism theories
$\Omega^{G}$, where $G$ in particular can be $SO,Spin,U, Spin^c$, see \cite{BSSW07}.
In these examples the natural transformation $c$ is the genus associated to a formal power series
$\phi(x)=1+a_1x+\dots$ with coefficients in some graded ring. These bordism theories admit a theory of orientations and push-forward which is very similar  to the case of $K$-theory. Concerning the product and the integration
bordism theories turn out to be much simpler than ordinary cohomology.
Motivated by this fact,
in a joint project with M. Kreck  we develop a bordism like version of the smooth extension of integral cohomology based on the notion of orientifolds.

We also have an equivariant version of the theory of the present paper for finite groups which will be presented in a future publication.

\subsection{Related constructions}\label{gabgh}

\subsubsection{}

Recall that \cite{MR2192936} provides a topological construction of smooth $K$-theory.
In this subsection we review the literature about analytic variants of smooth $K$-theory and related index theorems.
Note that we will completely ignore the development of holomorphic variants which are more related to arithmetic questions than to topology. 
This subsection will use the language which is set up later in the paper. It should be read in detail only after obtaining some familiarity with the main definitions
(though we tried to give sufficiently many forward references).

\subsubsection{}

Let $p\colon W\to B$ be a proper submersion with closed fibres. To give a $K$-orientation of $p$ is equivalent to give a $Spin^c$-structure on its vertical bundle $T^vp$. The $K$-orientation of $p$ yields, by a stable homotopy construction, a push-forward
$p_!\colon K(W)\to K(B)$.
Let $\hA(T^vp)$ denote the $\hA$-class of the vertical bundle, and let
$c_1(L^2)\in H^2(W,\Z)$ be the cohomology class determined by the $Spin^c$-structure (see \ref{l2def}).
The "index theorem for families" in the characteristic class version states that
$$\ch(p_!(x))=\int_{W/B} \hA(T^vp)\cup e^{\frac{1}{2}c_1(L^2)}\cup
\ch(x),\qquad\forall x\in K(W).$$
 If one realizes the push-forward in an analytic model, then
this statement is indeed an index theorem for families of Dirac operators.

\subsubsection{}\label{lottref}

The cofibre of the map of spectra $K\to H\R$ induced by the Chern character
represents a generalized cohomology theory $K\R/\Z$, called $\R/\Z$-$K$-theory. It is a module theory over  $K$-theory and therefore also admits a push-forward for $K$-oriented proper submersions. This push-forward is again defined by constructions in stable homotopy theory. An analytic/geometric  model of  $\R/\Z$-$K$-theory was proposed in \cite{MR913964}, \cite{MR1478702}. 
This led to the natural question whether there is an analytic description of the push-forward in $\R/\Z$-$K$-theory. This question was solved in \cite{MR1312690}.
The solution gives a topological interpretation of $\rho$-invariants.

Furthermore, in \cite{MR1312690}  a Chern character from $\R/\Z$-$K$-theory to cohomology with $\R/\Q$-coefficients
has been constructed, and an index theorem has been proved.

Let us now explain the relation of these constructions and results  with the present paper.
In the present paper we define the flat theory $\hat K_{flat}(B)$ as the kernel of the curvature
$R\colon \hat K(B)\to \Omega_{d=0}(B)$. It turns out that
$\hat K_{flat}(B)$ is isomorphic to $K\R/\Z(B)$ up to a degree-shift by one (Proposition \ref{rzind}).
One can actually  represent all classes of $K^{0}_{flat}(B)$ by pairs $(\cE,\rho)$, where $\cE$
is a geometric family with zero-dimensional fibre (see \ref{zerofibre}). If one restricts to these special cycles, then our model of $K^0_{flat}(B)$ and the model of $K\R/\Z^{-1}(B)$ of   \cite{MR1312690} coincide.

By an inspection of the constructions one can further check that the restriction of our cycle level push-forward (\ref{eq300})
to these particular flat cycles is the same as the one in \cite{MR1312690}. At a first glance our push-forward of flat classes seems to depend on a smooth refinement of the topological $K$-orientation of the map $p$, but
it is in fact independent of these geometric choices as can be seen using the  homotopy invariance of the flat theory.
The comparison with \cite{MR1312690} shows that the restriction of our push-forward to  flat classes coincides with the homotopy theorists' one.

The restriction of our smooth lift of the Chern character $\hat \ch\colon \hat
K(B)\to \hat H(B,\Q)$ (see Theorem \ref{mmmain}) to the flat 
theories exactly gives the Chern character of \cite{MR1312690}
$$\hat \ch\colon \hat K_{flat}(B)\to \hat H_{flat}(B,\Q)\ $$
 (using our notation and the isomorphism of $\hat H^*_{flat}(B)\cong
 H^{*-1}(B,\R/\Q)$).
If we restrict our index theorem \ref{main} to flat classes, then it
specializes to 
$$\hat\ch(\hat p_!(x))=\int_{W/B} \hA(T^vp)\cup
e^{\frac{1}{2}c_1(L^2)}\cup\hat \ch(x) ,\qquad \forall x\in \hat K(W),$$
and this is exactly the index theorem of \cite{MR1312690}.

In this sense the present paper is a direct generalization  of \cite{MR1312690} from the flat to the general case.

\subsubsection{}

The analytic model of $\R/\Z$-$K$-theory and the analytic construction of the push-forward 
in \cite{MR1312690} fits into a series of constructions of homotopy invariant functors with a push-forward  which encodes secondary spectral invariants. Let us mention the two examples
in \cite{MR1724894} which are based on flat bundles or flat bundles with duality, respectively.
The spectral geometric invariants in these examples are the analytic torsion forms of \cite{MR1303026}
and the $\eta$-forms introduced e.g.~in   \cite{MR1042214}.
The functoriality of the push-fowards under compositions is discussed in 
\cite{MR1899699} and \cite{MR2072502}.
But these construction do not fit (at least at the moment) into the world of
smooth cohomology theory, and it is still an open problem to
interpret the push-forward in topological terms.

Let us also mention  the paper \cite{MR1246617} devoted to smooth lifts of Chern classes.
\subsubsection{}

In \cite{math.DG/0703916}, \cite{math.DG/0611281} several variants of functors derived from $K$-theory are considered. In the following we recall the  names  of these groups used in that reference and explain, if possible,  their relation with the present paper.
\begin{enumerate}
\item relative $K$-theory $K_{rel}$: the  cycles are triples $(V,\nabla^V,f)$ of  $\Z/2\Z$-graded flat vector bundles and an odd selfadjoint bundle automorphism $f$ (which need not be parallel).
\item  free multiplicative $K$-theory $K_{ch}$ (also called transgressive in \cite{math.DG/0611281}): it is essentially\footnote{The connections are not assumed to be hermitean and the corresponding differential forms have complex coefficients.} a model of $\hat K^0$ based on cycles of the form $(\cE,\rho)$, where $\cE$ is a geometric family with zero-dimensional fibre coming from a geometric vector bundle (see \ref{zerofibre}). 
\item multiplicative $K$-theory $MK$: it is  the  same model of $K^0_{flat}$ as in 
\cite{MR1312690}, see  \ref{lottref}.
\item flat $K$-theory $K_{flat}$: it is the  Grothendieck group of flat vector bundles.
\end{enumerate}
Besides the definition of these groups and the investigation of their interrelation the main topic of \cite{math.DG/0703916}, \cite{math.DG/0611281} is the construction of push-forward operations.
In the following we will only discuss multiplicative and transgressive $K$-theory since they are related to the present paper. The difference to the constructions of  \cite{MR1312690} and the present paper is that Berthomiau's analytic push-forward (which we denote here by $p_!^B$) does not use the $Spin^c$-Dirac operator but the fibre-wise de Rham complex. From the point of view of analysis  the difference is essentially that the class $\hA(T^vp)\cup e^{\frac{1}{2}c_1(L^2)}$ or the corresponding differential form has to be replaced by
the Euler class $E(T^vp)$ or the Euler form of the vertical bundle. 

The advantage of working with the de Rham complex is that in  order to define
the push-forward $p_!^{B}$  one does not need a $Spin^c$-structure. If there
is one, then one can actually express $p^B_!$ in terms of $\hat p_!$
as 
$$p_!^B(x)=\hat p_!(x\cup s^*)\ ,$$ where $s^*\in K(W)$ is the class of the dual of the spinor bundle
$S^c(T^vp)$, or the $\hat K(W)$-class  represented by the geometric version of
this bundle in the case of transgressive $K$-theory, respectively. The point
here is that
the Dirac operator induced by the de Rham complex is the $Spin^c$-Dirac
operator twisted by $S^c(T^vp)^*$. 

As said above, the  homotopy theorists' $p_!$ is the push-forward associated
to a $K$-orientation of $p$.  
In contrast, the homotopy theorists' version of $p_!^B$ is the Gottlieb-Becker
transfer.

The motivation of  \cite{math.DG/0703916} , \cite{math.DG/0611281} to define the push-forward with the de Rham complex is that  it is  compatible with the push-forward  for flat $K$-theory.
The push-forward of a flat vector bundle is expressed in terms of fibre-wise cohomology which forms again a flat vector bundle on the base. This additional structure also plays a crucial role in  \cite{MR1724894}, \cite{MR1303026}, \cite{MR1899699},  and \cite{MR2072502}.
If one interprets the push-forward using the $Spin^c$-calculus, then the flat connection is lost. 
Let us mention that the first circulated version of the present paper predates
the papers \cite{math.DG/0703916} , \cite{math.DG/0611281} which actually
adapt some of our ideas.

\subsubsection{}
 
The topics of  \cite{MR2129894} are two index theorems involving $\hat
H(B,\Q)$-valued  characteristic classes. Here we only review the first one,
since the second is related to flat vector bundles. (Compare also
\cite{ma-2004} for a ``flat version'').
Let us formulate the result  of  \cite{MR2129894} in the language of the
present paper.

Let $p\colon W\to B$ be a proper submersion with closed fibres with a fibre-wise $spin$-structure over a compact base $B$.
The spin structure induces a $Spin^c$-structure, and we choose
 a representative of a smooth $K$-orientation $o:=(g^{T^vp},T^hp,\tilde \nabla,0)$, where $\tilde \nabla$ is indced from the Levi-Civita connection on $T^vp$ (see \ref{pap201} for details).
Let $\bV=(V,h^V,\nabla^V)$ be a geometric vector bundle over $W$ with
associated geometric family 
$\cV$ (compare \ref{zerofibre}). Then we can form the geometric family $\cE:=p_!\bV$ (see \ref{ddd7771}) over $B$.

The family of Dirac operators $D(\cE)$ acts on sections of a bundle of Hilbert spaces $H(\cE)\to B$. The geometric structures of the $K$-orientation $o$ and $\bV$ induce a connection $\nabla^{H(\cE)}$ (it is the connection part of the Bismut superconnection  \cite[Prop. 10.15]{bgv} associated to this situation).
We  assume that the family of Dirac operators of  $D(\cE)$  has a kernel bundle $K:=\ker(D(\cE))$.
This bundle has an induced metric $h^K$. The projection of $\nabla^{H(\cE)}$ to $K$ gives a hermitean connection $\nabla^{K}$. We thus get a geometric bundle $\bK:=(K,h^K,\nabla^K)$, and an associated  geometric family $\cK$ (see \ref{ghghghr}).
The index theorem in  \cite{MR2129894} calculates the smooth Chern character
$\hat \ch(\bK)\in \hat H(B,\Q)$ of \cite{MR827262} and states:
$$\hat \ch(\bK)=\hat p_!(\hat\hA(\mathbf{T^vp})\cup \hat \ch(\bV))+a(\eta^{BC}(\cE))\ ,$$
where we refer to (\ref{hahaha}) and \ref{etabcnot} for notation.

Note that this theorem could also be derived from our index Theorem \ref{main}.
By Corollary \ref{bcverg}, (\ref{eq300}) , our special choice of $o$, and Theorem \ref{main} (the marked step) we have 
\begin{eqnarray*}
\hat \ch(\bK)-a(\eta^{BC}(\cE))&=&\hat \ch[\cK,\eta^{BC}(\cE)]\\
&=&\hat \ch [\cE,0]\\
&=&\hat \ch([p_!\cV,0])\\
&=&\hat \ch (p_!([\cV,0]))\\
& \stackrel{!}{=}&\hat p_!^K( \hat \ch(\cV))\\
&=&p_!(\hat\hA(\mathbf{T^vp})\cup \hat \ch(\bV))\ .
\end{eqnarray*}

{\em Acknowledgement:
We thank Moritz Wiethaup for explaining to us his insights and result. We
further thank Mike Hopkins and Dan Freed for their interest in this work and
many helpful remarks. We thank the referee for many helpful comments which
lead to considerable improvements of the exposition.}

\section{Definition of smooth K-theory via cycles and relations}\label{tztewqze}
\subsection{Cycles}
\subsubsection{}
One goal of the present paper is to construct a multiplicative smooth extension of the pair $(K,\ch_\R)$ of the multiplicative generalized cohomology theory  $K$, complex $K$-theory, and the composition  $\ch_\R\colon K\stackrel{\ch}{\to} H\Q\to H\R$ of the Chern character with the natural map from ordinary cohomology with rational to real coefficients induced by the inclusion $\Q\to \R$.  In this section we define the smooth $K$-theory group $\hat K(B)$ 
of a smooth compact manifold, possibly with boundary, and construct the natural transformations $R,I,a$.
The main result of the present section is that our construction really yields
a smooth extension in the sense of Definition \ref{ddd556}. Wi discuss the multiplicative structure in Section \ref{jkjdkdqwdqdwd}.

Our restriction to compact manifolds with boundary is due to the fact that we work with absolute $K$-groups. One could in fact modify the constructions in order to produce compactly supported smooth $K$-theory or relative smooth 
$K$-theory. But in the present paper, for simplicity, we will not discuss relative smooth cohomology theories.

\subsubsection{}

We define the smooth $K$-theory $\hat K(B)$ as the group completion of a quotient of a semigroup
of isomorphism classes of cycles by an equivalence relation. We start with the
description of the cycles.

\begin{ddd}
Let $B$ be a compact manifold, possibly with boundary.
A cycle for a smooth $K$-theory class over $B$ is a pair
$(\cE,\rho)$, where $\cE$ is a geometric family, and $\rho\in
\Omega(B)/\im(d)$ is a class of differential forms.
\end{ddd}

\subsubsection{}

The notion of a geometric family has been introduced in \cite{math.DG/0201112}
in order to have a short name for the data needed to define a Bismut
super-connection \cite[Prop. 10.15]{bgv}. For the convenience of the reader we
are going to explain this notion in some detail.
\begin{ddd}
A geometric family over $B$ consists of the following data:
\begin{enumerate}
\item a proper submersion with closed fibres $\pi\colon E\to B$,
\item a vertical Riemannian metric $g^{T^v\pi}$, i.e.~a metric on the vertical
  bundle $T^v\pi\subset TE$, defined as $T^v\pi:=\ker(d\pi\colon TE\to \pi^*TB)$.
\item a horizontal distribution $T^h\pi$, i.e.~a bundle $T^h\pi\subseteq
  TE$ such that $T^h\pi\oplus T^v\pi=TE$.
\item a family of Dirac bundles $V\to E$,
\item an orientation of $T^v\pi$.
\end{enumerate}
\end{ddd}
Here, a family of Dirac bundles consists of\begin{enumerate}
\item a hermitean vector bundle with connection $(V,\nabla^V,h^V)$ on $E$,
\item a Clifford multiplication $c\colon T^v\pi\otimes V\to V$,
\item on the components where
$\dim(T^v\pi)$ has even dimension a $\Z/2\Z$-grading $z$.
\end{enumerate}
We require that the restrictions of the family Dirac bundles to the fibres $E_b:=\pi^{-1}(b)$, $b\in B$, give Dirac bundles in the usual sense (see \cite[Def. 3.1]{math.DG/0201112}): 
\begin{enumerate}
\item The vertical metric induces the Riemannian structure on $E_b$,
\item The Clifford multiplication turns $V_{|E_b}$ into a Clifford module (see \cite[Def.3.32]{bgv}) which is graded if $\dim(E_b)$ is even.
\item The restriction of the connection $\nabla^V$ to $E_b$ is  a Clifford connection   (see \cite[Def.3.39]{bgv}). 
\end{enumerate}

A geometric family is called even or odd, if $\dim(T^v\pi)$ is
even-dimensional or odd-dimensional, respectively. 

\subsubsection{}\label{zerofibre}

Here is a simple example of a geometric family with zero-dimensional fibres.
Let $V\to B$ be a complex $\Z/2\Z$-graded vector bundle. Assume that $V$ comes with a hermitean metric
$h^V$ and a hermitean connection $\nabla^V$ which are compatible with the $\Z/2\Z$-grading.
The geometric bundle $(V,h^V,\nabla^V)$ will usually be denoted by $\bV$.
 
We consider the submersion $\pi:=\id_B\colon B\to B$. In this case the vertical bundle
is the zero-dimensional bundle which has  a canonical vertical Riemannian metric $g^{T^v\pi}:=0$, and for the horizontal
bundle we must take $T^h\pi:=TB$. Furthermore, there is a canonical orientation of $p$.
The geometric bundle $\bV$   can naturally be interpreted as a family of Dirac bundles on $B\to B$. 
In this way $\bV$ gives rise to a geometric family over $B$ which we will usually denote by $\cV$.

\subsubsection{}

In order to define a representative of the negative of the smooth $K$-theory class represented by a cycle $(\cE,\rho)$ we introduce the notion of the opposite geometric family.

\begin{ddd}\label{oppdef} The opposite $\cE^{op}$ of a geometric family $\cE$ is obtained by reversing
the signs of the Clifford multiplication and the grading (in the even case) of the underlying family of Clifford bundles, and of the orientation of the vertical bundle. 
\end{ddd}

\subsubsection{}

Our smooth $K$-theory groups will be $\Z/2\Z$-graded. On the level of cycles the grading is reflected
by the notions of even and odd cycles.

\begin{ddd}
A cycle $(\cE,\rho)$ is called even (or odd, resp.), if $\cE$ is even (or odd, resp.) and $\rho\in \Omega^{odd}(B)/\im(d)$
( or $\rho\in \Omega^{ev}(B)/\im(d)$, resp.).
\end{ddd}

\subsubsection{}

Let $\cE$ and $\cE^\prime$ be two geometric families over $B$. An isomorphism $\cE\stackrel{\sim}{\to} \cE^\prime$ consists of the following data:
$$\xymatrix{V\ar[d]\ar[rr]^F&&V^\prime\ar[d]\\E\ar[dr]^\pi\ar[rr]^f&&E^\prime\ar[dl]_{\pi^\prime}\\&B&}\,$$
where
\begin{enumerate}
\item $f$ is a diffeomorphism over $B$,
\item $F$ is a bundle isomorphism over $f$,
\item $f$ preserves the horizontal distribution, the vertical metric and the orientation.
\item $F$ preserves the connection, Clifford multiplication and the grading.
 \end{enumerate}
\begin{ddd}
Two cycles $(\cE,\rho)$ and $(\cE^\prime,\rho^\prime)$ are called isomorphic if
$\cE$ and $\cE^\prime$ are isomorphic and $\rho=\rho^\prime$.
We let $G^*(B)$ denote the set of isomorphism classes of cycles over $B$ of parity $*\in \{ev,odd\}$.
\end{ddd}

\subsubsection{}
Given two geometric  families $\cE$ and $\cE^\prime$ we can form their sum $\cE\sqcup_B \cE^\prime$
over $B$. The underlying proper submersion with closed fibres  of the sum is $\pi\sqcup \pi^\prime\colon E\sqcup E^\prime\to B$.
The remaining structures of $\cE\sqcup_B\cE^\prime$ are induced in the obvious way.

\begin{ddd}\label{isodeff}
The sum of two cycles $(\cE,\rho)$ and $(\cE^\prime,\rho^\prime)$  is defined by
$$(\cE,\rho)+(\cE^\prime,\rho^\prime):=(\cE\sqcup_B \cE^\prime,\rho+\rho^\prime)\ .$$
\end{ddd}
The sum of cycles induces on $G^*(B)$ the structure of a graded abelian semigroup.
The identity element of $G^*(B)$ is the cycle $0:=(\emptyset, 0)$, where
$\emptyset$ is the empty geometric family.

\subsection{Relations}

\subsubsection{}\label{pap1}

In this subsection we introduce an equivalence relation $\sim$ on $G^*(B)$.
We show that it is compatible with the semigroup structure so that we get a semigroup 
$G^*(B)/\sim$. We then define the smooth $K$-theory 
$\hat K^*(B)$ as the group completion of  this quotient.

In order to define $\sim$ we first introduce a simpler relation "paired" which has a nice local  index-theoretic  meaning. The relation $\sim$ will be the equivalence relation generated by "paired".

\subsubsection{}\label{twdzqdqwdqwdqw}

 The main ingredients of our definition of "paired" are the notions
of a taming of a geometric family $\cE$ introduced in \cite[Def. 4.4]{math.DG/0201112}, and the $\eta$-form of a tamed family \cite[Def. 4.16]{math.DG/0201112}.

In this paragraph we shortly review the notion of a taming. For the definition of eta-forms we refer to 
\cite[Sec. 4.4]{math.DG/0201112}.  In the present paper we will use $\eta$-forms as a black box with a few important  properties which we explicitly state at the appropriate places below.
 
If $\cE$ is a geometric family over $B$, then we can form a family of Hilbert spaces $(H_b)_{b\in B}$, where $H_b:=L^2(E_b,V_{|E_b})$. If $\cE$ is even, then this family is in addition $\Z/2\Z$-graded. 
The geometric family $\cE$ gives rise to a family of Dirac operators
$(D(\cE_b))_{b\in B}$, where $D(\cE_b)$ is an unbounded selfadjoint operator on $H_b$, which is odd in the even case.

A pre-taming of $\cE$ is a family $(Q_b)_{b\in B}$ of selfadjoint operators $Q_b\in B(H_b)$
given by a smooth fibrewise integral kernel $Q\in C^\infty(E\times_B E,V\boxtimes V^*)$.
In the even case we assume in addition that $Q_b$ is odd, i.e. that it anticommutes with the grading $z$. 
The pre-taming is called a taming if $D(\cE_b)+Q_b$ is invertible for all $b\in B$.

The family of Dirac operators $(D(\cE_b))_{b\in B}$
has a $K$-theoretic  index which we denote by
$$\ind(\cE)\in K(B)\ .$$
If the geometric family $\cE$ admits a taming, then the associated family of Dirac operators operators admits an invertible compact perturbation, and hence $\ind(\cE)=0$. Vice versa, if $\ind(\cE)=0$ and the even part is empty or  has a component with  $\dim(T^v\pi)>0$, then 
by  \cite[Lemma. 4.6]{math.DG/0201112}
the geometric family admits a taming.

If the even part of $\cE$ has zero-dimensional fibres, then the existence of a taming
may require some stabilization. This means that we must add a geometric family $\cV\sqcup_B \cV^{op}$ (see \ref{zerofibre} and Definition \ref{oppdef}), where $\bV$ is the bundle $B\times \C^n\to B$ for sufficiently large $n$.
%
%

\subsubsection{}

\begin{ddd}\label{tgeomf}
A geometric family $\cE$ together with a taming will be denoted by $\cE_t$ and called a tamed geometric family.
\end{ddd}
Let  $\cE_t$ be a taming of the geometric family $\cE$ by the family $(Q_b)_{b\in B}$.
\begin{ddd} The opposite tamed family  $\cE_t^{op}$ is given by the taming
$(-Q_b)_{b\in B}$ of $\cE^{op}$.
\end{ddd}

\subsubsection{}\label{uzu1}

The local index form $\Omega(\cE)\in \Omega(B)$ is a differential form
canonically associated to a geometric family. For a detailed definition we refer to \cite[Def..4.8]{math.DG/0201112}, but
we can  briefly formulate its construction as follows.
The vertical metric $T^v\pi$ and the horizontal distribution $T^h\pi$ together induce a connection $\nabla^{T^v\pi}$ on $T^v\pi$ (see \ref{ldefr} for more details).
Locally on $E$ we can assume that $T^v\pi$ has a spin structure. We let
$S(T^v\pi)$ be the associated spinor bundle. Then we can write the family of Dirac bundles $V$ as 
$V=S\otimes W$ for a twisting bundle $(W,h^W,\nabla^W,z^W)$ with
metric, metric connection, and $\Z/2\Z$-grading which is determined uniquely up to isomorphism. 
The form
$\hat A(\nabla^{T^v\pi}) \wedge \ch(\nabla^{W})\in \Omega(E)$ is globally defined, and we get the local index form by applying the integration over the fibre $\int_{E/B}\colon \Omega(E)\to \Omega(B)$:
$$\Omega(\cE):=\int_{E/B}\hat A(\nabla^{T^v\pi}) \wedge \ch(\nabla^{W})\ .$$
The local index form is closed and represents a cohomology class
$[\Omega(\cE)]\in H_{dR}(B)$.
We let $\ch_{dR}\colon K(B)\to H_{dR}(B)$ be the composition
$$ \ch_{dR}\colon K(B)\stackrel{\ch}{\to} H(B;\Q)\stackrel{can}{\to} H_{dR}(B)\ .$$

The characteristic class version of the index theorem for families is 
\begin{theorem} [\cite{AS71}]\label{thm2}
$$\ch_{dR}(\ind(\cE))=[\Omega(\cE)]\ .$$
\end{theorem}
A proof using methods of local index theory has been given by \cite{B85}.
For a presentation of the proof we refer to  \cite{bgv}.
An alternative proof can be obtained from  \cite[Thm.4.18]{math.DG/0201112} by specializing to  the case of a family of closed manifolds.

\subsubsection{}\label{pap2}
If a geometric family $\cE$ admits a taming $\cE_t$ (see Definition \ref{tgeomf}), then we have $\ind(\cE)=0$.
In particular, the local index form $\Omega(\cE)$ is exact.
The important feature of local index theory in this case is that it provides an explicit form 
whose boundary is $\Omega(\cE)$
(see equation (\ref{detad}) below).  

 
Let $\cE_t$ be a tamed geometric family  over $B$.  In  \cite[Def. 4.16]{math.DG/0201112} we have  defined the  $\eta$-form $\eta(\cE_t)\in \Omega(B)$. By \cite[Theorem 4.13]{math.DG/0201112}) it
satisfies  
 \begin{equation}\label{detad}
d\eta(\cE_t)=\Omega(\cE)\ .
\end{equation} 
The first construction of $\eta$-forms has been given in \cite{MR1042214}, \cite{MR1052337}, \cite{MR1173033}
under the assumption that $\ker(D(\cE_b))$ vanishes or has constant dimension.
The variant which we use here has also been considered in \cite{MR1312690}, \cite{MR1484046}, \cite{MR1472895}.

Since the  analytic details of the definition of the $\eta$-form
$\eta(\cE_t)$ are quite complicated we will not repeat them here but refer to
\cite[Def. 4.16]{math.DG/0201112}. 
For most of the present paper we can use the construction of the $\eta$-form as a black box
refering to \cite{math.DG/0201112} for details of the construction and the proofs of properties.
Exceptions are arguments involving adiabatic limits for which we use \cite{MR2072502} as the reference.


\subsubsection{}

Now we can introduce the relations "paired" and $\sim$.
\begin{ddd}\label{uuu1}
We call two cycles $(\cE,\rho)$ and $(\cE^\prime,\rho^\prime)$ paired if there exists a taming
$(\cE\sqcup_B \cE^{\prime op})_t$ such that
$$\rho-\rho^\prime=\eta((\cE\sqcup_B \cE^{\prime op})_t)\ .$$
We let $\sim$ denote the equivalence relation generated by the relation "paired".
\end{ddd}

\begin{lem}\label{lem1}
The relation "paired" is symmetric and reflexive.
\end{lem}
\begin{proof}
In order to show that "paired" is reflexive and symmetric 
we are going employ the relation \cite[Lemma 4.12]{math.DG/0201112}
\begin{eqnarray}\label{eq1}
\eta(\cE_t^{op})=-\eta(\cE_t)\ .
\end{eqnarray}
Let $\cE$ be a geometric family over $B$, and let $H_b$ denote the Hilbert space of sections
of the Dirac bundle along the fibre over $b\in B$.
The family $\cE\sqcup_B \cE^{op}$ has an involution $\tau$ which flips
the components, the signs of the Clifford multiplications, the grading and the orientations.
We use the same symbol  $\tau$ in order to denote the action of $\tau$ on the  Hilbert space of sections of the Dirac bundle of $\cE_b\sqcup_B \cE_b^{op}$. The latter can be identified with $H_b\oplus H_b^{op}$, and in this picture
$$\tau=\left(\begin{array}{cc}0&1\\1&0\end{array}\right)\ .$$  
Note that  $\tau$ anticommutes with 
$$D_b:=D(\cE_b\sqcup_B \cE_b^{op})=\left(\begin{array}{cc}D(\cE_b)&0\\0&-D(\cE_b)\end{array}\right).
$$ We choose an even, compactly supported  smooth function $\chi\colon \R\to [0,\infty)$ such that $\chi(0)=1$ and form
$$Q_b:=\tau\chi(D_b)\ .$$  This operator also anticommutes with $D_b$,
and $(D_b+Q_b)^2=D_b^2+\chi^2(D_b)$ is positive and therefore invertible for all $b\in B$. 
The family $(Q_b)_{b\in B}$ thus defines a taming $(\cE\sqcup_B \cE^{op})_t$.

The involution  
$$\sigma:=\left(\begin{array}{cc}0&i\\-i&0\end{array}\right)$$ on the  Hilbert space $H_b\oplus H_b^{op}$ is induced by
 an isomorphism 
$$(\cE\sqcup_B \cE^{op})_t\cong (\cE\sqcup_B \cE^{op})_t^{op}\ .$$
Because of the relation  (\ref{eq1}) we have 
$\eta\left((\cE\sqcup_B \cE^{op})_t\right)=0$.
It follows that
$(\cE,\rho)$ is paired with $ (\cE,\rho)$.

Assume now that
$(\cE,\rho)$ is paired with $(\cE^\prime,\rho^\prime)$ via the taming
$(\cE\sqcup_B \cE^{\prime op})_t$ so that
$\rho-\rho^\prime=\eta\left((\cE\sqcup_B \cE^{\prime op})_t\right)$. Then
$(\cE\sqcup_B \cE^{\prime op})_t^{op}$ is a taming
of $\cE^\prime\sqcup_B \cE^{op}$ such that
$\rho^\prime-\rho=\eta\left((\cE\sqcup_B \cE^{\prime op})_t^{op}\right)$, again 
by (\ref{eq1}).
It follows that $(\cE^\prime,\rho^\prime)$ is paired with $(\cE,\rho)$.
 \end{proof}

\begin{lem}\label{lem5}
The relations "paired" and $\sim$ are compatible with the semigroup structure on $G^*(B)$.
\end{lem}
\begin{proof}
In fact, if $(\cE_i,\rho_i)$ are paired with $(\cE_i^\prime,\rho^\prime_i)$ via tamings
$(\cE_i\sqcup_B \cE_i^{ \prime op})_t$ for $i=0,1$,
then
$(\cE_0,\rho_0)+(\cE^\prime_0,\rho^\prime_0)$ is paired with $(\cE_1,\rho_1)+(\cE^\prime_1,\rho^\prime_1)$ via the taming
$$\left(\cE_0\sqcup_B \cE_1\sqcup_B (\cE_0^\prime\sqcup_B
  \cE_1^\prime)^{op}\right)_t:=(\cE_0\sqcup_B \cE_0^{ \prime
  op})_t\sqcup_B(\cE_1\sqcup_B \cE_1^{\prime op })_t\ .$$
In this calculation we use the additivity of the $\eta$-form \cite[Lemma
4.12]{math.DG/0201112}
$$\eta(\cE_t\sqcup_B \cF_t)=\eta(\cE_t)+\eta(\cF_t)\ .$$
The compatibilty of $\sim$ with the sum follows from the compatibility of "paired". \end{proof}
We get an induced semigroup structure on $G^*(B)/\sim$.

\begin{lem}\label{lem3}
If $(\cE_0,\rho_0)\sim (\cE_2,\rho_2)$, then 
there exists a cycle $(\cE^\prime,\rho^\prime)$
such that
$(\cE_0,\rho_0)+(\cE^\prime,\rho^\prime)$ is paired with $(\cE_2,\rho_2)+(\cE^\prime,\rho^\prime)$.
\end{lem}
\begin{proof}
Let $(\cE_0,\rho_0)$ be paired with 
$(\cE_1,\rho_1)$ via a taming $(\cE_0\sqcup_B \cE_1^{op})_t$, and 
$(\cE_1,\rho_1)$ be paired with $(\cE_2,\rho_2)$ via $(\cE_1\sqcup_B \cE_2^{op})_t$.
Then
$(\cE_0,\rho_0)+(\cE_1,\rho_1)$ is paired with $(\cE_2,\rho_2)+(\cE_1,\rho_1)$
via the taming
$$\left((\cE_0\sqcup_B \cE_1)\sqcup_B
  (\cE_2\sqcup_B\cE_1)^{op}\right)_t:=\left(\cE_0\sqcup_B
  \cE_1^{op})_t\sqcup_B (\cE_1\sqcup_B \cE_2^{op}\right)_t\ .$$
If $(\cE_0,\rho_0)\sim (\cE_2,\rho_2)$, then there is a chain
$(\cE_{1,\alpha},\rho_{1,\alpha})$, $\alpha=1,\dots,r$ with
$(\cE_{1,1},\rho_{1,1})=(\cE_0,\rho_0)$,
$(\cE_{1,r},\rho_{1,r})=(\cE_2,\rho_2)$, such that 
$(\cE_{1,\alpha},\rho_{1,\alpha})$ is paired with
$(\cE_{1,\alpha+1},\rho_{1,\alpha+1})$.
The assertion of the Lemma follows from an $(r-1)$-fold application of the
argument above.
\end{proof}

\subsection{Smooth $K$-theory}

\subsubsection{}

In this subsection we define the contravariant functor $B\to \hat K(B)$ from compact smooth manifolds to $\Z/2\Z$-graded abelian groups. Recall the definition \ref{isodeff} of the semigroup of 
isomorphism classes of cycles. By Lemma \ref{lem5} we can form the semigroup $G^*(B)/\sim$.

 \begin{ddd}
We define the smooth $K$-theory $\hat K^*(B)$ of $B$ to be the group completion of
the abelian semigroup $G^*(B)/\sim$.
 \end{ddd}
If $(\cE,\rho)$ is a cycle, then let $[\cE,\rho]\in \hat K^*(B)$ denote the corresponding class in smooth $K$-theory.

We now collect some simple facts which are helpful for computations in $\hat K(B)$ on the level of cycles.
\begin{lem}\label{lem22}
We have
$[\cE,\rho]+[\cE^{op},-\rho]=0$.
\end{lem}
\begin{proof}
We show that $(\cE,\rho)+(\cE^{op},-\rho)=(\cE\sqcup_B\cE^{op},0)$ is paired with $0=(\emptyset,0)$.
In fact, this relation is given by the taming
$((\cE\sqcup_B \cE^{op})\sqcup_B \emptyset^{op})_t=(\cE\sqcup \cE^{op})_t$ introduced in the proof of Lemma 
\ref{lem1} with $\eta((\cE\sqcup_B \cE^{op})_t)=0$.\end{proof}

\begin{lem}
Every element of $\hat K^*(B)$ can be represented in the form
$[\cE,\rho]$.
\end{lem}
\begin{proof}
An element of $\hat K^*(B)$ can be represented  by a difference
$[\cE_0,\rho_0]-[\cE_1,\rho_1]$. Using Lemma \ref{lem22} we get
$ [\cE_0,\rho_0]-[\cE_1,\rho_1]= [\cE_0,\rho_0]+[\cE_1^{op},-\rho_1]=[\cE_0\sqcup_B \cE^{op}_1,\rho_0-\rho_1]$.
\end{proof} 

\begin{lem}\label{lem4}
If $[\cE_0,\rho_0]=[\cE_1,\rho_1]$, then there exists a cycle
$(\cE^\prime,\rho^\prime)$ such that
$(\cE_0,\rho_0)+(\cE^\prime,\rho^\prime)$ is paired with $(\cE_1,\rho_1)+(\cE^\prime,\rho^\prime)$.
\end{lem}
\begin{proof}
The relation
 $[\cE_0,\rho_0]=[\cE_1,\rho_1]$ implies that there exists a cycle
$(\tilde\cE,\tilde \rho)$ such that
$(\cE_0,\rho_0)+(\tilde \cE,\rho)\sim (\cE_1,\rho_1)+(\tilde \cE,\tilde \rho)$.
The assertion now follows from Lemma \ref{lem3}. \end{proof} 

\subsubsection{}

 In this paragraph we extend $B\mapsto \hat K^*(B)$ to a contravariant functor
from smooth manifolds to $\Z/2\Z$-graded groups.
Let $f\colon B_1\to B_2$ be a smooth map. Then we have to define  a map  $f^*\colon \hat{K}^*(B_2)\to \hat K(B_1)$. 
We will first define a map of abelian semigroups
  $f^*\colon G^*(B_2)\to G^*(B_1)$, and then we show that it passes to $\hat{K}$.

If $\cE$ is a geometric family over $B_2$, then we can define an induced geometric family
$f^*\cE$ over $B_1$. The underlying submersion and vector bundle of $f^*\cE$ are given by the cartesian diagram
$$\xymatrix{f^*V\ar[d]\ar[r]&V\ar[d]\\f^*E\ar[d]^{f^*\pi}\ar[r]^F&E\ar[d]^\pi\\B_1\ar[r]^f&B_2}\ .$$
The metric  $g^{T^vf^*\pi}$ and the orientation of $T^vf^*\pi$ are  defined such that
$dF\colon T^vf^*\pi\to F^*T^v\pi$ is an isometry and orientation preserving. The horizontal distribution
$T^hf^*\pi$ is given by the condition that
$dF(T^hf^*\pi)\subseteq F^*T^h\pi$.
Finally, the Dirac bundle structure of $f^*V$ is induced from the Dirac bundle structure on $V$ in the usual way.
For $b_2\in B_2$ let $H_{b_2}$ be the Hilbert space of sections of $V$ along the fibre $E_{b_2}$.
If $b_1\in B_1$ satisfies $f(b_1)=b_2$, then we can identify the Hilbert space of sections
of $f^*V$ along the fibre $f^*E_{b_1}$ canonically with $H_{b_2}$. If $(Q_{b_2})_{b_2\in B_2}$ defines a taming $\cE_t$ of $\cE$, then the family $(Q_{f(b_1)})_{b_1\in B}$ is a taming $f^*\cE_t$ of $f^*\cE$.
We have the following relation of $\eta$-forms:
\begin{equation}\label{eq3}
\eta(f^*\cE_t)=f^*\eta(\cE_t)\ .
\end{equation}
In order to see this note the following facts. 
The geometric family $\cE$ gives rise to a bundle of Hilbert spaces $H(\cE)\to B_2$ with fibres
$H(\cE)_{b_2}=H_{b_2}$, using the notation introduced above.  We have a natural isomorphism $H(f^*\cE)\cong f^*H(\cE)$.
The geometry of $\cE$ together with the taming induces a family of super-connections
$A_s(\cE_t)$ on $H$ parametrized by $s\in (0,\infty)$ (see \cite[4.4.4]{math.DG/0201112} for explicit formulas). By construction we have $f^*A_s(\cE_t)=A_s(f^*\cE_t)$.
The $\eta$-form $\eta(\cE_t)$ is defined as an integral of the trace of a family of operators on $H(\cE)$ (with differential form coefficients)
build from $\partial_s A_s(\cE_t)$ and $A_s(\cE)^2$ \cite[Definition
4.16]{math.DG/0201112}. Equation (\ref{eq3}) now follows from
$f^*\partial_s A_s(\cE_t)=\partial_s A_s(f^*\cE_t)$ and $f^*A_s(\cE)^2=A_s(f^*\cE_t)^2$.

If $(\cE,\rho)\in G(B_2)$, then we define $f^*(\cE,\rho):=(f^*\cE,
f^*\rho)\in G(B_2)$. The pull-back preserves the disjoint union and opposites of geometric families. In particular, 
$f^*$ is a semigroup
homomorphism. Assume now that $(\cE,\rho)$ is paired with $(\cE^\prime,\rho^\prime)$ via the
taming $(\cE\sqcup_{B_2} \cE^{\prime op})_t$.
Then we can pull back the taming as well and get a taming
$f^*(\cE\sqcup_{B_2} \cE^{\prime op})_t$ of $f^*\cE\sqcup_{B_1} f^*\cE^{\prime op}$.
Equation (\ref{eq3}) now implies that
$f^*(\cE,\rho)$ is paired with $f^*(\cE^\prime,\rho^\prime)$
via the taming $f^*(\cE\sqcup_{B_2} \cE^{\prime op})_t$.
 
 Hence, the pull-back $f^*$ passes to $G^*(B)/\sim$, and being a semigroup 
  homomorphism, it induces a map of group completions
  \begin{equation*}
    f^*\colon  \hat{K}^*(B_2)\to \hat{K}^*(B_1).
  \end{equation*}

Evidently, $(\id_B)^*=\hat\id_{\hat K^*(B)}$. 
Let $f^\prime\colon B_0\to B_1$ be another smooth map.
If $\cE$ is a geometric family over $B_2$, then
$(f\circ f^\prime)^*\cE$ is isomorphic to $f^{\prime*}f^*\cE$.
This observation implies that
$$f^{\prime*}f^*=(f\circ f^\prime)^*\colon \hat K^*(B_2)\to \hat K(B_0)\ .$$
This finishes the construction of the contravariant functor
$\hat K^*$ on the level of morphisms.

 \subsection{Natural transformations and exact sequences}
 
\subsubsection{}

In this subsection we introduce the transformations $R,I,a$, and we show that they turn the functor
$\hat K$ into a smooth extension of $(K,\ch_\R)$ in the sense of Definition \ref{ddd556}.

\subsubsection{}

We first define the  natural transformation
$$I\colon \hat K(B)\to K(B)$$
by
$$I[\cE,\rho]:=\ind(\cE)\ .$$
We must show that $I$ is well-defined.
Consider
$\tilde I\colon G(B)\to K(B)$ defined by 
$\tilde I(\cE,\rho):=\ind(\cE)$.
If $(\cE,\rho)$ is paired with $(\cE^\prime,\rho^\prime)$,
then the existence of a taming
$(\cE\sqcup_B \cE^{\prime op})_t$ implies that
$\ind(\cE)=\ind(\cE^\prime)$.
The relation
\begin{equation}\label{eq5}
\ind(\cE\sqcup_B \cE^\prime)=\ind(\cE)+\ind(\cE^\prime)
\end{equation}
together with Lemma \ref{lem3} now implies that
$\tilde I$ descends to $G(B)/\sim$.
The additivity (\ref{eq5}) and the definition of $\hat K(B)$ as the group completion of
$G(B)/\sim$ implies that
$\tilde I$ further descends to the homomorphism
$I\colon \hat K(B)\to K(B)$.

The relation
$\ind(f^*\cE)=f^*\ind(\cE)$  shows that $I$ is a natural transformation of functors
from smooth manifolds to $\Z/2\Z$-graded abelian groups.

\subsubsection{}

\begin{lem}\label{lem:realize_index}
  For every compact  manifold $B$, the transformation $I\colon \hat K(B)\to K(B)$ is
  surjective. 
\end{lem}
\begin{proof}
  We discuss even and odd degrees seperately. In the even case, a K-theory
  class $\xi \in K(B)$ is represented by a $\Z/2Z$-graded vector bundle $V$ on 
  $B$. Simply choose a hermitean metric and a connection on $V$. We obtain a
  resulting geometric family $\bV$ on $B$, with underlying submersion
  $\id\colon 
  B\to B$ (i.e.~$0$-dimensional fibres) as in \ref{zerofibre}, and clearly
  $I(\bV)= \ind(\bV)=[V]=\xi\in K^0(B)$.

  For odd degrees, the statement is proved in
  \cite[3.1.6.7]{math.DG/0201112}. 
\end{proof}

\subsubsection{}

We consider the functor
$B\mapsto \Omega^*(B)/\im(d)$, $*\in \{ev,odd\}$ as a 
functor from manifolds to $\Z/2\Z$-graded abelian groups.
We construct a parity-reversing natural transformation
$$a\colon \Omega^*(B)/\im(d)\to \hat K^*(B)$$ by
$$a(\rho):=[\emptyset,-\rho]\ .$$

\subsubsection{}

Let $\Omega_{d=0}^*(B)$ be the group of closed forms of parity $*$ on $B$. 
Again we consider $B\mapsto \Omega_{d=0}^*(B)$ as a functor from smooth manifolds to $\Z/2\Z$-graded abelian groups.
We define a natural transformation
$$R\colon \hat K(B)\to\Omega_{d=0}(B)$$
by 
$$R([\cE,\rho])=\Omega(\cE)-d\rho\ .$$
Again we must show that $R$ is well-defined.
We will use the relation (\ref{detad}) 
of the $\eta$-form and the local index form, 
and the obvious properties  of local index forms
$$\Omega(\cE\sqcup_B \cE^\prime)=\Omega(\cE)+\Omega(\cE^\prime)\ ,\quad \Omega(\cE^{op})=-\Omega(\cE)\ .$$

We start with
$$\tilde R\colon G(B)\to \Omega(B)\ ,\quad \tilde R(\cE,\rho):=\Omega(\cE)-d\rho\
.$$ 
Since $\Omega(\cE)$ is closed, $\tilde R(\cE,\rho)$ is closed.
If
$(\cE,\rho)$ is paired with $(\cE^\prime,\rho^\prime)$ via the taming
$(\cE\sqcup_B \cE^{\prime op})_t$,
then $\rho-\rho^\prime=\eta((\cE\sqcup_B \cE^{\prime op})_t)$.
It follows \begin{eqnarray*}
R(\cE,\rho)&=&\Omega(\cE)-d\rho\\&=&\Omega(\cE)-d\rho^\prime-d \eta((\cE\sqcup_B \cE^{\prime op})_t)\\&=&
\Omega(\cE)-d\rho^\prime-\Omega(\cE)-\Omega (\cE^{\prime op})\\&=&\Omega(\cE^\prime)-d\rho^\prime\\&=&R(\cE^\prime,\rho^\prime)\ .\end{eqnarray*}
Since $\tilde R$ is additive it descends to
$G(B)/\sim$ and finally to the map $R\colon \hat K(B)\to \Omega_{d=0}(B)$.
It follows from
$\Omega(f^*\cE)=f^*\Omega(\cE)$ that $R$ is a natural transformation.

\subsubsection{}

The natural transformations satisfy the following relations:
\begin{lem}\label{tzw}
\begin{enumerate}
\item $R\circ a=d$
\item $\ch_{dR}\circ I=[\dots]\circ R$. 
\end{enumerate}
\end{lem}
\begin{proof}
The first relation is an immediate consequence of the definition of $R$ and $a$.
The second relation is the local index theorem \ref{thm2}. \end{proof}

\subsubsection{}
Via the embedding $H_{dR}(B)\subseteq \Omega(B)/\im(d)$,
the Chern character $\ch_{dR}\colon K(B)\to H_{dR}(B)$ can be considered as a natural transformation
$$\ch_{dR}\colon K(B)\to \Omega(B)/\im(d)\ .$$ 
\begin{prop}\label{prop1}
The following sequence is exact:
$$K(B)\stackrel{\ch_{dR}}{\to} \Omega(B)/\im(d)\stackrel{a}{\to} \hat K(B)\stackrel{I}{\to} K(B)\to 0\ .$$
\end{prop}
We give the proof in the following couple of subsection.

\subsubsection{}
We start with the surjectivity of $I\colon\hat K(B)\to K(B)$.
The main point is the fact that every element $x\in K(B)$ can be realized as
the index of a family of Dirac operators by Lemma \ref{lem:realize_index}. So
let $x\in K(B)$ and $\cE$ be a geometric family with
$\ind(\cE)=x$. Then we have $I([\cE,0])=x$. 

\subsubsection{}
Next we show exactness at $\hat K(B)$. For $\rho\in \Omega(B)/\im(d)$  we have
$I\circ a(\rho)=I([\emptyset,-\rho])=\ind(\emptyset)=0$, hence
 $I\circ a=0$.
Consider a class $[\cE,\rho]\in \hat K(B)$ which satisfies $I([ \cE,\rho])=0$.
We can assume that the fibres of the underlying submersion are not
zero-dimensional. Indeed, if necessary, we can replace $\cE$ by $\cE\sqcup_B(\tilde\cE\sqcup_B\tilde \cE^{op})$
for some even family with nonzero-dimensional fibres without changing the
smooth $K$-theory class by Lemma \ref{lem22}.
Since $\ind(\cE)=0$ this family admits a taming $\cE_t$
(\ref{twdzqdqwdqwdqw}). 
Therefore,
$(\cE,\rho)$ is paired with  $(\emptyset,\rho-\eta(\cE_t))$.
It follows that $[\cE,\rho]=a(\eta(\cE_t)-\rho)$.

\subsubsection{}\label{pap3}
In order to prepare the proof of exactness at $\Omega(B)/\im(d)$ in \ref{pap3331}  
we need some facts about the classification of tamings of a geometric family $\cE$.
The main idea is to measure the difference between tamings of $\cE$ using a local index theorem for $\cE\times [0,1]$ (compare \cite[Cor. 2.2.19]{math.DG/0201112}).
Let us assume that the underlying submersion  $\pi\colon E\to B$ decomposes as
$E=E^{ev}\sqcup_B E^{odd}$ such that the restriction of $\pi$ to the even and
odd parts  
is surjective with nonzero- and even-dimensional and odd-dimensional fibres,
and which is such that the
Clifford bundle is nowhere zero-dimensional. 
If $\ind(\cE)=0$, then there exists a taming $\cE_t$ (see \ref{twdzqdqwdqwdqw}).
Assume that $\cE_{t^\prime}$ is a second taming.
Both tamings together induce a boundary taming of the family with boundary
$(\cE\times [0,1])_{bt}$.  In \cite{math.DG/0201112} we have discussed in
detail geometric families with boundaries and the operation of taking a
boundary of a geometric family with boundary.
In the present case   $\cE\times [0,1]$ has two boundary faces labeled by the endpoints $\{0,1\}$ of the interval. We have $\partial_0(\cE\times [0,1])\cong \cE$ and $
\partial_1(\cE\times [0,1])\cong \cE^{op}$. 
A boundary taming  $(\cE\times [0,1])_{bt}$ is given
  by tamings of $\partial_i(\cE\times [0,1])$ for $i=0,1$ (see
  \cite[Def. 2.1.48]{math.DG/0201112}). We use $\cE_t$ at $\cE\times \{0\}$
  and $\cE_{t^\prime}^{op}$ at $\cE\times \{1\}$.

The boundary tamed family has an index
$\ind((\cE\times  [0,1])_{bt})\in K(B)$ which is the obstruction against extending the boundary taming to a taming
\cite[Lemma 2.2.6]{math.DG/0201112}. The construction of the local index form extends to geometric families with boundaries. Because of the geometric product structure of $\cE\times [0,1]$ we have 
$\Omega(\cE\times [0,1])=0$. The index theorem for boundary tamed families 
\cite[Theorem 2.2.18]{math.DG/0201112} gives
$$\ch_{dR}\circ \ind((\cE\times [0,1])_{bt})=[\eta(\cE_t)-\eta(\cE_{t^\prime})]\ .$$
On the other hand, given $x\in K(B)$ and $\cE_t$, since we have chosen our
family $\cE$ sufficiently big, 
there exists a taming $\cE_{t^\prime}$ such that $\ind((\cE\times [0,1])_{bt})=x$.

 To prove this, we argue as follows. Given tamings $\cE_t$ and
 $\cE_{t^\prime}$ we obtain a family $D(\cE_t,\cE_{t^\prime})$ of perturbed
 Dirac operators over
$B\times \R$ which restricts to $D(\cE_t)$ on $B\times \{\beta\}$ for
$\beta<0$, and to $D(\cE_{t^\prime})$ for $\beta\ge 1$, and which interpolates
these families for $\beta\in [0,1]$.
Since the restriction of $D(\cE_t,\cE_{t^\prime})$ is invertible outside of a compact subset of $B\times \R$ (note that $B$ is compact) it gives rise to a class
$[\cE_t,\cE_{t^\prime}]\in KK(\C,C(B)\otimes C_0(\R))$. The Dirac operator on $\R$
provides a class $[\partial]\in KK(C_0(\R),\C)$, and one checks ---using the
method of connections as in \cite[proof of Proposition 2.11]{Bun95} or directly
working with the unbounded picture \cite{BJ}--- that
$D(\cE\times [0,1])_{bt}$ represents the Kasparov product
$$[\cE_t,\cE_{t^\prime}]\otimes_{C_0(\R)}[ \partial]\in KK(\C,C(B))\ .$$
The map
$$K_c(B\times \R)\stackrel{\sim }{\to}KK(\C,C(B)\otimes
C_0(\R))\stackrel{\cdot\otimes_{C_0(\R)}
  [\partial]}{\to}KK(\C,C(B))\stackrel{\sim}{\to} K(B)$$ is by \cite[Paragraph
5, Theorem 7]{0464.46054} the inverse of the
suspension
isomorphism, so in particular surjective.
It remains to see that one can exhaust
$KK(\C,C(B)\otimes C_0(\R))$ with classes of the form
$[\cE_t,\cE_{t^\prime}]$ by varying  the taming $\cE_{t^\prime}$.

We sketch an argument in the even-dimensional case. The odd-dimensional case is similar.
For a separable infinite-dimensional Hilbert space $H$ let $GL_1(H)\subset GL(H)$
be the group of invertible operators of the form $1+K$ with $K\in K(H)$ compact.
The space $GL_1(H)$ has the homotopy type of the classifying space for $K^1$.
The bundle of Hilbert spaces $H(\cE)^+\to B$ gives rise to a (canonically trivial, up to homotopy) bundle of groups
$GL_1(H(\cE)^+)\to B$ by taking $GL_1(\dots)$ fibrewise (it is here where we use that the family is sufficiently big so that $H(\cE)^+$ is infinite-dimensional). Let $\Gamma(GL_1(H(\cE)^+))$ be the topological group of sections.
Then we have an isomorphism
$\pi_0\Gamma(GL_1(H(\cE)^+))\cong K^1(B)$.
Let $x\in  K^1(B)$ be represented by a section $s\in \Gamma(GL_1(H(\cE)^+))$.
We can approximate $s-1$ by a smooth family of smoothing operators. Therefore we can assume that $s-1$ is given by a smooth fibrewise integral kernel (a pretaming in the language of  
\cite{math.DG/0201112})\footnote{Alternatively one can directly produce such a section using the setup described in \cite{MR}.}.

There is a bijection between 
tamings $\cE_{t^\prime}$ and sections $s\in \Gamma(GL_1(H(\cE)^+))$ of this
type which maps  
$\cE_{t^\prime}$ to  $s:=D^+(\cE_t)^{-1} D^+(\cE_{t^\prime})$.
The map which associates the   $KK$-class $[\cE_t,\cE_{t^\prime}]$ to the section
$s$  is just one realization of the suspension isomorphism
$K^1(B)\to K^0_c(B\times \R)$
(using the Kasparov picture of the latter group).
In particular we see that all classes in $K^0_c(B\times \R)$ arise as $[\cE_t,\cE_{t^\prime}]$ for various tamings $\cE_{t^\prime}$.


\subsubsection{}\label{pap3331}

We now show exactness at $\Omega(B)/\im(d)$.
Let $x\in K(B)$. Then we have $a\circ \ch_{dR}(x)=[\emptyset,-\ch_{dR}(x)]$. We choose a geometric family $\cE$ as in \ref{pap3} and set $\tilde \cE:=\cE\sqcup_B \cE^{op}$. In the proof of Lemma \ref{lem1} we have constructed a taming $\tilde \cE_t$ such that $\eta(\tilde \cE_t)=0$. Using the discussion \ref{pap3} we choose a second taming
$\tilde \cE_{t^\prime}$ such that  $\ind((\tilde \cE\times [0,1])_{bt})=-x$,
hence
$\eta(\tilde \cE_{t^\prime})=\ch_{dR}(x)$.
By the taming $\tilde\cE_{t^\prime}$ we see that the cycle
$(\tilde \cE,0)$ pairs with $(\emptyset,-\ch_{dR}(x))$.
On the other hand, via $\tilde\cE_t$ the cycle
$(\tilde \cE,0)$ pairs with $0$.
It follows that $(\emptyset,-\ch_{dR}(x))\sim 0$ and hence $ a\circ\ch_{dR}=0$.

Let now $\rho\in \Omega(B)/\im(d)$ be such that $a(\rho)=[\emptyset,-\rho]=0$.
Then by Lemma \ref{lem4} there exists a cycle $(\hat \cE,\hat \rho)$ such that
$(\hat \cE,\hat \rho-\rho)$ pairs with $(\hat \cE,\hat \rho)$. Therefore there exists a taming
$\cE_{t'}$ of $\cE:=\hat \cE\sqcup_B \hat \cE^{op}$ such that
$\eta(\cE_{t'})=-\rho$.

Let $\cE_t$ be the taming with vanishing $\eta$-form constructed in the proof of Lemma 
\ref{lem1}. The two tamings induce a boundary taming $(\cE\times [0,1])_{bt}$ such that
$\ch_{dR}\circ \ind((\cE\times [0,1])_{bt})=-\eta(\cE_{t^\prime})=\rho$.
This shows that $\rho$ is in the image of $\ch_{dR}$. \hB

\subsubsection{}

We now improve Lemma \ref{lem3}. This result will be very helpful in verifying well-definedness of maps out of smooth $K$-theory, e.g.~the smooth Chern character.
\begin{lem}\label{lem333}
If $[\cE_0,\rho_0]=[\cE_1,\rho_1]$ and
at least one of these families has a higher-dimensional component, then
$(\cE_0,\rho_0)$ is paired with $(\cE_1,\rho_1)$.
\end{lem}
\begin{proof}
By Lemma \ref{lem3} there exists $[\cE^\prime,\rho^\prime]$ such that
$(\cE_0,\rho_0)+(\cE^\prime,\rho^\prime)$ is paired with $(\cE_1,\rho_1)+(\cE^\prime,\rho^\prime)$ by a taming
$\left(\cE_0\sqcup_B\cE^\prime\sqcup_B (\cE_1\sqcup_B \cE^\prime)^{op}\right)_t$.
We have
$$\rho_1-\rho_0=\eta\left((\cE_0\sqcup_B\cE^\prime\sqcup_B (\cE_1\sqcup_B \cE^\prime)^{op})_t\right)\ .$$
Since $\ind(\cE_0)=\ind(\cE_1)$ there exists a taming
$(\cE_0\sqcup_B \cE_1^{op})_t$. Furthermore, there exists a taming
$(\cE^\prime\sqcup_B\cE^{\prime op})_t$
with vanishing $\eta$-invariant (see the proof of Lemma \ref{lem1}). These two tamings combine
to  a taming 
$\left(\cE_0\sqcup_B\cE^\prime\sqcup_B (\cE_1\sqcup_B
  \cE^\prime)^{op}\right)_{t'}$. 
There exists $\xi\in K(B)$ such that
\begin{eqnarray*}\ch_{dR}(\xi)&=& 
\eta\left((\cE_0\sqcup_B\cE^\prime\sqcup_B (\cE_1\sqcup_B
  \cE^\prime)^{op})_t\right)-\eta\left((\cE_0\sqcup_B\cE^\prime\sqcup_B
  (\cE_1\sqcup_B
  \cE^\prime)^{op})_{t'}\right)\\&=&\eta\left((\cE_0\sqcup_B\cE^\prime\sqcup_B
(\cE_1\sqcup_B \cE^\prime)^{op})_t\right)-\eta\left((\cE_0\sqcup_B
\cE_1^{op})_t\right) \ . 
\end{eqnarray*}
We can now adjust (using \ref{pap3}) the taming $(\cE_0\sqcup_B \cE_1^{op})_t$ such that we can choose $\xi=0$.
It follows that $\rho_1-\rho_0=\eta\left((\cE_0\sqcup_B
  \cE_1^{op})_t\right)$.
\end{proof}

\subsection{Comparison with the Hopkins-Singer theory and the flat theory}

\subsubsection{}

An important consequence of the axioms \ref{ddd556} for a smooth generalized cohomology theory is the  homotopy formula.
Let $\hat h$ be a smooth extension of a pair $(h,c)$.
Let $x\in \hat h([0,1]\times B)$, and let $i_k\colon B\to \{k\}\times B\subset
[0,1]\times B$, $k=0,1$, be the inclusions.
\begin{lem}\label{lem23}
$$i_1^*(x)-i_0^*(x)=a\left(\int_{[0,1]\times B/B} R(x)\right)\ .$$
\end{lem}
\begin{proof}
Let $p\colon  [0,1]\times B\to B$ denote the projection.
If $x=p^*y$, then on the one hand the left-hand side of the equation is zero.
On the other hand, $R(x)=p^*R(y)$ so that
$\int_{[0,1]\times B/B} R(x)=0$, too.

Since $p$ is a homotopy equivalence there exists $\bar y \in h(B)$ such that $I(x)=p^*(\bar y)$.
Because of the surjectivity of $I$ we can choose $y\in \hat h(B )$ such that
$I(y)=\bar y$. It follows that $I(x-p^*y)=0$. By the exactness of (\ref{exax})   there exists a form
$\omega\in \Omega(I\times B)/\im(d)$ such that $x-p^*y=a(\omega)$.
By Stokes' theorem we have the equality
$i_1^*\omega-i_0^*\omega=\int_{[0,1]\times B/B} d\omega$
in $\Omega(B)/\im(d)$.
By (\ref{drgl}) we have
$d\omega=R(a(\omega))$. It follows that
$$\int_{[0,1]\times B/B} d\omega=\int_{[0,1]\times B/B} R(a(\omega))=\int_{[0,1]\times B/B} R(x-p^*y)
=\int_{[0,1]\times B/B} R(x)\ .$$
This implies
$$i_1^*x-i_0^*x=i_1^*a(\omega)-i_0^*a(\omega)=a\left(i_1^*\omega-i_0^*\omega)=a(\int_{[0,1]\times
    B/B} R(x)\right)\ .$$
\end{proof}

\subsubsection{}

Let $\hat h$ be a smooth extension of a pair $(h,c)$. We use the notation introduced in \ref{axi1}.
\begin{ddd}
The associated flat functor is defined by 
$$B\mapsto \hat h_{flat}(B):=\ker \{ R\colon \hat h(B)\to \Omega_{d=0}(B,N)\}\ .$$
\end{ddd}
Recall that a functor $F$ from smooth manifolds is homotopy invariant, if for
the two embeddings $i_k\colon B\to \{k\}\times B\to [0,1]\times B$, $k=0,1$,
we have $F(i_0)=F(i_1)$.
As a consequence of the homotopy formula Lemma \ref{lem23} the functor $\hat h_{flat}$ is homotopy invariant.

In interesting cases it is part of a generalized cohomology theory.
The map $c\colon h\to HN$ gives rise to a cofibre sequence in the stable homotopy category
$$h\stackrel{c}{\to}HN\to h_{N,\R/\Z}$$ which defines a spectrum $
h_{N,\R/\Z}$. 
\begin{prop}\label{hsrz}
If $\hat h$ is the Hopkins-Singer extension of $(h,c)$, then we have a natural isomorphism 
$$\hat h_{flat}(B)\cong h_{N,\R/\Z}(B)[-1]\ .$$
\end{prop}
In the special case that
$N=h^*\otimes_\Z \R$ this is \cite[(4.57)]{MR2192936}.

\subsubsection{}

In the case of $K$-theory and the Chern character $\ch_\R\colon K\to H(K^*\otimes_\Z\R)$
one usually writes $$K\R/\Z:=h_{K^*\otimes_\Z\R,\R/\Z}\ .$$  
The functor $B\mapsto K\R/\Z(B)$ is called $\R/\Z$-$K$-theory.  Since $\R/\Z$ is an injective abelian group
we have a universal coefficient formula
\begin{equation}\label{univc}
K\R/\Z^*(B)\cong \Hom(K_*(B),\R/\Z)\ ,
\end{equation}
where $K_*(B)$ denotes the $K$-homology of $B$. A geometric interpretation of $\R/\Z$-$K$-theory
 was first proposed in  \cite{MR913964}, \cite{MR1478702}. In these reference
 it was called multiplicative $K$-theory. The analytic construction
of the push-forward has been given in  \cite{MR1312690}.

\subsubsection{}

\begin{prop}\label{rzind}
There is a natural  isomorphism of functors
$\hat K_{flat}(B)\cong  K\R/\Z(B)[-1]$.
\end{prop}
\begin{proof}
{In the following (the paragraphs \ref{twezrwe1}, \ref{twezrwe}) we sketch two conceptually very different arguments. For 
details we refer to \cite[Section 5, Section 7]{bs2009}.
}
\subsubsection{}\label{twezrwe1}
In the first step one extends $\hat K_{flat}$ to a reduced cohomology theory on smooth manifolds.
The reduced group of a pointed manifold is defined  as the kernel of the restriction
to the point. The missing structure is  a suspension
isomorphism. It is induced by the map
$\hat K(B)\to \hat K(S^1\times B)$
given by $x\mapsto \pr_1^*x_{S^1}\cup \pr_2^*x$, where $x_{S^1}\in  \hat K^1(S^1)$ is  defined in
Definition \ref{dddxs}, and the $\cup$-product is defined below in \ref{proddefin}. The inverse is induced by the push-forward
$(\hat\pr_2)_!\colon \hat K(S^1\times B)\to \hat K(B)$
along $\pr_2\colon S^1\times B\to B$ introduced below in \ref{ddd1}.
Finally one verifies the exactness of mapping cone sequences.

In order to identify the resulting reduced cohomology theory with $\R/\Z$-$K$-theory one constructs  a pairing between $\hat K_{flat}$ and $K$-homology, using an analytic model as in  
\cite{MR1312690}. This pairing, in view of the universal coefficient formula (\ref{univc})  
gives a map of cohomology theories $\hat K_{flat}(B)\to K\R/\Z(B)[ -1]$ which
is an isomorphism by comparison of coefficients.
\subsubsection{}\label{twezrwe}

{The second argument is based on the comparison with the Hopkins-Singer theory.
We let $B\mapsto \hat K_{HS}(B)$ denote the version of the smooth $K$-theory functor
defined by Hopkins-Singer  \cite{MR2192936}. In \cite[Section 5]{bs2009} we show that there is a unique natural isomorphism
$\hat K^{ev}\stackrel{\sim}{\to} \hat K^{ev}_{HS}$.
In view of  \ref{hsrz} we get the isomorphism
$$\hat K_{flat}^{ev}(B)\stackrel{\sim}{\to} \hat K^{ev}_{HS,flat}(B)\stackrel{\sim}{\to}
K\R/\Z^{ev}[-1](B)\ .$$
In \cite{bs2009} we furthermore show that using the integration
for $\hat K$ and the suspension isomorphism for $K\R/\Z$ this isomorphism extends to
the odd parts.}

 \end{proof}
%
%

\subsubsection{}

Many of the interesting examples given in Section \ref{sec5} can be understood
(at least to a large extend) already at this stage. We recommend to look them
up now, if one is less interseted in structural questions. This should also
serve as a motivation for the constructions in Sections \ref{idwiqdqwwd} and
\ref{jkjdkdqwdqdwd}.

\section{Push-forward}\label{idwiqdqwwd}

\subsection{$K$-orientation}

\subsubsection{}\label{fuwefiu}

The groups $Spin(n)$ and $Spin^c(n)$ fit into exact sequences
$$
\begin{CD}
  1 @>>>\Z/2\Z @>>> Spin(n) @>>>SO(n) @>>> 1\\
  && @VVV @VVV @VV{\id}V\\
  1 @>>> U(1) @>{i}>> Spin^c(n) @>{\pi}>> SO(n) @>>> 1
\end{CD}
$$
$$1\to \Z/2\Z\to Spin^c(n)\stackrel{(\lambda,\pi)}{\to} U(1)\times SO(n)\to 1$$
such that $\lambda\circ i\colon U(1)\to U(1)$ is a double covering.
Let $P\to B$ be an $SO(n)$-principal bundle. We let $Spin^c(n)$ act on $P$ via the projection $\pi$.
\begin{ddd}\label{spincred}
A $Spin^c$-reduction of $P$ is a diagram
$$\xymatrix{Q\ar[dr]\ar[rr]^f&&P\ar[dl]\\&B&}\ ,$$
where $Q\to B$ is a $Spin^c(n)$-principal bundle and $f$ is $Spin^c(n)$-equivariant.
\end{ddd}

\subsubsection{}

Let $p\colon W\to B$ be a proper submersion with vertical bundle $T^vp$.
We assume that $T^vp$ is oriented. A choice of a vertical metric $g^{T^vp}$ gives an $SO$-reduction $SO(T^vp)$ of the frame bundle $\Fr(T^vp)$, the bundle of oriented orthonormal frames.

Usually one calls a map between manifolds $K$-oriented if its stable normal bundle is equipped with a 
$K$-theory Thom class. It is a well-known fact \cite{MR0167985}  that this is equivalent to the choice of a $Spin^c$-structure
on the stable normal bundle. Finally, isomorphism classes of choices of $Spin^c$-structures on $T^vp$ and the stable normal bundle
of $p$ are in bijective correspondence. So for the purpose of the present paper we adopt the following definition.
\begin{ddd}\label{topkor}
A topological $K$-orientation of $p$ is a $Spin^c$-reduction of  $SO(T^vp)$.
\end{ddd} 

In the present paper we prefer to work with $Spin^c$-structures on the vertical bundle since it directly gives
rise to a family of Dirac operators along the fibres.
The goal of this section is to introduce the notion of smooth $K$-orientation which refines a given topological $K$-orientation.

\subsubsection{}\label{ldefr}

In order to define such a family of Dirac operators we must choose additional geometric data.
If we choose a horizontal distribution $T^hp$, then we get a connection
$\nabla^{T^vp}$ which restricts to the Levi-Civita connection along the
fibres. Its  construction goes as follows. First one chooses a metric $g^{TB}$
on $B$. It induces a horizontal metric $g^{T^hp}$ via the isomorphism
$dp_{|T^hp}\colon T^hp\stackrel{\sim}{\to}
p^*TB$.  We get a metric $g^{T^vp}\oplus g^{T^hp}$ on $TW\cong T^vp\oplus T^hp$ which gives rise to a Levi-Civita connection. Its projection to $T^vp$ is $\nabla^{T^vp}$. Finally one checks that this connection is independent of the choice of $g^{TB}$.

\subsubsection{}

The connection $\nabla^{T^vp}$ can be considered as an $SO(n)$-principal bundle connection on the frame bundle $SO(T^vp)$. 
In order to define a family of Dirac operators, or better, the 
Bismut super-connection we must choose a $Spin^c$-reduction $\tilde \nabla$ of $\nabla^{T^vp}$, i.e. a connection on the $Spin^c$-principal bundle $Q$ which reduces to  $\nabla^{T^vp}$.  If we think of the connections $\nabla^{T^vp}$ and $\tilde \nabla$ in terms of horizontal distributions $T^hSO(T^vp)$ and $T^hQ$, then we say that  $\tilde \nabla$ reduces to $\nabla^{T^vp}$ if $d\pi(T^hQ)=\pi^*(T^hSO(T^vp))$.

\subsubsection{}\label{origeo}

The $Spin^c$-reduction of $\Fr(T^vp)$ determines a spinor bundle $S^c(T^vp)$, and the choice of $\tilde \nabla$ turns $S^c(T^vp)$ into a family of Dirac bundles.

In this way the choices of the $Spin^c$-structure and $(g^{T^vp},T^hp,\tilde \nabla)$ turn
$p\colon W\to B$ into a geometric family $\cW$.

\subsubsection{}\label{l2def}

Locally on $W$ we can choose a $Spin$-structure on $T^vp$ with associated spinor bundle $S(T^vp)$. Then we can write $S^c(T^vp)=S(T^vp)\otimes L$ for a hermitean line bundle $L$ with connection. The spin structure is given
by a $Spin$-reduction $q\colon R\to SO(T^vp)$ (similar to \ref{spincred}) which can actually be considered as a subbundle of $Q$. Since $q$ is a double covering and thus has discrete fibres, the connection $\nabla^{T^vp}$ (in contrast to the $Spin^c$-case) has a unique lift
to a $Spin(n)$-connection on $R$. The spinor bundle $S(T^vp)$ is associated to $R$ and has an induced connection. In view of the relations of the groups \ref{fuwefiu} the square of the locally defined line bundle
$L$ is the globally defined bundle $L^2\to W$ associated to the $Spin^c$-bundle $Q$ via the representation $\lambda\colon Spin^c(n)\to U(1)$. The connection $\tilde \nabla$ thus induces a connection on $\nabla^{L^2}$, and hence a connection on the locally defined square root $L$.
Note that vice versa, $\nabla^{L^2}$ and $\nabla^{T^vp}$ determine $\tilde \nabla$ uniquely.

\subsubsection{}

We introduce the form \begin{equation} \label{uzu111} c_1(\tilde \nabla):=\frac{1}{4\pi i}R^{L^2}\end{equation} 
which would be the Chern form of the bundle $L$ in case of a global $Spin$-structure.
Let $R^{\nabla^{T^vp}}\in \Omega^2(W,\End(T^vp))$ denote the curvature of $\nabla^{T^vp}$.
The closed form
$$\hA(\nabla^{T^vp}):={\det}^{1/2}\left(
\frac{
\frac{ R^{\nabla^{T^vp}}}{4\pi}}
{\sinh\left(\frac{R^{\nabla^{T^vp}}}{4\pi}\right)}\right)
$$
represents the $\hA$-class of $T^vp$. 
\begin{ddd}\label{uzu3} The relevant differential form for local index theory
  in the $Spin^c$-case is
$$\hA^c(\tilde \nabla):=\hA(\nabla^{T^vp})\wedge e^{c_1(\tilde \nabla)}\ .$$
\end{ddd}
If we consider $p\colon W\to B$ with the geometry $(g^{T^vp},T^hp,\tilde \nabla)$ and the Dirac bundle
$S^c(T^vp)$ as a geometric family $\cW$ over $B$, then by comparison with the description \ref{uzu1} of the local index form $\Omega(\cW)$ we see that  
$$\int_{W/B} \hA^c(\tilde \nabla)=\Omega(\cW)\ .$$

\subsubsection{}

The dependence of the form $\hA^c(\tilde \nabla)$ on the data is described in terms of the transgression form.
Let $(g_i^{T^vp},T_i^hp,\tilde \nabla_i)$, $i=0,1$, be two choices of geometric data. Then we can choose
geometric data $(\overline  g^{T^vp},\overline T^hp, \overline {\tilde \nabla})$ on $\overline p=\id_{[0,1]}\times p\colon [0,1]\times W\to [0,1]\times B$
(with the induced $Spin^c$-structure on $T^v\overline p$) which restricts to
$(g_i^{T^vp},T_i^hp,\tilde \nabla_i)$ on $\{i\}\times B$.
The class
$$\tilde \hA^c(\tilde \nabla_1,\tilde \nabla_0):=\int_{[0,1]\times W/W}\hA^c(\overline{\tilde \nabla})\in \Omega(W)/\im(d)$$ 
is independent of the extension and satisfies
\begin{equation}\label{eqq7}
d\tilde \hA^c(\tilde \nabla_1,\tilde\nabla_0)=\hA^c(\tilde \nabla_1)-\hA^c(\tilde \nabla_0)\ .
\end{equation}
\begin{ddd}\label{uzu4}
The form $\tilde \hA^c(\tilde \nabla_1,\tilde \nabla_0)$ is called the transgression form.
\end{ddd}
Note that we have the identity
\begin{equation}\label{eq7}
\tilde \hA^c(\tilde \nabla_2,\tilde \nabla_1)+\tilde \hA^c(\tilde \nabla_1,\tilde\nabla_0)=\tilde \hA^c(\tilde \nabla_2,\tilde\nabla_0)\ .
\end{equation}
As a consequence we get the identities
\begin{equation}\label{eq8}
\tilde \hA^c(\tilde \nabla,\tilde \nabla)=0\ ,\quad \tilde \hA^c(\tilde \nabla_1,\tilde \nabla_0)=-\hA^c(\tilde \nabla_0,\tilde \nabla_1)\ .
\end{equation}

\subsubsection{}\label{pap201}

We can now introduce the notion of a smooth $K$-orientation of a proper submersion $p\colon W\to B$.
We fix an 
 underlying topological $K$-orientation of $p$ (see Definition \ref{topkor}) which is given by a $Spin^c$-reduction of $SO(T^vp)$. In order to make this precise we must choose an orientation and a metric on $T^vp$.

We consider the set $\cO$ of tuples
$(g^{T^vp},T^hp,\tilde \nabla,\sigma)$ where the first three entries have the
same meaning as above (see \ref{ldefr}), and
$\sigma\in \Omega^{odd}(W)/\im(d)$. We introduce a relation $o_0\sim o_1$ on $\cO$:
Two tuples
$(g_i^{T^vp},T_i^hp,\tilde \nabla_i,\sigma_i)$, $i=0,1$ are related if and only if
$\sigma_1-\sigma_0=\tilde \hA(\tilde \nabla_1,\tilde \nabla_0)$.
We claim that
$\sim$ is an equivalence relation. In fact, symmetry and reflexivity follow from (\ref{eq8}), while transitivity
is a consequence of (\ref{eq7}).

\begin{ddd}\label{smmmmzuz}
The set of smooth $K$-orientations which refine a fixed underlying topological
$K$-orientation
of $p\colon W\to B$ is the set of equivalence classes $ \cO/\sim$.
\end{ddd}

\subsubsection{}
Note that $\Omega^{odd}(W)/\im(d)$ acts on the set of smooth $K$-orientations.
If $\alpha\in \Omega^{odd}(W)/\im(d)$ and $(g^{T^vp},T^hp,\tilde \nabla,\sigma)$ represents 
a smooth $K$-orientation, then the translate of this orientation by $\alpha$ is represented by
$(g^{T^vp},T^hp,\tilde \nabla,\sigma+\alpha)$.
As a consequence of (\ref{eq7}) we get:
\begin{kor}
The set of smooth $K$-orientations refining a fixed underlying topological $K$-orientation is a torsor over
$\Omega^{odd}(W)/\im(d)$.
\end{kor}

\subsubsection{}

If $o=(g^{T^vp},T^hp,\tilde \nabla,\sigma)\in \cO$ represents a smooth $K$-orientation, then we will write 
$$\hA^c(o):=\hA^c(\tilde \nabla)\ ,\quad \sigma(o):=\sigma\ .$$

\subsection{Definition of the Push-forward}

\subsubsection{}\label{pap8}

We consider a proper submersion $p\colon W\to B$  with a choice of a topological $K$-orientation.
Assume that $p$ has closed fibres.
Let $o=(g^{T^vp},T^hp,\tilde \nabla,\sigma)$ represent a smooth $K$-orientation which refines the given  topological one. 
To every geometric family $\cE$ over $W$ we want to associate a geometric family $p_!\cE$ over $B$.

Let
$\pi\colon E\to W$ denote the underlying proper submersion with closed fibres of $\cE$ which comes with the  geometric data $g^{T^v\pi}$, $T^h\pi$ and the family of Dirac bundles $(V,h^V,\nabla^V)$.

The underlying proper submersion with closed fibres of $p_!\cE$ is 
$$q:=p\circ\pi\colon E\rightarrow B\ .$$
The horizontal bundle of $\pi$ admits a 
decomposition $T^h\pi\cong \pi^* T^vp\oplus \pi^*T^hp$,
where the isomorphism is induced by $d\pi$. 
We define $T^hq\subseteq T^h\pi$ such that $d\pi\colon T^hq\cong \pi^* T^hp$.
Furthermore we have an identification
$T^vq=T^v\pi\oplus \pi^* T^vp$.
Using this decomposition we define the vertical metric
$g^{T^vq}:=g^{T^v\pi}\oplus \pi^* g^{T^vp}$.
The orientations of $T^v\pi$ and $T^vp$ induce an orientation of $T^vq$.
Finally we must construct the Dirac bundle $p_!\cV\rightarrow E$.
Locally on $W$ we choose a $Spin$-structure on $T^vp$ and let $S(T^vp)$ be the spinor bundle. Then we can  write
$S^c(T^vp)=S(T^vp)\otimes L$ for a hermitean line bundle with connection.
Locally on $E$ we can choose a $Spin$-structure on $T^v\pi$ with spinor bundle $S(T^v\pi)$. Then we can write $V=S(T^v\pi)\otimes Z$,
where   $Z$ is the twisting
bundle of $V$, a hermitean vector bundle with connection ($\Z/2\Z$-graded in
the even case).
The local spin structures on $T^v\pi$ and $\pi^*T^vp$ induce a local $Spin$-structure on
$T^vq=T^v\pi\oplus\pi^*T^vp$.
Therefore locally we can define the family of Dirac bundles
$p_!V:=S(T^vq)\otimes \pi^*L \otimes Z$. It is easy to see
that this bundle is well-defined independent of the choices of local
$Spin$-structures and therefore is a globally defined family of Dirac bundles.
 \begin{ddd}\label{ddd7771}
Let
$p_!\cE$ denote the geometric family given by $q\colon E\to B$ and $p_!V\to E$ with the geometric structures
defined above.
\end{ddd}
It immediately follows from the definitions, that
$p_!(\cE^{op})\cong (p_!\cE)^{op}$.

\subsubsection{}\label{pap55}

Let $p\colon W\to B$ be a proper submersion  with a smooth $K$-orientation represented by  $o$.
In \ref{pap8} we have constructed for each geometric family $\cE$ over $W$ a push-forward $p_!\cE$.
Now we introduce a parameter $\lambda\in (0,\infty)$ into this construction.
\begin{ddd}
For $\lambda\in (0,\infty)$ we define  the geometric family $p_!^\lambda\cE$ as in \ref{pap8} with the only difference that 
the metric on $T^vq=T^v\pi\oplus \pi^*T^vp$ is given by $g_\lambda^{T^vq}=
{\lambda}^2g^{T^v\pi}\oplus  \pi^*g^{T^vp}$. 
\end{ddd}
More specifically, we use scaling invariance of the spinor bundle to
canonically identify the Dirac bundle for the metric $g_\lambda$ locally with 
$p_!V:=S(T^vq)\otimes \pi^*L\otimes Z$ (for $g_1$). This uses the
description of $S(T^vp)$ in terms of tensor products of $S(T^v\pi)$ and
$\pi^*S(T^vp)$ (compare \cite[Section 2.1.2]{math.DG/0201112}) and the scaling
invariance of $S(T^v\pi)$. However, with this identification the Clifford
multiplication by vectors in $T^vq=T^v\pi\oplus \pi^*T^vp$ is rescaled on the
summand $T^v\pi$ by $\lambda$. The connection is slightly more complicated,
but converges for $\lambda\to 0$ to some kind of sum connection. 

The family of geometric families $p^\lambda_!\cE$ is called the adiabatic deformation of $p_!\cE$.
There is a natural way to define a geometric family $\cF$ on $(0,\infty)\times B$ such that its restriction
to $\{\lambda\}\times B$ is $p_!^\lambda\cE$. In fact, we define
$\cF:=(\id_{(0,\infty)}\times p)_!((0,\infty)\times \cE)$ with the exception
that we take the appropriate vertical metric. Note again that the underlying
bundle can be canonically identified with $(0,\infty)\times p_!V$. In the
following, we work with this identifications throughout.

Although the vertical metrics of $\cF$ and $p^\lambda_!\cE$ collapse as $\lambda\to 0$ the induced connections and the curvature tensors
on the vertical bundle $T^vq$ converge and simplify in this limit. This fact is heavily used in local index theory, and we refer to \cite[Sec 10.2]{bgv} for details. In particular, the integral
\begin{equation}\label{dtqdutwqdwqdq}
\tilde \Omega(\lambda,\cE):=\int_{(0,\lambda)\times B/B}\Omega(\cF)\end{equation}
converges, and we have
\begin{equation}\label{eq88}
\lim_{\lambda\to 0}\Omega(p_!^\lambda\cE)=\int_{W/B}\hA^c(o)\wedge \Omega(\cE)\ ,\quad \Omega(p_!^\lambda\cE)-\int_{W/B}\hA^c(o)\wedge \Omega(\cE)=d\tilde \Omega(\lambda,\cE)\ .
\end{equation}

\subsubsection{}

Let $p\colon W\to B$ be a proper submersion with closed fibres with a  smooth $K$-orientation represented by  $o$.
We now start with the construction of the push-forward $p_!\colon \hat K(W)\to \hat K(B)$. For $\lambda\in (0,\infty)$
and  a cycle $(\cE,\rho)$ we define 
\begin{equation}\label{eq300}
\hat p^\lambda_!(\cE,\rho):=[p^\lambda_!\cE,\int_{W/B}  \hA^c(o)\wedge \rho + \tilde
\Omega(\lambda,\cE)+\int_{W/B}\sigma(o) \wedge R([\cE,\rho])]\in \hat K(B)\ .
\end{equation}
Since $\hA^c(o)$ and $R([\cE,\rho])$ are closed, the maps
$$\Omega(W)/\im(d)\ni \rho\mapsto  \int_{W/B}  \hA^c(o)\wedge \rho\in
\Omega(B)/\im(d)\ ,$$ 
$$\Omega(W)/\im(d)\ni \sigma(o)\mapsto \int_{W/B}\sigma(o) \wedge
R([\cE,\rho])\in \Omega(B)/\im(d)$$ are well-defined. It immediately follows
from the definition that 
$\hat p_!^\lambda\colon G(W)\to \hat K(B)$ is a homomorphism of semigroups.

\subsubsection{}\label{pap101}

The homomorphism $\hat p^\lambda_!\colon G(W)\to \hat K(B)$ commutes with pull-back.
More precisely, let $f\colon B^\prime\to B$ be a smooth map.
Then we define the submersion $p^\prime\colon W^\prime\to B^\prime$ by the cartesian
diagram 
$$\xymatrix{W^\prime\ar[d]^{p^\prime}\ar[r]^F&W\ar[d]^p\\B^\prime\ar[r]^f&B}\ .$$
The differential
$dF\colon TW^\prime\to F^*TW$ induces an isomorphism
$dF\colon T^vW^\prime\stackrel{\sim}{\to} F^*T^vW$.
Therefore the metric, the orientation, and the $Spin^c$-structure of $T^v p$ induce by pull-back corresponding structures on $T^vp^\prime$.
We define the horizontal distribution $T^hp^\prime$ such that
$dF(T^hp^\prime)\subseteq F^*T^hp$. Finally we set $\sigma^\prime:=F^*\sigma$.
The representative of a smooth $K$-orientation given by these structures will be denoted by 
$o^\prime:=f^*o$. An inspection of the definitions shows:
\begin{lem}\label{djdgejwdewd}
The pull-back of representatives of smooth $K$-orientations preserves  equivalence and hence
 induces a pull-back of smooth $K$-orientations. 
\end{lem}
Recall from \ref{origeo} that the representatives $o$ and $o^\prime$ of the smooth $K$-orientations enhance $p$ and $p^\prime$ to geometric families $\cW$ and $\cW^\prime$. We have $f^*\cW\cong \cW^\prime$.

Note that we have
$F^*\hA^c(o)=\hA^c(o^\prime)$.
If $\cE$ is a geometric family over $W$, then an inspection of the definitions shows that
$f^*p_!(\cE)\cong p^\prime_!(F^*\cE)$. The following lemma now follows immediately from the definitions
\begin{lem}\label{lem100}
We have $f^*\circ \hat p^\lambda_!=\hat{p^\prime}^\lambda_!\circ F^*\colon G(W)\to \hat K(B^\prime)$.
\end{lem}

\subsubsection{}

\begin{lem}
The class $\hat p_!^\lambda(\cE,\rho)$ does not depend on $\lambda\in (0,\infty)$.
\end{lem}
 \begin{proof}
Consider $\lambda_0<\lambda_1$.
Note that
$$\hat p_!^{\lambda_1}(\cE,\rho)-\hat p_!^{\lambda_0}(\cE,\rho)=[p^{\lambda_1}_!\cE,\tilde \Omega(\lambda_1,\cE)]-[p_!^{\lambda_0}\cE,\tilde \Omega(\lambda_0,\cE)]\ .$$
 Consider the inclusion $i_\lambda\colon B\to\{\lambda\}\times B\subset [\lambda_0,\lambda_1]\times B$ and let
$\cF$ be the family over $[\lambda_0,\lambda_1]\times B$ as in \ref{pap55} such that $p_!^\lambda\cE=i^*_\lambda\cF$. 
We apply the homotopy formula Lemma  \ref{lem23}
to $x=[\cF,0]$:
$$i_{\lambda_1}^*(x)-i_{\lambda_0}^*(x)=a\left(\int_{[\lambda_0,\lambda_1]\times
    B/B}R(x)\right)=a\left(\int_{[\lambda_0,\lambda_1]\times B/B}\Omega(\cF)\right)=a\left(\tilde
  \Omega(\lambda_1,\cE)-\tilde \Omega(\lambda_0,\cE)\right)\ ,$$ 
where the last equality follows directly from the definition of $\tilde \Omega$.
This equality is equivalent to 
$$
[p^{\lambda_1}_!\cE,\tilde \Omega(\lambda_1,\cE)]=[p_!^{\lambda_0}\cE,\tilde \Omega(\lambda_0,\cE)]\ .
$$
\end{proof} 
In view of this Lemma we can omit the superscript $\lambda$ and write
$\hat p_!(\cE,\rho)$ for $\hat p_!^\lambda(\cE,\rho)$.

\subsubsection{}\label{pap66}

Let $\cE$ be a geometric family over $W$ which admits a taming $\cE_t$.
Recall that the taming is given by a family of smoothing operators $(Q_w)_{w\in W}$.

We have identified the Dirac bundle of $p_!^\lambda\cE$ with the Dirac bundle
of  $p^1_!\cE$ in a natural way in \ref{pap55}. The
$\lambda$-dependence of the 
Dirac operator takes the form 
$$D(p^\lambda_!\cE)=\lambda^{-1} D(\cE)+ (D^H+ R(\lambda))\ ,$$ where $D^H$ is the horizontal Dirac operator, and $R(\lambda)$ is of zero
order and remains bounded as $\lambda\to 0$. 
We now replace
$D(\cE)$ by the invertible operator $D(\cE)+Q$.
Then for small $\lambda>0$ the operator
$$ \lambda^{-1} (D(\cE)+Q)+ (D^H+ R(\lambda))$$ is invertible.
To see this, we consider its square which has the structure
$$\lambda^{-2} (D(\cE)+Q)^2+\lambda^{-1} \{D(\cE)+Q,(D^H+ R(\lambda))\}
+(D^H+ R(\lambda))^2\ .$$
The anticommutator $\{D(\cE),D^H+R(\lambda) \}$ is a first-order vertical operator which is thus dominated by a multiple of
the positive second order    
$(D(\cE)+Q)^2$. The remaining parts of the anticommutator are zero-order and
therefore also dominated by  multiples of $(D(\cE)+Q)^2$. The last summand is a square of a selfadjoint operator and hence non-negative.

The family of operators along the fibres of $p_!\cE$ induced by $Q$ is not a taming   since it is not given by a family of integral operators along the fibres of $p_!E\to B$. In order to understand its structure note the following.
For $b\in B$ the fibre of $(p_!\cE)_b$ is the total space of the  bundle
$E_{|W_b}\to W_b$. The integral kernel $Q$ induces a 
family of smoothing operators on the bundle of Hilbert spaces $H(\cE_{|W_b})\to W_b$. Using 
the natural identification
$$H(p_!\cE)_b\cong L^2(W,S(T^vp)\otimes H(\cE_{|W_b}))$$
we get the induced operator on $H(p_!\cE)_b$. We will call a family of  operators with this structure
a generalized taming.

Now recall that the $\eta$-form $\eta(\cF_t)$ of a tamed  or generalized tamed
family $\cF_t$ is build from a family of superconnections $A_s(\cF_t)$
parametrized by $s\in (0,\infty)$ (see \cite[2.2.4.3]{math.DG/0201112}).  For
$0<s<1$ the family coincides with the usual rescaled Bismut superconnection
and is independent of the taming. Therefore the taming does not affect the
analysis of $\partial_s A_s(\cF_t)\ee^{-A_s(\cF_t)^2}$ for $s\to 0$. In the
interval $s\in [1,2]$ the family $A_s(\cF_t)$  smoothly connects with the
family of superconnections given by
$$A_s(\cF_t)=sD(\cF_t)+ \mbox{terms with higher form degree}$$
for $s\ge 2$. In order to define the $\eta$-form $\eta(\cF_t)$  the main points are:
\begin{enumerate}
\item For small $s$ the family $A_s(\cF_t)$ behaves like the Bismut superconnection. 
The formula  (\ref{detad}) 
$$d\eta(\cF_t)=\Omega(\cF)$$
only depends on the  behavior of  $A_s(\cF_t)$ for small $s$. Therefore this formula
continues to hold  for generalized tamings. 
\item $\partial_s A_s(\cF_t)\ee^{-A_s(\cF_t)^2}$ is given by a family of integral operators
with smooth integral kernel. This holds true for tamed families as well as for familes which are tamed in the generalized sense explained above. A proof can be based on Duhamel's principle.
\item The integral kernel of $\partial_s A_s(\cF_t)\ee^{-A_s(\cF_t)^2}$ together with all derivatives vanishes exponentially as $s\to \infty$. This follows by spectral estimates from the invertibility and selfadjointness of $D(\cF_t)$. Now the invertibility of $D(\cF_t)$ is exactly the desired effect of a taming or generalized taming.
\end{enumerate}

Coming back to our iterated fibre bundle we see that we can use the generalized taming
for sufficiently small $\lambda>0$
like a taming in order to define an $\eta$-form
which we will denote by $\eta(p_!^\lambda\cE_t)$.
To be precise this eta form is associated to the family of operators
 $$A_s(p^\lambda_!\cE)+\chi(s\lambda^{-1})s\lambda^{-1}Q\ ,\quad s\in (0,\infty)\ ,$$ where $\chi$ vanishes near zero and is equal to $1$ on $[1,\infty)$. This means that we switch on the taming at time $s\sim \lambda$, and we rescale it in the same way as the vertical part of the Dirac operator.
 
We can control the behaviour of
$\eta(p^\lambda_!\cE_t)$ in the adiabatic limit $\lambda\to 0$.
\begin{theorem}\label{adia1}
$$\lim_{\lambda\to 0}\eta(p^\lambda_!\cE_t)=\int_{W/B} \hA^c(o)\wedge \eta(\cE_t)\
.$$
\end{theorem}
\begin{proof}
To write out a formal proof of this theorem seems too long for the present
paper, without giving fundamental new insights. Instead we point out the
following references. Adiabatic limits of $\eta$-forms of twisted signature
operators were 
studied in \cite[Section 5]{MR2072502}. The same methods apply in the present
case. 
The $L$-form in \cite[Section 5]{MR2072502} is the local index form of the signature operator. In the present case it must be replaced by the form $\hA^c(o)$, the  local index form of the $Spin^c$-Dirac operator.
The absence of small eigenvalues simplifies matters considerably. \end{proof} 

Since the geometric family $p_!^\lambda\cE$ admits a generalized taming  it follows that 
$\ind(p_!^\lambda\cE)=0$. Hence we can also choose a taming $(p_!^\lambda\cE)_t$.
The latter choice together with the generalized taming
induce a generalized boundary taming of the family
$p_!^\lambda\cE\times [0,1]$ over $B$.
The index theorem 
\cite[Theorem 2.2.18]{math.DG/0201112}
can be extended to generalized boundary tamed families
(by copying the proof)
and gives:
\begin{lem}\label{lem33}
The difference of $\eta$-forms 
$\eta((p_!^\lambda\cE)_t)-\eta(p^\lambda_!\cE_t)$ 
is closed. Its de Rham cohomology class satisfies
$$[\eta((p_!^\lambda\cE)_t)-\eta(p^\lambda_!\cE_t)]\in \ch_{dR}(K(B))\ .$$
\end{lem}

\subsubsection{}

We now show that
$\hat p_!\colon G(W)\to \hat K(B)$ passes through the equivalence relation $\sim$.
Since $\hat p_!$ is additive it suffices by Lemma \ref{lem3} to show the following assertion.
\begin{lem}\label{pass}
If $(\cE,\rho)$ is paired with $(\tilde \cE,\tilde \rho)$,
then $\hat p_!(\cE,\rho)=\hat p_!(\tilde \cE,\tilde\rho)$.
\end{lem}
\begin{proof}
Let $(\cE\sqcup_W \tilde \cE^{op})_t$ be the taming which induces the relation between the two cycles,
i.e. $\rho-\tilde \rho=\eta\left((\cE\sqcup_W \tilde \cE^{op})_t\right)$.
In view of the discussion in \ref{pap66} we can choose a taming
$p^\lambda_!(\cE\sqcup \tilde \cE^{op})_t$.
\begin{eqnarray*}
[p^\lambda_!\cE,0]-[p^\lambda_!\tilde \cE,0]&=&
[p^\lambda_!(\cE\sqcup_W \tilde \cE^{op}),0]\\
&=&a\left(\eta\left(p^\lambda_!(\cE\sqcup_W \tilde \cE^{op})_t \right)\right)\
.\end{eqnarray*}
By Proposition \ref{prop1} and Lemma \ref{lem33} we can replace the taming by the generalized  taming
and still get 
$$[p^\lambda_!\cE,0]-[p^\lambda_!\tilde \cE,0] =a\left(\eta\left(p^\lambda_!(\cE\sqcup_W \tilde \cE^{op})_t\right)\right)\ .$$

For sufficiently small $\lambda>0$ we thus get 
\begin{eqnarray*}
\hat p_!(\cE,\rho)-\hat p_!(\tilde \cE,\tilde\rho)&=&a\left(\eta\left(p^\lambda_!(\cE\sqcup_W \tilde \cE^{op})_t\right)\right)-
\int_{W/B}  \hA^c(o)\wedge (\rho-\tilde \rho)\\&&+ \tilde
\Omega(\lambda,\cE) - \tilde
\Omega(\lambda,\tilde \cE) )
\end{eqnarray*}
We now go to the limit $\lambda\to 0$ and use Theorem \ref{adia1} in order to get 
\begin{eqnarray*}
\hat p_!(\cE,\rho)-\hat p_!(\tilde \cE,\tilde\rho)&=&
a\left(\int_{W/B}\hA^c(o)\wedge \eta\left((\cE\sqcup_W \tilde
    \cE^{op})_t\right)\right)\\
&=&- 
\int_{W/B}  \hA^c(o)\wedge (\rho-\tilde \rho) \\
&=& 0
\end{eqnarray*} 
\end{proof}

We let
$$\hat p_!\colon \hat K(W)\to \hat K(B)$$ denote the map induced
by  the construction (\ref{eq300}). Though not indicated in the notation until now this map may depend on the choice
of the representative of the smooth $K$-orientation $o$ (later in Lemma \ref{pap99} we see that it only depends on the smooth $K$-orientation).

\subsubsection{}

Let $p\colon W\to B$ be a proper submersion with closed fibres with a smooth $K$-orientation represented by $o$.
 We now have constructed a homomorphism
$$\hat p_!\colon \hat K(W)\to \hat K(B)\ .$$
In the present paragraph we study the compatibilty of this construction with
the curvature map $R\colon \hat K\to \Omega_{d=0}$.
\begin{ddd}\label{def313}
We define the integration of forms 
$p_!^o\colon \Omega(W)\to \Omega(B)$
by
$$p_!^o(\omega)=\int_{W/B} (\hA^c(o)-d\sigma(o))\wedge \omega$$
\end{ddd}
Since $  \hA^c(o)-d\sigma(o) $ is closed  we also have a factorization
 $$p_!^o\colon \Omega(W)/\im(d)\to \Omega(B)/\im(d)\ .$$
\begin{lem}\label{lem24}
For $x\in \hat K(W)$ we have
$$R(\hat p_!(x))=p^o_!(R(x))\ .$$
\end{lem}
\begin{proof}
Let $x=(\cE,\rho)$. We insert the definitions, $R(x)=\Omega(\cE)-d\rho$, and  (\ref{eq88})  in the marked step.
\begin{eqnarray*}
R(\hat p_!(x))&=&\Omega(p_!^\lambda\cE)-d(\int_{W/B}  \hA^c(o)\wedge \rho + \tilde
\Omega(\lambda,\cE)+\int_{W/B}\sigma(o) \wedge R(x))\\&\stackrel{!}{=}&
\Omega(p_!^\lambda\cE)-\int_{W/B}  \hA^c(o)\wedge d\rho + \int_{W/B}\hA^c(o)\wedge \Omega(\cE)- \Omega(p_!^\lambda\cE)-\int_{W/B}d\sigma(o) \wedge R(x)\\&=&
\int_{W/B} (\hA^c(o)-d\sigma(o))\wedge R(x) \\
&=&p_!^o(R(x))
\end{eqnarray*}
\end{proof} 
\subsubsection{}

Our constructions of the homomorphisms
$$\hat p_!\colon \hat K(W)\to \hat K(B)\ ,\quad p_!^o\colon \Omega(W)\to
\Omega (B)$$
involve an explicit choice of a representative $o=(g^{T^vp},T^hp,\tilde \nabla,\sigma)$
of the smooth $K$-orientation lifting the given topological $K$-orientation of
$p$. 
In this paragraph we show: 

\begin{lem}\label{pap99}
The homomorphisms $\hat p_!\colon \hat K(W)\to \hat K(B)$ and 
$p_!^o\colon \Omega(W)\to \Omega (B)$
only depend on the smooth $K$-orientation represented by $o$.
\end{lem}
\begin{proof}
Let $o_k:=(g_k^{T^vp},T_k^hp,\tilde \nabla_k,\sigma_k)$, $k=0,1$
be two representatives of a smooth $K$-orientation.
Then we have
$\sigma_1-\sigma_0=\tilde \hA^c(\tilde \nabla_1,\tilde \nabla_0)$.
For the moment we indicate by a superscript $\hat p_!^k$ which representative
of the smooth $K$-orientation is used in the definition. Let $\omega\in
\Omega(W)$.
Then using (\ref{eqq7}) we get 
\begin{eqnarray*}
p_!^{o_1}(\omega)-p_!^{o_0}(\omega)&=&\int_{W/B}(\hA^c(o_1)-\hA^c(o_0)
-d(\sigma_1-\sigma_0))  \wedge
\omega\\&=&\int_{W/B}(\hA^c(\tilde\nabla_1)-\hA^c(\tilde\nabla_0) -d\tilde
\hA^c(\tilde \nabla_1,\tilde \nabla_0))\wedge \omega\\&=&0\ .\end{eqnarray*} 
We now consider the projection $\overline p\colon [0,1]\times W\to [0,1]\times B$ with the induced topological $K$-orientation.
It can be refined to a smooth $K$-orientation $\overline o$ which restricts to 
$o_k$ at $\{k\}\times B$. Let $q\colon [0,1]\times W\to W$ be the projection
and $x\in \hat K(W)$.
Furthermore let
$i_k\colon B\to \{k\}\times B\to [0,1]\times B$ be the embeddings. The
following chain of equalities follows from the homotopy formula Lemma
\ref{lem23}, the curvature formula Lemma \ref{lem24}, Stokes' theorem and the
definition of  $\tilde \hA^c(\tilde \nabla_1,\tilde \nabla_0)$, and finally
from the fact that $o_0\sim o_1$. 
\begin{eqnarray*}
\hat p_!^1(x)-\hat p_!^0(x)&=&i_1^*\hat{\overline p}_! q^*(x)-i_0^*\hat{\overline p}_! q^*(x)\\
&=&a\left(\int_{[0,1]\times B/B} R(\hat{\overline p}_!q^*x)\right)\\
&=&a\left(\int_{[0,1]\times B/B} \overline p_!^{\overline o} R(q^*(x))\right)\\
&=&a\left(\int_{[0,1]\times B/B} \overline p_!^{\overline o} q^*(R(x))\right)\\
&=&a\left(\int_{[0,1]\times B/B}\int_{[0,1]\times W/[0,1]\times B} (\hA^c(\overline o)-d\sigma(\overline o))\wedge q^*R(x)\right)\\
&=&a\left(\int_{W/B}[\int_{[0,1]\times W/W} (\hA^c(\overline
  o)-d\sigma(\overline o))]\wedge R(x)\right)\\
&=&a\left(\int_{W/B}[ \tilde \hA^c(\tilde \nabla_1,\tilde
  \nabla_0)-(\sigma(o_1)-\sigma(o_0))]\wedge R(x)\right)\\
&=&0\ .
\end{eqnarray*}
\end{proof}

\subsubsection{}

Let $p\colon W\to B$ be a proper submersion with closed fibres with a topological $K$-orientation. 
We choose a smooth $K$-orientation which refines
the topological $K$-orientation.  In this case we say that $p$ is smoothly $K$-oriented.
\begin{ddd}\label{ddd1}
We define the push-forward
$\hat p_!\colon \hat K(W)\to \hat K(B)$ to be the map induced by (\ref{eq300})
for some choice of a representative of the smooth $K$-orientation
\end{ddd}
We also have well-defined maps
$$p_!^o\colon \Omega(W)\to \Omega(B)\ ,\quad p^o_!\colon \Omega(W)/\im(d)\to \Omega(B)/\im(d)$$
given by integration of forms along the fibres.
Let us state the result about the compatibility of $\hat p_!$ with the
structure maps of smooth $K$-theory as follows.
\begin{prop}\label{mainprop}
The following diagrams commute:
\begin{equation}\label{uppersq}
  \begin{CD}
    K(W) @>{\ch_{dR}}>> \Omega(W)/\im(d) @>a>>\hat K(W) @>{I}>> K(W)\\
    @VV{p_!}V     @VV{p^o_!}V     @VV{\hat p_!}V     @VV{p_!}V \\
    K(B) @>{\ch_{dR}}>> \Omega(B)/\im(d) @>a>>\hat K(B) @>{I}>> K(B)\\
  \end{CD}
\end{equation}
\begin{equation}\label{lowersq}
  \begin{CD}
    \hat K(W) @>{R}>> \Omega_{d=0}(W)\\
    @VV{\hat p_!}V      @VV{p^o_!}V \\
    \hat K(B) @>{R}>> \Omega_{d=0}(B)
  \end{CD}
\end{equation}

\end{prop}
\begin{proof}
The maps between the topological $K$-groups are the usual push-forward maps defined by the $K$-orientation of $p$.  The other two are defined above.
 The square (\ref{lowersq}) commutes by Lemma \ref{lem24}.
The right square of (\ref{uppersq}) commutes because we have the well-known fact from index theory
$$\ind(p_!(\cE))=p_!(\ind(\cE))\ .$$
Let $\omega\in \Omega(W)/\im(d)$.
Then we have
\begin{eqnarray*}
\hat p_!(a(\omega))&=&[\emptyset,\int_{W/B}\sigma(o)\wedge d\omega-\int_{W/B}\hA^c(o)\wedge \omega]\\
&=&[\emptyset,-\int_{W/B}(\hA^c(o)-d\sigma(o))\wedge \omega]\\&=&
a\left(p_!(\omega)\right)\ .
\end{eqnarray*}
This shows that the middle square in (\ref{uppersq}) commutes.
Finally, the commutativity of the left square in (\ref{uppersq}) is a consequence of the  Chern character version of the family index theorem
$$\ch_{dR}(p_!(x))=\int_{W/B}\hA^c(T^vp)\wedge \ch_{dR}(x)\ ,\quad x\in K(W)\ .$$
\end{proof}

 If $f\colon B^\prime\to B$ is a smooth map then we consider the cartesian diagram
 $$
 \begin{CD}
   W' @>F>> W\\
   @VV{p'}V @VV{p}V\\
   B' @>f>> B
 \end{CD}
\ .
 $$ 
We equip $p^\prime$ with the induced smooth $K$-orientation (see  \ref{pap101}).
\begin{lem}\label{cartcomp}
The following diagram commutes:
$$
\begin{CD}
  \hat K(W) @>{F^*}>> \hat K(W^\prime)\\
  @VV{p_!}V  @VV{p_!^\prime}V \\
  \hat K(B) @>{f^*}>> \hat K(B^\prime)
\end{CD}
\ .$$
\end{lem}
\begin{proof}
This follows from Lemma \ref{lem100}. \end{proof}

\subsection{Functoriality}

\subsubsection{}\label{pap200}

We now discuss the functoriality of the push-forward with respect to
iterated fibre bundles.
Let $p\colon W\to B$ be as before together with a representative of a smooth $K$-orientation
$o_p=(g^{T^vp},T^hp,\tilde \nabla_p,\sigma(o_p))$.
Let $r\colon B\rightarrow A$ be another proper submersion with closed fibres
with a topological $K$-orientation which is refined by a smooth $K$-orientation represented by $o_r:=(g^{T^vr},T^hr,\tilde \nabla_r ,\sigma(o_r))$.

We can consider the geometric family
$\cW:=(W\to B,g^{T^vp},T^hp,S^c(T^vp))$ and apply the construction \ref{pap55} 
in order to define the geometric family $r^\lambda_!(\cW)$ over $A$.
The underlying submersion of the family is
$q:=r\circ p\colon W\to A$. Its vertical bundle has a metric $g_\lambda^{T^vq}$ and a horizontal distribution $T^hq$.
The topological $Spin^c$-structures of $T^vp$ and $T^vr$ induce a topological $Spin^c$-structure on
$T^vq=T^vp\oplus p^*T^vr$. The family of Clifford bundles of $p_!\cW$ is the
spinor bundle associated to this $Spin^c$-structure. 

In order to understand how the connection $\tilde \nabla_q^\lambda$ behaves as
$\lambda\to 0$ we choose local spin structures on $T^vp$ and $T^vr$. Then we write
$S^c(T^vp)\cong S(T^vp)\otimes L_p$ and $S^c(T^vr)\cong S(T^vr)\otimes L_r$ for one-dimensional twisting bundles with connection $L_p, L_r$. The two local spin structures induce a local spin structure  on $T^vq\cong T^vp\oplus p^*T^vr$. We get $S^c(T^vq)\cong S(T^vq)\otimes L_q$ with $L_q:=L_p\otimes p^*L_r$.
The connection $\nabla_q^{\lambda,T^vq}$ converges as $\lambda\to 0$. Moreover,
the twisting connection on $L_q$ does not depend on $\lambda$ at all.
Since $\nabla_q^{\lambda,T^vq}$ and $\nabla_q^{L}$ determine $\tilde \nabla^\lambda_q$ (see \ref{origeo})
we conclude that the connection $\tilde \nabla^\lambda_q$ converges as $\lambda\to 0$. We introduce the following notation for this adiabatic limit:
$$\tilde \nabla^{adia}:=\lim_{\lambda\to 0}\tilde \nabla^\lambda_q\ .$$

\subsubsection{}

We keep the situation described in \ref{pap200}.
 \begin{ddd}\label{def100}
We define the composite
$o^\lambda_q:=o_r\circ_\lambda o_p$ of the representatives of smooth $K$-orientations of
$p$ and $r$ 
by
$$o^\lambda_q:=(g_\lambda^{T^vq},T^hq,\tilde \nabla^\lambda_q,\sigma(o^\lambda_q))\ ,$$
where
$$\sigma(o_q^\lambda):=\sigma(o_p)\wedge p^*\hA^c(o_r)+\hA^c(o_p)\wedge
p^*\sigma(o_r)- \tilde\hA^c (\tilde \nabla^{adia},\tilde \nabla^\lambda_q)-d\sigma(o_p)\wedge p^*\sigma(o_r)\ .$$
\end{ddd}
\begin{lem}\label{lem19123}
This composition of representatives of smooth $\hat K$-orientations preserves
equivalence and induces a well-defined composition of smooth $K$-orientations
which is independent of $\lambda$. 
\end{lem}
\begin{proof}
We first show that $o_q^\lambda$ is independent of $\lambda$. 
In view of \ref{pap201} for $\lambda_0<\lambda_1$ we must show that
$\sigma(o^{\lambda_1}_q)-\sigma(o_q^{\lambda_0})=\tilde \hA^c(\tilde \nabla^{\lambda_1}_q,\tilde \nabla^{\lambda_0}_q)$.
In fact, inserting the definitions and using  (\ref{eq7}) and (\ref{eq8}) we have 
$$\sigma(o^{\lambda_1}_q)-\sigma(o_q^{\lambda_0})= - \tilde\hA^c (\tilde \nabla^{adia},\tilde \nabla^{\lambda_1}_q)+\tilde\hA^c (\tilde \nabla^{adia},\tilde \nabla^{\lambda_0}_q)=\tilde \hA^c(\tilde \nabla^{\lambda_1}_q,\tilde \nabla^{\lambda_0}_q)\ .$$

Let us now take another representative $ o_p^\prime$.
The following equalities hold in the limit $\lambda\to 0$.
\begin{eqnarray*}
\lefteqn{
\sigma(o_q)-\sigma(o_q^\prime)}&&\\&=&
(\sigma(o_p)- \sigma( o_p^\prime))\wedge p^*\hA^c(o_r)+
(\hA^c(o_p)-\hA^c(o_p^\prime))\wedge p^*\sigma(o_r)
 -d(\sigma(o_p)-\sigma(o_p^\prime))\wedge p^*\sigma(o_r)\\
&=&\tilde \hA^c(\tilde \nabla_p,\tilde  \nabla^\prime_p) \wedge
p^*\hA^c( o_r)+(\hA^c(\tilde \nabla_p)-\hA^c(\tilde \nabla^\prime_p)-d\tilde
\hA^c(\tilde \nabla_p,\tilde  \nabla^\prime_p))\wedge p^*\sigma(o_r)\\
 &=&\tilde \hA^c(\tilde \nabla^{adia}_q,\tilde \nabla^{\prime adia}_q)  
\end{eqnarray*}
The last equality uses (\ref{eqq7}) and that in the adiabatic limit
\begin{equation}\label{eq3001} \hA^c(\tilde \nabla^{adia}_q)=\hA^c(\tilde \nabla_p)\wedge p^*\hA^c( \nabla_r)\ ,\end{equation} 
which implies
a corresponding formula for the adiabatic limit of transgressions,
$$\tilde \hA^c(\tilde \nabla^{adia}_q,\tilde \nabla^{\prime adia}_q)=\tilde \hA^c(\tilde
\nabla_p,\tilde \nabla^\prime_p)\wedge p^*\hA^c(\nabla_r)\ .$$  

Next we consider the effect of changing the representative $o_r$ to the equivalent one $ o_r^\prime$. 
We compute in the adiabatic limit
\begin{eqnarray*}
\sigma(o_q)-\sigma(o_q^\prime)&=&\sigma(o_p)\wedge ( p^*\hA^c(o_r)-
p^*\hA^c(o_r^\prime))
+(\hA^c(o_p)-d\sigma(o_p))\wedge p^* (\sigma(o_r)-\sigma(o_r^\prime))\\ 
 &=&\sigma(o_p)\wedge dp^*\tilde \hA^c(\tilde \nabla_r,\tilde \nabla^\prime_r)+(\hA^c(o_p)-d\sigma(o_p))\wedge p^*
\tilde\hA^c(\tilde \nabla_r,\tilde\nabla^\prime_r)\\
&=&\hA^c(o_p)\wedge p^*\tilde \hA^c(\tilde \nabla_r,\tilde \nabla^\prime_r)\\
&=&\tilde \hA^c(\tilde \nabla^{adia}_q,\tilde \nabla^{\prime adia}_q)\ .
 \end{eqnarray*}
In the last equality we have used again (\ref{eq3001}) and the corresponding
equality
$$\tilde \hA^c(\tilde \nabla^{adia}_q,\tilde \nabla^{\prime adia}_q)=\hA^c(o_p)\wedge p^*\tilde \hA^c(\tilde \nabla_r,\tilde\nabla^\prime_r)\ .$$
\end{proof} 

\subsubsection{}\label{tz123}

We consider the composition of proper $K$-oriented submersions
$$\xymatrix{W\ar@/_0.5cm/[rr]_q\ar[r]^p&B\ar[r]^r&A}$$
with representatives of smooth $K$-orientations $o_p$ of $p$ and $o_r$ of $r$.  We let $o_q:=o_r\circ o_p$ be the composition.
These choices define push-forwards $\hat p_!$, $\hat r_!$ and $\hat q_!$ in
smooth $K$-theory.

\begin{theorem}\label{funktt}
We have the equality of homomorphisms $\hat K(W)\rightarrow \hat K(A)$
$$\hat q_!=\hat r_!\circ \hat p_!\ .$$
\end{theorem}
\begin{proof}
We calculate the push-forwards and the composition of the $K$-orientations using the parameter $\lambda=1$ (though we do not indicate this in the notation).
We take a class $[\cE,\rho]\in\hat K(W)$.
The following equality holds since $\lambda=1$:
$$q_!\cE=r_!(p_!\cE)\ .$$
So we must show that 
\begin{eqnarray}\lefteqn{
\int_{W/A}  \hA^c(o_q)\wedge \rho + \tilde
\Omega(q,1,\cE) +\int_{W/A}\sigma(o_q) \wedge 
R([\cE,\rho])}&&\label{rr1}\\
&\equiv&\int_{B/A}\hA^c(o_r)\wedge \left[\int_{W/B}  \hA^c(o_p)\wedge \rho + \tilde
\Omega(p,1,\cE)+\int_{W/B}\sigma(o_p) \wedge  R([\cE,\rho])\right]
\nonumber \\&&+\tilde
\Omega(r,1,p_!\cE) +\int_{B/A}\sigma(o_r)\wedge
R(p_![\cE,\rho] )\ .\nonumber
\end{eqnarray}
where $\equiv$ means equality modulo $\im(d)+\ch_{dR}(K(A))$.
The form $\Omega(q,1,\cE)$ is given by (\ref{dtqdutwqdwqdq}). Since in the present
paragraph we consider these transgression forms for various bundles we have included the projection $q$ as an argument.

By Proposition \ref{mainprop} we have
$$
R (\hat p_![\cE,\rho])=\int_{W/B}
(\hA^c(o_p)-d\sigma(o_p))\wedge R([\cE,\rho])\ .
$$
Next we observe that \begin{equation}\label{eq400}
\tilde
\Omega(q,1,\cE)\equiv \tilde
\Omega(r,1,p_!\cE)+\int_{W/A} \tilde
\hA^c(\tilde \nabla^{adia},\tilde \nabla_q)\wedge
\Omega(\cE)+\int_{B/A}\hA^c(o_r)\wedge \tilde \Omega(p,1,\cE)\ ,\end{equation}
(where $\equiv$ means equality up to $\im(d)$).
To see this we consider the two-parameter family
$r^\lambda_!\circ p_!^\mu(\cE)$, $\lambda,\mu>0$, of geometric families. There is a natural geometric family $\cF$ over $(0,1]^2\times A$ which restricts to $r^\lambda_!\circ p_!^\mu(\cE)$ on $\{(\lambda,\mu)\}\times A$ (see \ref{pap55} for the one-parameter case).
Note that the local index form $\Omega(\cF)$ extends by continuity to $[0,1]^2\times A$.
If $P\colon[0,1]\hookrightarrow [0,1]^2$ is a path, then one can form the
integral 
$\int_{P\times A/A} \Omega(\cF_{|P\times A})$, the transgression of the local index form of $r^\lambda_!\circ p_!^\mu(\cE)$ along the path $P$.  The following square indicates four paths in the $(\lambda,\mu)$-plane. The arrows are labeled by the evaluations of $\Omega(\cF)$ (which follow from the adiabatic limit formula \ref{eq88}), and their integrals, the  corresponding transgression forms:
$$\xymatrix{(0,1)\ar[rrrrrr]^{\tilde
\Omega(r,1,p_!\cE)}_{\Omega(r_!^\lambda\circ p_!(\cE))}&&&&&&(1,1)\\&&&&&&\\&&&&&&\\(0,0)\ar[uuu]^{\int_{B/A}\hA^c(o_r)\wedge \Omega(p_!^\mu\cE)}_{\int_{B/A}\hA^c(o_r)\wedge \tilde \Omega(p,1,\cE)}\ar[rrrrrr]^{\int_{W/A}\hA^c(o_r\circ_\lambda o_p)\wedge \Omega(\cE)}_{\int_{W/A} \tilde \hA^c(\tilde \nabla_q,\tilde \nabla^{adia})\wedge
\Omega(\cE)}&&&&&&(1,0)\ar[uuu]_{\tilde
\Omega(q,1,\cE)}^{\Omega(r_!\circ p^\mu_!(\cE))}}.$$
Note the equality  $r_!\circ p^\mu_!(\cE)=q^\mu_!(\cE)$ which is relevant for
the right vertical path. Also note that for the lower horizontal path that ,
as $\mu\to 0$, the fibres of $\cE$ are scaled to zero, whereas the fibres of
$p$ are scaled by $\lambda$. The latter is exactly the effect of the scaled
composition $o_r\circ_\lambda o_p$ of orientations defined in \ref{pap200},
explaining its appearence in the above formula. 
The equation (\ref{eq400}) follows since the transgression
is additive under composition of paths, and since the transgression along a closed contractible path gives an exact form.

We now insert Definition \ref{def100} of $\sigma(o_q)$ in order to get  
\begin{eqnarray}\lefteqn{
\int_{W/A}\sigma(o_q) \wedge  R([\cE,\rho])}&&\nonumber\\
&=&\int_{W/A}\left[\sigma(o_p)\wedge p^*\hA^c(o_r)+\hA^c(o_p)\wedge
p^*\sigma(o_r) -d\sigma(o_p)\wedge p^*\sigma(o_r)-\tilde\hA^c
(\tilde \nabla^{adia},\tilde \nabla_q)\right]\wedge R([\cE,\rho]) \nonumber\\
&=&\int_{W/A}\left[\sigma(o_p)\wedge p^*\hA^c(o_r)+\hA^c(o_p)\wedge
p^*\sigma(o_r) -d\sigma(o_r)\wedge p^*\sigma(o_r) \right]\wedge R([\cE,\rho])\nonumber \\&&-
\int_{W/A}\tilde\hA^c
(\tilde \nabla^{adia},\tilde \nabla_q)\wedge \Omega(\cE) + \int_{W/A}\tilde\hA^c
(\tilde \nabla^{adia},\tilde \nabla_q)\wedge d\rho\nonumber\\
&=&\int_{W/A}\left[\sigma(o_p)\wedge p^*\hA^c(o_r)+\hA^c(o_p)\wedge
p^*\sigma(o_r) -d\sigma(o_p)\wedge p^*\sigma(o_r) \right]\wedge R([\cE,\rho]) \nonumber\\&&-
\int_{W/A}\tilde\hA^c
(\tilde \nabla^{adia},\tilde \nabla_q)\wedge \Omega(\cE)+\int_{W/A}  \left(\hA^c(o_p)\wedge
p^*\hA^c(o_r)-\hA^c(o_q)\right) \wedge
\rho\label{eq355}
\end{eqnarray}
We insert (\ref{eq355}) and  (\ref{eq400}) into the left-hand side of (\ref{rr1}).

\begin{eqnarray*}\lefteqn{\int_{W/A}  \hA^c(o_q)\wedge \rho + \tilde
\Omega(q,1,\cE) +\int_{W/A}\sigma(o_q) \wedge 
R([\cE,\rho])}&&\\
&\equiv& \int_{W/A}\hA^c(o_q)\wedge \rho \\
&& +\tilde
\Omega(r,1,p_!\cE)+\int_{W/A} \tilde
\hA^c(\tilde \nabla^{adia},\tilde \nabla_q)\wedge
\Omega(\cE)+\int_{B/A}\hA^c(o_r)\wedge \tilde \Omega(p,1,\cE)\\
&&+ \int_{W/A}\left[\sigma(o_p)\wedge p^*\hA^c(o_r)+\hA^c(o_p)\wedge
p^*\sigma(o_r) -d\sigma(o_p)\wedge p^*\sigma(o_r) \right]\wedge R([\cE,\rho]) \nonumber\\&&-
\int_{W/A}\tilde\hA^c
(\tilde \nabla^{adia},\tilde \nabla_q)\wedge \Omega(\cE)+\int_{W/A}  \left(\hA^c(o_p)\wedge
p^*\hA^c(o_r)-\hA^c(o_q)\right) \wedge
\rho \\
&=&\tilde
\Omega(r,1,p_!\cE) +\int_{B/A}\hA^c(o_r)\wedge \tilde \Omega(p,1,\cE)
 \\&&+\int_{W/A}\left[\sigma(o_p)\wedge p^*\hA^c(o_r)+\hA^c(o_p)\wedge
p^*\sigma(o_r) -d\sigma(o_p)\wedge p^*\sigma(o_r) \right]\wedge R([\cE,\rho]) \\
&&+\int_{W/A}  \hA^c(o_p)\wedge
p^*\hA^c(o_r) \wedge
\rho\ .
\end{eqnarray*}
An inspection shows that this is exactly the right-hand side of
(\ref{rr1}).
 \end{proof}

\section{The cup product}\label{jkjdkdqwdqdwd}

\subsection{Definition of the product}

\subsubsection{}
In this section we define and study the cup product 
$$\cup\colon \hat K(B)\otimes \hat K(B)\rightarrow \hat K(B)\ .$$
It turns smooth $K$-theory into a functor on manifolds with 
values in $\Z/2\Z$-graded rings and into a multiplicative extension of the
pair $(K,\ch_\R)$ in the sense of Definition \ref{multdef1}.

\subsubsection{}\label{wefwefwefwefwf}

Let $\cE$ and $\cF$ be geometric families over $B$.
The formula for the product involves the product $\cE\times_B\cF$ of geometric families over $B$.
The detailed description of the product is easy to guess, but let us employ the following trick in order 
to give an alternative definition.

Let $p\colon F\to B$ be the proper submersion with closed fibres underlying $\cF$. Let us for the moment assume that the vertical
metric, the horizontal distribution, and the orientation of $p$ are complemented by a topological
$Spin^c$-structure together with a $Spin^c$-connection $\tilde \nabla$ as in \ref{pap8}.
The Dirac bundle $\cV$ of $\cF$ has the form
$\cV\cong W\otimes S^c(T^vp) $ for a twisting bundle $W$ with a hermitean metric and unitary connection 
(and $\Z/2\Z$-grading in the even case), which is uniquely determined up to isomorphism.
Let $ p^*\cE\otimes W$ denote the geometric family which is obtained from $p^*\cE$ by twisting its
Dirac bundle with $\delta^*W$, where $\delta\colon E\times_BF\to F$ denotes the underlying proper submersion with closed fibres of $p^*\cE$. Then we have
$$\cE\times_B \cF\cong p_!(p^*\cE\otimes W)\ .$$
This description may help to understand the meaning of the adiabatic
deformation which blows up $\cF$, which in this notation is given by
$p^\lambda_!( p^*\cE\otimes W)$.

In the description of the product of geometric families we could interchange the roles of $\cE$ and $\cF$.

If the vertical bundle of $\cE$ does not have a global $Spin^c$-structure, then it has at least a local one. In this case the description above again gives a complete description of the local geometry of $\cE\times_B\cF$.

\subsubsection{}

We now proceed to the definition of the product in terms of cycles.
In order to write down the formula we assume that the cycles
$(\cE,\rho)$ and $(\cF,\theta)$ are homogeneous of degree
$e$ and $f$, respectively.

\begin{ddd}\label{proddefin}
We define
$$(\cE,\rho)\cup (\cF,\theta):=[\cE\times_B\cF,(-1)^e\Omega(\cE)\wedge
\theta +\rho\wedge \Omega(\cF)-(-1)^ed\rho\wedge \theta]\ .$$
\end{ddd}
\begin{prop}\label{prooow}
The product is well-defined. It turns $B\mapsto \hat K(B)$ into a functor from smooth manifolds
to unital graded-commutative rings.
\end{prop}
\begin{proof}
We first show that this product is bilinear and compatible with the  equivalence relation $\sim$ (\ref{uuu1}).
The product is obviously biadditive and natural with respect to pull-backs along maps $B^\prime\to B$. We now show that the product preserves the equivalence relation in the first argument.
Assume that $\cE$ admits  a taming $\cE_t$.
Then we have
$(\cE,\rho)\sim (\emptyset,\rho-\eta(\cE_t))$. Using the latter representative 
we get
\begin{eqnarray*}
(\emptyset,\rho-\eta(\cE_t))\cup (\cF,\theta) &=&[\emptyset,(\rho-\eta(\cE_t))\wedge \Omega(\cF)
-(-1)^ed\rho\wedge \theta+(-1)^ed\eta(\cE_t)\wedge \theta]\\
&=&[\emptyset,\rho\wedge\Omega(\cF)+(-1)^e\Omega(\cE)\wedge \theta 
-(-1)^ed\rho\wedge \theta -\eta(\cE_t)\wedge \Omega(\cF)]\ .
\end{eqnarray*}
On the other hand, similar to in \ref{pap66}, the taming $\cE_t$ induces a generalized taming 
$(\cE\times_B\cF)_t$. Using Lemma \ref{lem33} and argueing as in the proof of Lemma \ref{pass}
we get
\begin{eqnarray*}
\lefteqn{[\cE\times_B\cF,(-1)^e\Omega(\cE)\wedge
\theta +\rho\wedge \Omega(\cF)-(-1)^ed\rho\wedge \sigma]}&&\\&=&[\emptyset,(-1)^e\Omega(\cE)\wedge
\theta +\rho\wedge \Omega(\cF)-(-1)^ed\rho\wedge
\sigma-\eta((\cE\times_B\cF)_t)] \ .\end{eqnarray*}
It suffices to show that
\begin{equation}\label{iuqhduiwqdqdqwwqw}
\eta(\cE_t)\wedge \Omega(\cF)- \eta((\cE\times_B\cF)_t)\in\im(\ch_{dR})\ .
\end{equation}
We will actually show that this difference is exact.

We first consider the adiabatic limit in which we blow up the metric of
$\cF$. We get from  Theorem \ref{adia1} 
\begin{equation}\label{uzddwqdddededdeded}
\lim_{adia} \eta((\cE\times_B\cF)_t)=\eta(\cE_t)\wedge \Omega(\cF)\ .
\end{equation}
In order to see this we use that 
$\cE\times_B\cF\cong p_!(p^*\cE\otimes W)$
(see \ref{wefwefwefwefwf}), where $p\colon F\to B$ and $W\to F$ is the twisting bundle of this family. The taming $\cE_t$ induces a taming $p^*\cE_t$, and hence a  taming $( p^*\cE\otimes W)_t$. It follows from standard properties of the induced superconnection on a tensor product bundle (alternatively  one can use the special case of Theorem \ref{adia1} where the second fibration has zero-dimensional fibres) that
$\eta(p^*\cE\otimes W)_t=p^*\eta(\cE_t)\wedge\ch(\nabla^W)$. From Theorem
\ref{adia1} we get ($\tilde \nabla$ is associated to $p$) 
\begin{eqnarray*}
\lim_{adia}\eta((\cE\times_B\cF)_t)&=&\lim_{\lambda\to 0}\eta(p_!^\lambda(p^*\cE\otimes W)_t)\\
&=&\eta(\cE_t)\wedge \left(\int_{F/B} \hA^c(\tilde \nabla)\wedge \ch(\nabla^W)\right)\\&=&\eta(\cE_t)\wedge \Omega(\cF)
\end{eqnarray*}
 As in \ref{pap55} we now let $\cG_t$ be the tamed family over $(0,\infty)\times B$ with underlying projection $r\colon (0,\infty)\times E\times_BF\to (0,\infty)\times B$ which
restricts to $p_!^\lambda(p^*\cE\otimes W)_t$ on $\{\lambda\}\times B$.
Then we have $d\eta(\cG_t)=\Omega(\cG)$. Using the formulas for $\nabla^{T^vr}$ given in
\cite[Prop. 10.2]{bgv} we observe that
$i_{\partial^H_\lambda}R^{\nabla^{T^vr}}=0$, where $\partial_\lambda^H$ is a horizontal lift of $\partial_\lambda$. This implies that $i_{\partial_\lambda}d\eta(\cG_t)=i_{\partial_\lambda}\Omega(\cG)=0$.  
We get
$$\eta(p_!^\lambda(p^*\cE\otimes W)_t)-\eta(p_!^1(p^*\cE\otimes W)_t)=d\int_{[\lambda,1]\times B/B} \eta(\cG^t)\ .$$
The exactness of the difference (\ref{iuqhduiwqdqdqwwqw}) now follows by taking the limit $\lambda\to 0$ and the fact that the range of $d$ is closed since
$\lim_{\lambda\to 0}\eta(p_!^\lambda(p^*\cE\otimes W)_t)=\eta(\cE_t)\wedge \Omega(\cF)$ by (\ref{uzddwqdddededdeded}) and
$\eta(p_!^1(p^*\cE\otimes W)_t)=\eta((\cE\times_B\cF)_t)$ by construction.

In order to avoid repeating this argument for the second argument 
we show that the product is graded commutative.
  Note that
$\cE\times_B\cF\cong \cF\times_B\cE$ except if both families are odd, in which case
$\cE\times_B\cF\cong (\cF\times_B\cE)^{op}$
\begin{eqnarray*}
[\cE,\rho]\cup [\cF,\theta]&=&[\cE\times_B\cF,(-1)^e\Omega(\cE)\wedge
\theta +\rho\wedge \Omega(\cF)-(-1)^ed\rho\wedge \theta]\\
&=&[(-1)^{ef}\cF\times_B\cE,(-1)^{e+e(f-1)}\theta \wedge\Omega(\cE)
+(-1)^{f(e-1)} \Omega(\cF)\wedge \rho-\rho\wedge d \theta]\\
&=&[(-1)^{ef}\cF\times_B\cE,(-1)^{ef}\theta \wedge\Omega(\cE)
+(-1)^{ef} (-1)^f \Omega(\cF)\wedge \rho-(-1)^{ef}(-1)^f d\theta\wedge \rho]\\
&=&(-1)^{ef} [\cF,\theta]\cup [\cE,\rho]\ .
\end{eqnarray*}

\subsubsection{}
We now have a well-defined $\Z/2\Z$-graded commutative product
$$\cup\colon \hat K(B)\otimes \hat K(B)\to \hat K(B)\ .$$
We show next that it is associative. First of all observe that the fibre product of geometric families is associative.
Let $e,f,g$ be the parities of the homogeneous classes
$[\cE,\rho]$, $[\cF,\theta]$, and $[\cG,\kappa]$.
\begin{eqnarray*}
\lefteqn{([\cE,\rho]\cup [\cF,\theta])\cup [\cG,\kappa]}&&\\
&=&
[\cE\times_B\cF,(-1)^e\Omega(\cE)\wedge
\theta +\rho\wedge \Omega(\cF)-(-1)^ed\rho\wedge \theta]\cup [\cG,\kappa]\\
&=&[\cE\times_B\cF\times_B\cG,((-1)^e\Omega(\cE)\wedge
\theta +\rho\wedge \Omega(\cF)-(-1)^ed\rho\wedge \theta)\wedge \Omega(\cG)\\&&+
(-1)^{e+f}\Omega(\cE\times_B\cF)\wedge \kappa - (-1)^{e+f}d((-1)^e\Omega(\cE)\wedge
\theta +\rho\wedge \Omega(\cF)-(-1)^ed\rho\wedge \theta)\wedge \kappa]\\
&=&[\cE\times_B\cF\times_B\cG,(-1)^e\Omega(\cE)\wedge \theta\wedge
\Omega(\cG)+\rho\wedge
\Omega(\cF)\wedge\Omega(\cG)\\&&-(-1)^ed\rho\wedge\theta\wedge\Omega(\cG)+(-1)^{e+f}\Omega(\cE)\wedge
\Omega(\cF)\wedge \kappa-(-1)^{e+f}\Omega(\cE)\wedge d\theta
\wedge \kappa\\&&-(-1)^{e+f}d\rho\wedge\Omega(\cF)\wedge
\kappa+(-1)^{e+f}d\rho\wedge d\theta \wedge \kappa]\\
\end{eqnarray*}
On the other hand
\begin{eqnarray*}\lefteqn{
[\cE,\rho]\times ([\cF,\theta]\times [\cG,\kappa])}&&\\&=&
[\cE,\rho]\times[\cF\times_B \cG,(-1)^f\Omega(\cF)\wedge \kappa +
\theta\wedge \Omega(\cG)-(-1)^fd\theta\wedge \kappa]\\
&=&[\cE\times_B\wedge \cF\times_B\cG,(-1)^{e}\Omega(\cE)\wedge ((-1)^f\Omega(\cF)\wedge \kappa +
\theta\wedge \Omega(\cG)-(-1)^fd\theta\wedge \kappa)\\&&+\rho\wedge
\Omega(\cF\times_B\cG)-(-1)^ed\rho\wedge  ((-1)^f\Omega(\cF)\wedge \kappa +
\theta\wedge \Omega(\cG)-(-1)^fd\theta\wedge \kappa)]\\
&=&[\cE\times_B
\cF\times_B\cG,(-1)^{e+f}\Omega(\cE)\wedge\Omega(\cF)\wedge\kappa+(-1)^e\Omega(\cE)\wedge
\theta\wedge\Omega(\cG)\\
&&-(-1)^{ e+f}\Omega(\cE)\wedge d
\theta\wedge\kappa+\rho\wedge\Omega(\cF)\wedge\Omega(\cG)-(-1)^{e+f}d\rho\wedge\Omega(\cF)\wedge\kappa\\&&
-(-1)^ed\rho\wedge\theta\wedge\Omega(\cG)+(-1)^{e+f}d\rho\wedge d\theta\wedge\kappa]
\end{eqnarray*}
By an inspection we see that the two right-hand sides agree.

\subsubsection{}
Let us observe that the unit $1\in \hat K(B)$ is simply given by
$(B\times \C,0)$, i.e.~the trivial $0$-dimensional family with fibre
the graded vector space $\C$ concentrated in even degree, and with
curvature form $1$. The
definition shows that this is actually a unit on the level of cycles.
This finishes the proof of Proposition \ref{prooow}.
\end{proof}  
 
\subsubsection{}

In this paragraph we study the compatibility of the cup product in smooth $K$-theory with the cup product in topological K-theory and the wedge product of differential forms.
\begin{lem}\label{rring}
For $x,y\in \hat K(B)$ we have 
$$R(x\cup y)=R(x)\wedge R(y)\ ,\quad I(x\cup y)=I(x)\cup I(y)\ .$$
Furthermore, for $\alpha\in \Omega(B)/\im(d)$ we have
$$a(\alpha)\cup x=a(\alpha\wedge R(x))\ .$$
\end{lem}
\begin{proof}
Straight forward calculation using the definitions.
\end{proof}

\begin{kor}
With the $\cup$-product smooth $K$-theory $\hat K$ is a multiplicative extension of the pair $(K,\ch_\R)$. 
\end{kor}

\subsection{Projection formula}

\subsubsection{}

Let $p\colon W\rightarrow B$ be a proper submersion with closed fibres with a smooth $K$-orientation represented by $o$.  In this case  we have a well-defined  push-forward
$\hat p_!\colon \hat K(W)\to \hat K(B)$. The explicit formula in terms of cycles is (\ref{eq300}).
The projection formula states the compatibility of the push-forward with the $\cup$-product.
\begin{prop} \label{projcl}
Let $x\in \hat K(W)$ and $y\in \hat K(B)$. 
Then
$$\hat p_!(p^*y\cup x)=y\cup \hat p_!(x)\ .$$
\end{prop}
\begin{proof} Let $x=[\cF,\sigma]$ and $y=[\cE,\rho]$. 
By an inspection of the constructions we observe that the projection formula holds true on the level of
geometric families $$p_!( p^*\cE\times_W \cF)\cong \cE\times_B p_!\cF\ .$$
This implies  $$\Omega(p^\lambda_!(p^*\cE\times_W
\cF))=\Omega(\cE) \wedge \Omega(p^\lambda_!(\cF))\ .$$ Consequently we have
$\tilde \Omega(\lambda,p^*\cE\times_W \cF)=(-1)^e \Omega(\cE)\wedge \tilde \Omega(\lambda,\cF)$.
 Inserting the definitions of the product and the push-forward we get up to
 exact forms
\begin{eqnarray}
\lefteqn{\hat p_!(p^*y\cup
  x)}&&\nonumber\\&=&\hat p_!([p^*\cE\times_W\cF,(-1)^ep^*\Omega(\cE)\wedge
\sigma+p^*\rho\wedge \Omega(\cF)-(-1)^ep^*d\rho \wedge
\sigma])\nonumber\\
&=&[p_!( p^*\cE\times_W \cF),
\int_{W/B}\hA^c(o)\wedge \left[(-1)^ep^*\Omega(\cE)\wedge
  \sigma+p^*\rho\wedge \Omega(\cF)-(-1)^ep^*d\rho\wedge \sigma\right]\nonumber\\&&+ 
\int_{W/B}\sigma(o) \wedge R(p^*y\cup x)+\tilde
\Omega(1,p^*\cE\times_W \cF)]\nonumber\\
&=&[\cE\times_B p_!\cF, \rho\wedge
\int_{W/B}\hA^c(o)\wedge \Omega(\cF) +
(-1)^e\Omega(\cE)\wedge \int_{W/B}\hA^c(o)\wedge
  \sigma\nonumber\\
&&+(-1)^e \Omega(\cE)\wedge \tilde
\Omega(1,\cF)\nonumber\\
&&- \rho\wedge \int_{W/B}\hA^c(o) \wedge d\sigma+ (-1)^e
R(y)\wedge \int_{W/B} \sigma(o)\wedge R (x)] \label{eq788}\ .
\end{eqnarray}
Up to exact forms we have
\begin{eqnarray*}
&&\rho\wedge
\int_{W/B}\hA^c(o)\wedge \Omega(\cF) +
(-1)^e\Omega(\cE)\wedge \int_{W/B}\hA^c(o)\wedge
  \sigma\\
&&+(-1)^e \Omega(\cE)\wedge \tilde
\Omega(1,\cF)\\
&&- \rho\wedge \int_{W/B}\hA^c(o) \wedge d\sigma+ (-1)^e
R(y)\wedge \int_{W/B} \sigma(o)\wedge R(x)\\
&=&(-1)^e \Omega(\cE)\wedge \left(  \int_{W/B}\hA^c(o)\wedge
  \sigma+ \tilde
\Omega(1,\cF) + \int_{W/B} \sigma(o)\wedge R (x)
\right)\\
&&+\rho\wedge  \int_{W/B}\hA^c(o)\wedge 
  \left(\Omega(\cF)-d\sigma\right))-(-1)^e d\rho\wedge \int_{W/B} \sigma(o)\wedge R (x)
\\
&=&(-1)^e \Omega(\cE)\wedge \left(  \int_{W/B}\hA^c(o)\wedge
  \sigma+ \tilde
\Omega(1,\cF) + \int_{W/B} \sigma(o)\wedge R (x)
\right)\\
&&+\rho\wedge  \int_{W/B}(\hA^c(o)-d\sigma(o))\wedge 
  R(x) 
\\&=&(-1)^e \Omega(\cE)\wedge \left(  \int_{W/B}\hA^c(o)\wedge
  \sigma+ \tilde
\Omega(1,\cF) + \int_{W/B} \sigma(o)\wedge R (x)
\right)\\
&&+\rho\wedge  R(\hat p_!x)\ .
\end{eqnarray*}
Thus the form component of (\ref{eq788}) is exactly the one needed for the product
$y\cup p_!(x)$. \end{proof}

\subsection{Suspension}

\subsubsection{}

We consider the projection $\pr_2\colon S^1\times B\to B$. 
The goal of this subsection is to verify the relation
$$(\hat\pr_2)_!\circ \pr_2^*=0$$
which is an important ingredient in the uniqueness result Theorem \ref{wieth}.

\subsubsection{}\label{pullpush1}

The projection $\pr_2$  fits into the cartesian diagram
$$\xymatrix{S^1\times B\ar[r]^{\pr_1}\ar[d]^{\pr_2}&S^1\ar[d]^p\\B\ar[r]^r&{*}}\ .$$
We choose the metric $g^{TS^1}$ of unit volume and the bounding spin structure on $TS^1$.
This spin structure induces a $Spin^c$ structure on $TS^1$ together with the connection $\tilde \nabla$.
In this way we get a representative $o$ of a smooth $K$-orientation of $p$. 
By pull-back we get the representative $r^*o$ of a smooth $K$-orientation of $\pr_2$
which is used to define $(\hat\pr_2)_!$.

\subsubsection{}\label{uiwfweqfewfwfef}

Using the projection formula Proposition \ref{projcl} we get for $x\in \hat
K(B)$ 
$$(\hat\pr_2)_!(\pr_2^*(x))=(\hat\pr_2)_!(\pr_2^*(x)\cup 1)=x\cup
(\hat\pr_2)_!1\ .$$
Using the compatibility of the push-forward with cartesian diagrams Lemma \ref{cartcomp}
we get
$$(\hat\pr_2)_!1=(\hat\pr_2)_!(\pr_1^*(1))=r^*\hat p_!(1)\ .$$

We let $\cS^1$ denote the geometric family over $*$ given by $p\colon S^1\to *$ with the geometry described above.
Since $S^1$ has the bounding $Spin$-structure the Dirac operator is invertible and has a symmetric spectrum.
The family $\cS^1$ therefore has a canonical taming $\cS^1_t$ by the zero smoothing operator, and we have 
$\eta(\cS^1_t)=0$. This implies
$$\hat p_!(1)=[\cS^1,0]=[\emptyset,\eta(\cS_t^1)]=[\emptyset,0]=0\ .$$
\begin{kor}\label{pullpush0}
We have $(\hat\pr_2)_!\circ \pr_2^*=0$.
\end{kor}

\section{Constructions of natural smooth $K$-theory classes}\label{sec5}

\subsection{Calculations}

\subsubsection{}

\begin{lem}\label{calpunk}
We have 
$$\hat K^*(*)\cong \left\{\begin{array}{cc} \Z & *  = 0\\
\R/\Z & *=1\end{array}\right.\ .$$
\end{lem}
\begin{proof}
We use the exact sequence given by  Proposition \ref{prop1}.
The  assertion follows from the obvious identities
$$\hat K^0(*)\cong K^0(*)\cong \Z\ ,\quad \hat K^1(*)\cong \Omega^{ev}(*)/\ch_{dR}(K^0(*))\cong \R/\Z\ .$$
 \end{proof}

\subsubsection{}

\begin{lem}\label{lemma6777}
There are exact sequences
\begin{eqnarray*}
0\rightarrow \R/\Z\rightarrow &\hat K^0(S^1)&\rightarrow
\Z\rightarrow 0\\
0\rightarrow C^\infty(S^1)/\Z\rightarrow &\hat K^1(S^1)&\rightarrow
\Z\rightarrow 0\ .
\end{eqnarray*}
\end{lem}
\begin{proof} 
These assertions again follow from Proposition \ref{prop1} and the identifications
$$K^0(S^1)\cong \Z\ ,\quad K^1(S^1)\cong \Z \ , \quad 
\Omega^{ev}(S^1)/\ch_{dR}(K^0(S^1))\cong C^\infty(S^1)/\Z\ .$$
\end{proof} 

\subsubsection{}

Let $\bV:=(V,h^V,\nabla^V,z)$ be a geometric $\Z/2\Z$-graded bundle over $S^1$
such that $\dim(V^+)=\dim(V^-)$. Let $\cV$ denote the corresponding geometric family.
By Lemma \ref{lemma6777} the class
$[\cV,0]\in \hat K^0(S^1)$ satisfies $I([\cV,0])=0$ and hence corresponds to an element
of $\R/\Z$. This element is calculated in the following lemma.
Let $\phi^\pm\in U(n)/conj$ denote the holonomies
of $V^\pm$ (well defined modulo conjugation in the group $U(n)$).

\begin{lem}\label{dersa1}
We have
$$[\cV,0]=a\left(\frac{1}{2\pi i}\log\frac{\det(\phi^+)}{\det(\phi^-)}\right)\ .$$
\end{lem}
\begin{proof}
We consider the map $q\colon S^1\to *$ with the canonical $K$-orientation \ref{pullpush1}.
By Proposition \ref{mainprop} we have a commutative diagram
$$
\begin{CD}
  \R/\Z @>{\sim}>> \Omega^1(S^1)/(\im(d)+\im(\ch_{dR})) @>a>> \hat K^1(S^1)\\
  @VV{=}V @VV{q^o_!}V @VV{\hat q_!}V\\
  \R/\Z @>{\sim}>> \Omega^0(*)/\im(\ch_{dR}) @>a>> \hat K^0(*)
\end{CD}
\ .$$
In order to determine $[\cV,0]$ it therefore suffices to calculate
$\hat q_!([\cV,0])$.
Now observe that $q\colon S^1\to *$ is the boundary of $p\colon D^2\to *$.
Since the underlying topological $K$-orientation of $q$ is given by the bounding $Spin$-structure
we can choose a smooth $K$-orientation of $p$ with product structure which restricts to the smooth $K$-orientation of $q$. 
The bundle
$\bV$ is topologically trivial. Therefore we can find a geometric bundle
$\bW=(W,h^W,\nabla^W,z)$, again with product structure, on $D^2$ which restricts to $\bV$ on the boundary.
Let $\cW$ denote the corresponding geometric family over
$D^2$. Later we prove the bordism formula Proposition \ref{bordin}. It gives
$$\hat q_!([\cV,0])=[\emptyset,p_!R([\cW,0])]=-a\left(\int_{D^2/*} \Omega^2(\cW)\right)\ .$$
Note that $$\Omega^2(\cW)=\ch_2(\nabla^W)=\ch_2(\nabla^{\det(W^+)})-\ch_2(\nabla^{\det(W^-)})=\frac{-1}{2\pi i}\left[  R^{\nabla^{\det W^+}}-
 R^{\det \nabla^{W^-}}\right]\ .$$
The holonomy $\det(\phi^\pm)\in U(1)$ of $\det(\bV^\pm)$ is equal to the integral  of the curvature of $\det \bW^\pm$:
$$\log\det(\phi^\pm)=\int_{D^2} R^{\nabla^{\det(W^\pm)}}\ .$$
It follows that
$$\hat q_!([\cV,0])=a\left(\frac{1}{2\pi i}\log\frac{\det
    (\phi^+)}{\det(\phi^-)} \right)\ .
$$
\end{proof}

\subsection{The smooth $K$-theory class of a mapping torus}

\subsubsection{}

Let $\cE$ be a geometric family over a point and consider an automorphism $\phi$  
of $\cE$. Then we can form the mapping torus
$T(\cE,\phi):=(\R\times\cE)/\Z$, where $n\in \Z$ acts on
$\R$ 
by $x\mapsto x+n$, and by $\phi^n$ on $\cE$. 
The product $\R\times \cE$ is a $\Z$-equivariant geometric family
over $\R$ (the pull-back of $\cE$ by the projection $\R\to *$).
The geometric structures descend to the quotient and 
turn the  mapping torus $T(\cE,\phi)$ into a geometric family
over $S^1=\R/\Z$. In the present subsection we study the class
$$[T(\cE,\phi),0]\in \hat K(S^1)\ .$$
In the following we will assume that
the parity of $\cE$ is even, and that $\ind(\cE)=0$.

\subsubsection{} 
 
Let $\dim\colon K^0(S^1)\to \Z$ be the dimension homomorphism, which in this case is an isomorphism.
Since $\dim I([T(\cE,\phi),0])=\dim(\ind(\cE))=0$ we have in fact
$[T(\cE,\phi),0]\in\R/\Z\subset \hat K^0(S^1)$, where we consider $\R/\Z$ as a subgroup of $\hat K^0(S^1)$ according to Lemma \ref{lemma6777}.

Let $V:=\ker(D(\cE))$. This graded vector space is preserved by the action of $\phi$. We use the same symbol in order to denote the induced action on $V$.

 We form the zero-dimensional family
$\cV:=(\R\times V)/\Z$ over $S^1$. This bundle is isomorphic to the
kernel bundle of $T(\cE,\phi)$. The bundle of Hilbert spaces of the family
$T(\cE,\phi)\sqcup_{S^1}\cV^{op}$ has a canonical subbundle of the form $\cV\oplus \cV^{op}$.
We choose the taming $(T(\cE,\phi)\sqcup_{S^1}\cV^{op})_t$
which is induced by the isomorphism
$$\left(\begin{array}{cc}0&1\\1&0\end{array}\right)$$ on this subbundle.
Note that
$[T(\cE,\phi),0]=[\cV,\eta((T(\cE,\phi)\sqcup_{S^1}\cV^{op})_t)]$.
Since the pull-back of $(T(\cE,\phi)\sqcup_{S^1}\cV^{op})_t$ under $\R\to \R/\Z$ is isomorphic to a tamed family pulled back under $\R\to *$
  we see that the one-form $\eta((T(\cE,\phi)\sqcup_{S^1}\cV^{op})_t)=0$.

\subsubsection{}
 
Thus it remains to evaluate $[T(\cE,\phi),0]=[\cV,0]\in \R/\Z$.
By Lemma \ref{dersa1} this number can be expressed in terms of the  holonomy of the determinant
bundle $\det(\cV)$. Let $\phi^\pm\in \Aut(V^\pm)$ be the induced transformations. 
\begin{prop}
We have $[T(\cE,\phi),0]=[\frac{1}{2\pi i} \log(\frac{\det \phi^+}{\det
  \phi^-})]_{\R/\Z}$.
In particular, if $D(\cE)$ is invertible, then $[T(\cE,\phi),0]=0$.
\end{prop}

\subsection{The smooth $K$-theory class of a geometric family with kernel bundle}\label{shskajd}

\subsubsection{}\label{ghghghr}
Let $\cE$ be an even-dimensional  geometric family over the base $B$. By $(D_b)_{b\in B}$ we denote the associated family of Dirac operators on the  family of Hilbert spaces
$(H_b)_{b\in B}$. 
The geometry of $\cE$ induces a connection $\nabla^H$ on this family (the connection part of the Bismut superconnection \cite[Prop. 10.15]{bgv}).
We assume that $\dim(\ker(D_b))$ is constant. 
In this case we can form a vector bundle $K:=\ker(D)$. The projection of $\nabla^H$ to $K$ gives
a connection $\nabla^{K}$. Hence we get a geometric bundle $\bK:=(K,h^K,\nabla^K)$ and an associated geometric family $\cK$ (see \ref{zerofibre}).

\subsubsection{}

The sum $\cE\sqcup_B \cK^{op}$ has a natural taming $(\cE\sqcup_B \cK^{op})_t$ which is given by
$$\left(\begin{array}{cc}
0&u\\u^*&0
\end{array}\right)\in \End(H_b\oplus K^{op}_b)\ ,$$
where $u\colon K_b\to H_b$ is the  embedding.
We thus have the following equality in $\hat K(B)$:
$$[\cE,0]=[\cK,\eta((\cE\sqcup_B \cK^{op})_t)]\ .$$

\subsubsection{}\label{etabcnot}

Under the standing assumption that $\dim(\ker(D_b))$ is constant we also have the $\eta$-form of Bismut-Cheeger $\eta^{BC}(\cE)\in \Omega(B)$ (see
\cite{MR1173033}, \cite{MR1052337}, \cite{MR1042214}).
Since other authors use  $\eta^{BC}(\cE)$, in the following two paragraphs we shall analyse the relation
between this and $\eta((\cE\sqcup_B \cK^{op})_t)$.

We form the geometric family
$[0,1]\times (\cE\sqcup_B \cK^{op})$ over $B$.
The taming $(\cE\sqcup_B \cK^{op})_t$ induces a boundary taming at
$\{0\}\times (\cE\sqcup_B \cK^{op})$.  In index theory the boundary taming is used to construct a perturbation of the  Dirac operator which is invertible 
at $-\infty$ of $(-\infty,1]\times (\cE\sqcup_B \cK^{op})$ (see \cite{math.DG/0201112} for details). 
On the other side $\{1\}\times (\cE\sqcup_B \cK^{op})$
we consider APS-boundary conditions. 
We thus get a family of perurbed Dirac operators on $(-\infty ,1]\times (\cE\sqcup_B \cK^{op})$. The $L^2$-boundary condition at $\{-\infty\}\times (\cE\sqcup_B \cK^{op})$ and the APS-boundary condition at $\{1\}\times (\cE\sqcup_B \cK^{op})$
together imply the Fredholm property (which can be checked locally for the various boundary components or ends).
In this way the family of Dirac operators
on $[0,1]\times (\cE\sqcup_B \cK^{op})$ gives rise to a family of Fredholm operators. 
We will denote this structure by 
$([0,1]\times (\cE\sqcup_B \cK^{op}))_{bt, APS}$.

The Chern character of its index $\ind(([0,1]\times (\cE\sqcup_B \cK^{op}))_{bt,APS})\in K(B)$ can be calculated using the methods of local index theory.

\subsubsection{}

Using \ref{pap3} we can choose a possibly different taming $(\cE\sqcup_B
\cK^{op})_{t^\prime}$
such that the corresponding index $\ind(([0,1]\times (\cE\sqcup_B \cK^{op}))_{bt',APS}) \in K(B)$ vanishes. In this case we can extend the boundary taming to a taming
$\ind(([0,1]\times (\cE\sqcup_B \cK^{op}))_{t',APS})$. 

We set up the method of local index theory as usual by forming the family of rescaled Bismut superconnections
$A_s:=A_s(([0,1]\times (\cE\sqcup_B \cK^{op}))_{t',APS})$ which take
the tamings and boundary tamings into account as explained in \cite[2.2.4.3]{math.DG/0201112}, see also  \ref{pap66}.
Invertibility of $D(([0,1]\times (\cE\sqcup_B \cK^{op}))_{t',APS})$
ensures exponential vanishing of the integral kernel of $\ee^{-A_s^2}$ for $s\to \infty$.
The usual transgression integral expresses the local index form $\Omega([0,1]\times (\cE\sqcup_B \cK^{op}))$ as 
a sum of  contributions of the boundary components or ends (see \cite[proof of Lemma 2.2.15 ]{math.DG/0201112}). These contributions can be calculated separately for each part.

Because of the product structure we have
$\Omega([0,1]\times (\cE\sqcup_B \cK^{op}))=0$.  The contribution of the boundary
$\{1\}\times (\cE\sqcup_B \cK^{op})$ is given by the proof of the APS-index theorem of 
\cite{MR1173033}, \cite{MR1052337}, \cite{MR1042214},
and it is equal to $\eta^{BC}(\cE\sqcup_B\cK^{op})=\eta^{BC}(\cE)$. The second equality holds true, since the Dirac operator for $\cK^{op}$ is trivial. The contribution of the boundary 
$\{0\}\times (\cE\sqcup_B \cK^{op})$ is calculated in the proof 
 of \cite[Lemma 2.2.15]{math.DG/0201112} and equal to $-\eta((\cE\sqcup_B \cK^{op})_{t^\prime})$.
Therefore we have
$\eta^{BC}(\cE)=\eta((\cE\sqcup_B \cK^{op})_{t^\prime})$ (note that we
calculate modulo exact forms).
We now use 
 \ref{pap3} and a relative index theorem (compare
 (\ref{wiufhewfwefewfewfwfwefwef})) in order to see that 
$$\eta((\cE\sqcup_B \cK^{op})_{t^\prime})-\eta((\cE\sqcup_B \cK^{op})_t)=
\ch_{dR}(\ind(([0,1]\times (\cE\sqcup_B \cK^{op}))_{bt,APS}))\in
\ch_{dR}(K(B))\ .$$
Using Proposition \ref{prop1} we get:
\begin{kor}\label{bcverg}
We have $[\cE,0]=[\cK,\eta^{BC}(\cE)]$.
\end{kor}

\subsubsection{}

Let $p\colon W\to B$ be a proper submersion with closed fibres with a smooth $K$-orientation represented by $o$.
Let $\bV$ be a geometric vector bundle over $W$, and let $\cV$ denote the associated geometric family.
Then we can form the geometric family $\cE:=p_!\cV$ (see Definition \ref{ddd7771}).
Assume that the kernel of the family of Dirac operators $(D(\cE_b))_{b\in B}$
has constant dimension, forming thus the kernel bundle $\cK$.
Since $\cV$ has zero-dimensional fibres we have $\tilde \Omega(1,\cV)=0$. From (\ref{eq300}) we get
\begin{eqnarray*}
\hat p_![\cV,\rho]&=&[p_!\cV,\int_{W/B} \hA^c(o)\wedge \rho+\int_{W/B}\sigma(o)\wedge (\Omega(\cV)-d\rho)]\\
&=&[\cE,\int_{W/B}\hA^c(o)\wedge \rho +\int_{W/B}\sigma(o)\wedge (\Omega(\cV)-d\rho)]\\
&=&[\cK,\eta^{BC}(\cE)+\int_{W/B}\hA^c(o)\wedge \rho+ \int_{W/B}\sigma(o)\wedge (\Omega(\cV)-d\rho)]\ .
\end{eqnarray*}

\subsection{A canonical $\hat K^1$-class on $S^1$}

\subsubsection{}

We construct in a natural way an element $x_{S^1}\in \hat K^1(S^1)$
coming from the Poincar{\'e} bundle over $S^1\times S^1$.
Let us identify $S^1\cong \R/\Z$.
We consider the complex line bundle  $L:=(\R\times \R/\Z\times \C)/\Z$ over $\R/\Z\times \R/\Z$, where the $\Z$-action is given by $n(s,t,z)=(s+n,t,\exp(-2\pi i n t)z)$.
On $\R\times \R/\Z\times \C\rightarrow \R\times\R/\Z$
we have the $\Z$-equivariant connection
$\nabla:=d+2\pi i s dt$ with curvature $R^\nabla=2\pi ids\wedge dt$.
This connection descends to a connection $\nabla^L$ on $L$.
The unitary line bundle with connection $\bL:=(L,h^L,\nabla^L)$ gives a geometric family
$\cL$ over $\R/\Z\times\R/\Z$. 
It represents $v:=[\cL,0]\in \hat K^0( \R/\Z\times  \R/\Z)$. Note that $R(v)=1+ds\wedge dt$. We now consider the projection
$p\colon \R/\Z\times \R/Z\rightarrow \R/\Z$ on the second factor.
This fibre bundle has a natural smooth $\hat K$-orientation
$(g^{T^vp},T^hp,\tilde \nabla,0)$. 
The vertical metric and the horizontal distribution  come from the metric of $S^1$ and the product structure.
Moreover, $T^vp$ is trivialized by the $S^1$-action. Hence it has a preferred orientation. We take the bounding  $Spin$-structure on the fibres which induces the  $Spin^c$-structure and the connection $\tilde \nabla$.
\begin{ddd}\label{dddxs}
We define $x_{S^1}:=\hat p_! v\in\hat K^1(S^1)$.
\end{ddd}
\subsubsection{}

We have $R(x_{S^1})=dt$.
Let $t\in S^1$. Then we compute $t^*x_{S^1}\in\hat K^1(*)\cong \R/\Z$ (identification again as in Lemma \ref{lemma6777}).
Note that $0^*x_{S^1}$ is represented by the trivial line bundle over
$S^1$. Since we choose the bounding spin structure, the corresponding
Dirac operator is invertible. 
Its spectrum is symmetric and its
$\eta$-invariant vanishes (compare \ref{uiwfweqfewfwfef}). Therefore we have $0^*x_{S^1}=0$.
It now follows by the homotopy formula (or by an explicit computation
of $\eta$-invariants), that 
\begin{equation}\label{shiftt}
t^*x_{S^1}=-t\ .
\end{equation}

\subsubsection{}

Let $f\colon B\rightarrow S^1$ be given. Then we define
\begin{ddd}
$<f>:=f^*x_{S^1}\in \hat K^1(B)$.
\end{ddd}
Assume now that we have two such maps
$f,g\colon B\rightarrow S^1$. As an interesting illustration we  characterize
$$<f>\cup <g>\in \hat K^0(B)\ .$$
It suffices to consider the universal example $B=T^2=S^1\times S^1$.
We consider the projections $\pr_i\colon S^1\times S^1\rightarrow S^1$,
$i=1,2$. Let $x:=\hat\pr_1^*x_{S^1}$ and $y:=\hat\pr_2^*x_{S^1}$.
Then we must compute $x\cup y\in \hat K^0(T^2)$.
We identify $T^2=\R/\Z\times\R/\Z$ with coordinates $s,t$.
First note that
$R(x\cup y)=R(x)\cup R(y)=ds\wedge dt$.
Thus the class
$x\cup  y -v+1$ is flat, i.e. 
$$x\cup  y-v+1\in K^0_{flat}(T^2)\ .$$
In fact, since $K^0(T^2)$ is torsion-free, we have
$$K^0_{flat}(T^2)\cong H^{odd}(T^2)/\im(\ch_{dR})=\R^2/\Z^2\ .$$
 In order to determine this element we must
compute its holonomies along the circles
$S^1\times 0$ and $0\times S^1$.
The holonomy of $v$ along these circles is trivial.
Since $0^*x=0$ and $0^*y=0$ we see that $x\times y$
also has trivial  holonomies along these circles. Therefore
we conclude
\begin{prop}
$x\cup  y=v-1$ 
\end{prop}
We can now solve our original problem.
The two maps $f,g$ induce a map
$f\times g\colon B\rightarrow T^2$.
\begin{kor}
We have $<f>\cup  <g>=(f\times g)^*v-1$.
\end{kor}

\subsection{The product of $S^1$-valued maps  
and line-bundles}

\subsubsection{}\label{blcons}
Let $f\colon B\rightarrow S^1$ be a smooth map and
$\bL:=(L,\nabla^L,h^L)$ be a hermitean line bundle with 
connection over $B$. It gives rise to a geometric family
$\cL$ (see \ref{zerofibre}).
We consider the smooth $K$-theory classes
$<f>$ and $<\bL>:=[\cL,0]-1$. It is again interesting  to determine
the class $$<f>\cup <\bL>\in \hat K^1(B)\ .$$ 
An explicit answer is  only known in special cases.

First we compute the curvature:
$$R(<f>\cup<\bL>)=R(<f>)\wedge R(<\bL>)=  df \wedge (e^{c_1(\nabla^{L})}-1) \ ,$$
where $df:=f^*dt$ and $c_1(\nabla^{L}):=-\frac{1}{2\pi i} R^{\nabla^L}$.
  
\subsubsection{}

Note that the degree-one component of the odd form $R(<f>\cup<\bL>)$ vanishes.
Let now $q\colon \Sigma\rightarrow B$ be a smooth map from an oriented
closed  surface. Then
$R(q^*(<f>\cup <\bL>))=q^*R((<f>\cup <\bL>))=0$.
Therefore $$q^*(<f>\cup <\bL>)\in \hat K^1_{flat}(\Sigma)\cong
H^{ev}(\Sigma,\R)/\im(\ch)\cong \R/\Z\oplus \R/\Z\ ,$$
where the first component corresponds to $H^0(\Sigma,\R)$ and the
second to $H^2(\Sigma,\R)$.
In order to evaluate the first component we
restrict to a point. Since the restriction of $<\bL>$ to a point
vanishes, the first component of $q^*(<f>\cup <\bL>)$ vanishes.
Therefore it remains to determine the second component.

\subsubsection{}

Let us assume that $q^*L$ is trivial. We choose a
trivialization. Then we can define the transgression Chern form
$\tilde c_1(\nabla^{q^*L},\nabla^{triv})\in \Omega^1(\Sigma)$ such that
$d\tilde c_1(\nabla^{q^*L},\nabla^{triv})=q^* c_1(\nabla^L)$.
By the homotopy formula we have
$$q^*<\bL>=[\emptyset,-\tilde c_1(\nabla^{q^*L},\nabla^{triv})]\ .$$
In this special case we can compute
\begin{eqnarray*}q^*(<f>\cup <\bL>)&=&
q^*<f>\cup\:\: q^* <\bL>\\&=&<q^*f>\cup \:\:q^*<\bL>\\&=&
[\emptyset,q^*df\wedge \tilde
c_1(\nabla^{q^*L},\nabla^{triv})]\ .\end{eqnarray*}
We see that the second component is
$$\left[\int_{\Sigma} q^*df\wedge \tilde
c_1(\nabla^{q^*L},\nabla^{triv})\right]_{\R/\Z}\ .$$

We do not know a good answer in the general case where $q^*L$ is non-trivial.

\subsection{A  bi-invariant $\hat K^1$- class on $SU(2)$}

\subsubsection{}

Let $G$ be a group acting on the manifold $M$.
\begin{ddd}
A class  $x\in \hat K(M)$ is called invariant, if $g^*x=x$ for all $x\in G$.
\end{ddd}

\subsubsection{}
For example, the class $x_{S^1}\in \hat K^1(S^1)$ defined in \ref{dddxs} is not invariant under the action
$L_t$, $t\in S^1$, of $S^1$ on itself. Note that $R(x_{S^1})=dt$ is invariant. 
Therefore $L_t^*x_{S^1}-x_{S^1}\in \R/\Z$. In fact by (\ref{shiftt}) we have 
$$L_t^*x_{S^1}-x_{S^1}=-t\ .$$  Since $dt$ is the only invariant form with
integral one we see that the only way to produce an invariant smooth refinement of the generator
of $H^1(S^1,\Z)\cong \Z$ would be to perturb $x_{S^1}$ by a class $b\in H^0(S^1,\R/\Z)$. 
But $b$ is of course homotopy invariant, hence $L_t^*b=b$.
We conclude that the generator
of $H^1(S^1,\Z)$ (and also every non-trivial multiple) does not admit any invariant lift.

\subsubsection{}
The situation is different for simply-connected groups. Let us consider the following example.
The group $G:=SU(2)\times SU(2)$ acts on $SU(2)$ by
$(g_1,g_2)h:=g_1hg_2^{-1}$. Let $\vol_{SU(2)}\in \Omega^3(SU(2))$ denote the  normalized  volume form.
Furthermore we let $i\colon *\to SU(2)$ denote the embedding of the identity.
\begin{prop}
For $k\in \Z$ there exists a unique class $x_{SU(2)}(k)\in \hat K^1(SU(2))$ such that
$R(x_{SU(2)})=k\vol_{SU(2)}$ and $i^*x=0$.
This element is $SU(2)\times SU(2)$-invariant 
\end{prop}
\begin{proof}
Assume, that $x,y\in \hat K^1(SU(2))$ satisfy 
$R(x)=R(y)$.  Then we have
$x-y\in \hat K_{flat}^1(SU(2))\cong K_{flat}^1(S^3)\cong \R/\Z$.
Since  $i^*x=i^*y=0$ we have in fact that $x=y$. 
Therefore, if the class $x_{SU(2)}(k)$ exists, then it is unique. 

We show the existence of an invariant class in an abstract manner. Note that
$k\vol_{SU(2)}$ represents a class $\ch(Y)$ for some $Y\in
K^1(S^3)$. In terms of classifying maps, $Y$ for $k=1$ is given by the embedding
$SU(2)\rightarrow U(2)\rightarrow U(\infty)\cong K^1$.
We have the exact sequence
$$0\rightarrow \Omega^{ev}(SU(2))/\im(\ch_{dR})\stackrel{a}{\rightarrow} \hat
K^1(SU(2))\stackrel{I}{\rightarrow} K^1(SU(2))\rightarrow 0 \ .$$ 
Therefore we can choose any class $y\in \hat K^1(SU(2))$ such that
$I(y)=Y$.  
Then the continuous group cocycle 
$G\ni t\rightarrow c(t)=t^*y-y\in  \Omega^{ev}(SU(2))/\im(\ch_{dR})$
represents an element 
$[c]\in H_c^1(G,\Omega^{ev}(SU(2))/\im(\ch_{dR}))$.

We claim that this cohomology group is trivial.
Note that
$\Omega^{ev}(SU(2))/\im(\ch_{dR})\cong \Omega^0(SU(2))/\Z\oplus
\Omega^2(SU(2))/\im(d)$.
Since $ \Omega^2(SU(2))/\im(d)$ is a real topological vector space with a
continuous action  of the compact group $G$ we immediately
conclude that $H_c^1(G,\Omega^2(SU(2))/\im(d))=0$ by the usual
averaging argument.
We consider the exact sequence of $G$-spaces
$$0\rightarrow \Z\rightarrow \Omega^0(SU(2))\rightarrow
\Omega^0(SU(2))/\Z\rightarrow 0\ .$$
Since $G$ is simply-connected we see that taking continuous functions from
$G\times\dots \times G$ with values in these spaces, we obtain again
exact sequences of $\Z$-modules.
It follows that we have a long exact sequence in continuous
cohomology. The relevant part reads
$$H^1_c(G,\Z)\rightarrow H^1_c(G, \Omega^0(SU(2)))\rightarrow 
H^1_c(G,\Omega^0(SU(2))/\Z)\rightarrow H^2_c(G,\Z)\ .$$
Since $\Z$ is discrete and $G$ is connected we see that
$H^i_c(G,\Z)=0$ for $i\ge 1$. 
Therefore,
$$H^1_c(G, \Omega^0(SU(2)))\cong 
H^1_c(G,\Omega^0(SU(2))/\Z)\ .$$
But $\Omega^0(SU(2))$ is again a continuous representation
of $G$ on a real vector space so that
$H^1_c(G, \Omega^0(SU(2)))=0$. The claim follows.

We now can choose $w\in \Omega^{ev}(SU(2))/\im(\ch_{dR})$ such that
$t^*w-w=t^*y-y$ for all $t\in G$. We can further assume that
$i^*w=i^*y$ by adding a constant.
Then we set
$x_{SU(2)}(k)=y-w\in \hat K^1(SU(2))$. This element has the required
properties. \end{proof}

It is an interesting problem to write down an invariant cycle which represents  the class $x_{SU(2)}$.

\subsubsection{}

Note that
$x_{SU(2)}(k)=kx_{SU(2)}(1)$.
Let $\Sigma\subset SU(2)$ be an embedded oriented hypersurface.
Then $R (x_{SU(2)}(1))_{|\Sigma}=0$ so that
$(x_{SU(2)})_{|\Sigma}\in \hat K_{flat}^1(\Sigma)$. Since $x_{SU(2)}(1)$ evaluates
trivially on points we have in fact
$$(x_{SU(2)}(1))_{|\Sigma}\in \ker\left(\hat K_{flat}^1(\Sigma)\to \hat K^1_{flat}(*)\right)\cong \R/\Z\ .$$
This number can be determined by integration over $\Sigma$.
Formally, let $p\colon \Sigma\rightarrow \{*\}$ be the projection. If we choose some smooth $K$-orientation, 
then we can ask for $\hat p_!(x_{SU(2)}(1))_{|\Sigma}\in \hat K_{flat}^1(*)\cong
\R/\Z$. 
The hypersurface  $\Sigma$ decomposes  $SU(2)$ in two parts $SU(2)_{\Sigma}^\pm$. 
Let $SU(2)_{\Sigma}^+$ be the part such that $\partial
SU(2)_{\Sigma}^+$ has the orientation given by $\Sigma$.
We choose a $K$-orientation  $o$ of  the projection $q\colon SU(2)_{\Sigma}^+\to *$
which has a product structure such that $\sigma(o)=0$ and $\hA^c(o)=1$.
In order to get the latter equality we choose a $Spin^c$-structure coming from a spin structure. The smooth $K$-orientation of $q$   induces a smooth $K$-orientation of $p$.
Then $q\colon SU(2)_{\Sigma}^+\to *$  provides a zero-bordism of $\Sigma$, and
of $(x_{SU(2)}(1))_{|\Sigma}$. Therefore, we have by Proposition \ref{bordin} 
$$\hat p_!(x_{SU(2)}(1))_{|\Sigma} =\left[\emptyset, \int_{SU(2)_{\Sigma}^+}
  R(x_{SU(2)}(1))\right]=
-[\vol(SU(2)_{\Sigma}^+)]_{\R/\Z}\ ,$$
where $[\lambda]_{\R/\Z}$ denotes the class of $\lambda\in \R$. 
Note that the identification $\hat K_{flat}^1(*)\cong
\R/\Z$ is induced by $a\colon \R\cong \Omega^{odd}(*)/\im(d)\to K_{flat}^1(*)$ given by $\lambda\mapsto [\emptyset,-\lambda]$. This explains the minus sign in the second equality above.

\subsection{Invariant classes on homogeneous spaces}

\subsubsection{}

Some of the arguments from the $SU(2)$-case generalize.
Let $G$ be a compact connected and simply-connected Lie group and
$G/H$ be a homogenous space.

Given $Y\in K(G/H)$ we can find a lift $y\in\hat K(G/H)$.
We form the cocycle $G\ni g\mapsto c(g):=g^*y-y\in
\Omega(G/H)/\im(\ch_{dR})$. Since
$\Omega(G/H)/\im(\ch_{dR})$ is the quotient of a vector space by a
lattice and $G$ is connected and simply-connected we can use the
arguments as in the $SU(2)$-case in order to conclude that
$H^1_c(G,\Omega(G/H)/\im(\ch_{dR}))=0$. Therefore we can
choose the lift $y$ such that $g^*y=y$ for all $g\in G$. In
particular, $R(y)\in \Omega(G/H)$ is now an invariant form
representing $\ch(Y)$. Note that an invariant form is in general not determined by this condition.

\subsubsection{}

If we specialize to the case that $G/H$ is symmetric, then
invariant forms exactly represent the cohomology.
In this case we see that two choices of invariant lifts $y_0,y_1$ of
$Y$ have the same curvature so that $y_1-y_0\in \hat K_{flat}(G/H)$.
Since the $y_i$ also have the same index, we indeed have
$y_1-y_0\in H(G/H,\R)/\im(\ch_{dR})$. We have thus shown the following lemma.
\begin{lem}
Assume that $G/H$ is a symmetric space with $G$ connected and simply
connected. Then every $Y\in K(G/H)$ has an invariant lift
$y\in \hat K(G/H)$ which is uniquely determined up to
$ H(G/H,\R)/\im(\ch_{dR})$.
\end{lem}

\subsubsection{}
We can apply this in certain cases.
First we write $S^{2n+1}\cong Spin(2n+2)/Spin(2n+1)$, $n\ge 1$.
Note that $ K^{1}(S^{2n+1})\cong \Z$.
Since $H^{ev}(S^{2n+1},\R)/\im(\ch_{dR})=\R/\Z$
is concentrated in degree zero we have the following result.
\begin{kor} Let $n\ge 1$.
For each $k\in\Z$ there is a unique $x_{S^{2n+1}}(k)\in \hat K^{1}(S^{2n+1})$
which is invariant, has index $k\in \Z\cong  K^{1}(S^{2n+1})$, and
evaluates trivially on points.
\end{kor}

\subsubsection{}

In the even-dimensional case we write $S^{2n}\cong Spin(2n+1)/Spin(2n)$, $n\ge 1$.
Note that $ K^{0}(S^{2n})\cong \Z\oplus \Z$ and $H^{odd}(S^{2n},\R)/\im(\ch_{dR})=0$.
 \begin{kor}
For each $k\in\Z$ there is a unique $x_{S^{2n}}(k)\in \hat K^{0}(S^{2n})$
which is invariant and  has index $k\in \Z\cong \tilde  K^{0}(S^{2n})$, and evaluates trivially on points  
\end{kor}

\subsubsection{}
We write $\C\P^n:=SU(n+1)/S(U(1)\times U(n))$.
Then $H^{odd}(\C\P^n,\R)/\im(\ch_{dR})=0$. Therefore we conclude:
\begin{lem}
For each $Y\in K^0(\C\P^n)$ there is a unique $SU(n+1)$-invariant
class $y_{\C\P^n}(Y)\in \hat K^0(\C\P^n)$ such that $I(y_{\C\P^n}(Y))=Y$.
\end{lem}

\subsubsection{}

Let $G$ be a connected and simply-connected Lie group.
Let $T\subset G$ be a maximal torus.
Then we have a $G$-map $P\colon G/T\times T\rightarrow G$, $P([g],t):=gtg^{-1}$,
where $G$ acts on the left-hand side by $g([h],t):=([gh],t)$, and by
conjugation on the right-hand side.
Let $x\in \hat K^*(G)$ be an invariant element. It is an interesting question how
$P^*x$ looks like.

Let us consider the special case $G=SU(2)$ and $x_{SU(2)}=x_{SU(2)}(1)\in \hat K^1(SU(2))$.
In this case we have $T=S^1$ and  $G/T\cong \C \P^1$.
First we compute the curvature of
$P^*x_{SU(2)}$. For this we must compute
$P^*\vol_{SU(2)}$ which is given by Weyl's integration formula.
We have
$$P^*\vol_{SU(2)}=\vol_{\C \P^1}\wedge 4 \sin^2(2\pi t) dt\ .$$
There is a unique class $z\in\hat K^1(S^1)$ with curvature
$4 \sin^2(2\pi t)dt$ such that $0^*z=0$. Furthermore, there is a unique
class $<\bL>\in\hat K^0(\C \P^1)$ with curvature $\vol_{\C\P^1}$ which is
in fact the class $<\bL>$ considered in \ref{blcons} associated to the canonical line bundle $\bL$
on $\C\P^1$.

The product $<\bL>\cup z$ has now the same curvature as
$P^*x_{SU(2)}$.
We conclude that
$$P^*x_{SU(2)}-<\bL>\cup z\in H^{ev}(\C \P^1\times S^1,\R)/\im(\ch_{dR})\ .$$
Now note that 
\begin{eqnarray*}\lefteqn{H^{ev}(\C \P^1\times S^1,\R)/\im(\ch_{dR})}&&\\&\cong&
\left(H^0(\C \P^1,\R)\otimes H^0(S^1,\R)\oplus H^2(\C \P^1,\R)\otimes
H^0(S^1,\R)\right)/\im(\ch_{dR})\\&\cong& \R/\Z\oplus \R/\Z\ .\end{eqnarray*}
The first component can be determined by evaluating
the difference $P^*x_{SU(2)}-<\bL>\cup z$ at a point. Since $x_{SU(2)}$ is trivial on points,
this first component vanishes. The second component can be determined
by evaluating $P^*x_{SU(2)}-<\bL>\cup z$ at $\C \P^1\times \{0\}$.
Note that $P_{\C \P^1\times \{0\}}^*x_{SU(2)}=0$, since
$P_{|\C\P^1\times \{0\}}$ is constant. Furthermore, $0^*z=0$ implies
that $<\bL>\cup z_{|\C \P^1\times \{0\}}=0$.
Thus we have shown  (using $S^2\cong \C\P^1$):
\begin{lem}
$P^*x_{SU(2)}=x_{S^2}(1)\cup z$
\end{lem}

\subsection{Bordism}

\subsubsection{}

A zero bordism of a geometric family $\cE$ over $B$ is a
geometric family $\cW$ over $B$ with boundary such that $\cE=\partial
\cW$. The notion of a geometric family with boundary is explained in
\cite{math.DG/0201112}. It is important to note that in our set-up a geometric family with boundary  always has a product structure.

\begin{prop}\label{kop1}
If $\cE$ admits a zero bordism $\cW$, then
in $\hat K^*(B)$ we have the identity
$$[\cE,0]=[\emptyset,\Omega(\cW)] .$$
\end{prop} 
\begin{proof}
Since $\cE$ admits a zero bordism we have $\ind(\cE)=0$ so that $\cE$
admits a taming $\cE_t$. This taming induces a boundary taming 
$\cW_{bt}$. The obstruction against extending the boundary taming to a taming
of $\cW$ is $\ind(\cW_{bt})\in K(B)$ \cite[Lemma 2.2.6]{math.DG/0201112}.

Let us assume for simplicity that $\cE$ is not zero-dimensional. Otherwise we
may have to stabilize in the following assertion. Using \ref{pap3} we can adjust the taming
$\cE_t$ such that $\ind(\cW_{bt})=0$.
At this point we employ a version of the relative index theorem \cite{Bun95}
\begin{equation}\label{wiufhewfwefewfewfwfwefwef}
\ind(\cW_{bt^\prime})=\ind(\cW_{bt})+\ind((\cE\times [0,1])_{bt})\ ,
\end{equation}
where $\cE_t$ and $\cE_{t^\prime}$ define the boundary taming $(\cE\times [0,1])_{bt}$.

If $\ind(\cW_{bt})=0$,
then we can extend the boundary taming $\cW_{bt}$ to
a taming $\cW_t$. We now apply the identity
\cite[Thm. 2.2.13]{math.DG/0201112}:
$$\Omega(\cW)=d\eta(\cW_t)-\eta(\cE_t)\ .$$
Note that this equality is more precise than needed since it holds on the level of forms without factoring by $\im(d)$.
We see that $(\cE,0)$ is paired with $(\emptyset,\Omega(\cW))$. This implies the assertion.
\end{proof}

\subsubsection{}

Let $p\colon W\to B$ be a proper submersion from a manifold with boundary $W$ which restricts to a submersion
$q:=p_{|\partial W}\colon V:=\partial W\to B$. We assume that $p$ has a topological $K$-orientation and a smooth $K$-orientation represented by $o_p$ which refines the topological $K$-orientation. 
We assume that the geometric data of $o_p$ has a product structure near $V$
(see \cite[Section 2.1]{math.DG/0201112} for a detailed discussion of such
product structures). 
Recall $o_p=(g^{T^vp},T^hp,\tilde \nabla_p,\sigma_p)$. By the assumption of a product structure
we have a quadruple $(g^{T^vq},T^hq,\tilde \nabla_q,\sigma_q)$ and an
isomorphism of a neighbourhood of $p_{|\partial W}\colon \partial W\to B$ with
the bundle
$\cE\times [0,1)\stackrel{\pr_\cE}{\to} \cE\stackrel{p}{\to} B$ such that the geometric
data are related as follows.
\begin{enumerate}
\item $T^vp_{|\cE\times [0,1)}\cong \pr_\cE^*T^vq\oplus \pr^*_{[0,1)}T[0,1)$ and 
$g^{T^vp}_{|\cE\times [0,1)}=\pr_\cE^*g^{T^vq}+ \pr_{[0,1)}^*dr^2$, where $r\in [0,1)$ is the coordinate.
\item $T^hp_{|\cE\times [0,1)}=\pr_\cE^* T^hq$.
\item $(\sigma_p)_{|\cE\times [0,1)}=\pr_\cE^*\sigma_q$.
\item The $Spin^c$-structure on $T^vq$  and the canonical $Spin^c$-structure
on $T[0,1)$ induce a $Spin^c$-structure on the vertical bundle $T^v\cong \pr_\cE T^v\cE\oplus \pr_{[0,1)}^*T[0,1)$ of  $\cE\times [0,1)$ in a canonical way so that the associated spinor bundle is
$S(T^v)=\pr_\cE^*S^c(T^vq)$ or $\pr_\cE^* S^c(T^vq)\otimes \C^2$ depending on the dimension of $T^vq$. In particular, the connection $\tilde \nabla_q$ gives rise to a connection $\tilde \nabla_{prod}$. The product structure
identifies the restricted $Spin^c$-structure of
$T^vp_{|\cE\times [0,1)}$ with this product $Spin^c$-structure such that $\tilde \nabla_{|\cE\times [0,1)}$ becomes   $\tilde \nabla_{prod}$.
\end{enumerate}
{}From this description we deduce that
$$\hA^c(\tilde \nabla)_{|\cE\times [0,1)}=\pr_\cE^* \hA^c(\tilde \nabla_q)\ ,\quad \hA^c(o_p)_{|\cE\times [0,1)}=\pr_\cE^* \hA^c(o_q)\ .$$
It is now  easy to see that the restriction of representatives (with product structure)
preserves equivalence and gives a well-defined restriction of smooth $K$-orientations.
We have the following version of bordism invariance of the push-forward in smooth $K$-theory.
\begin{prop}\label{bordin} For $y\in\hat K(W)$ we set $x:=y_{|V}\in \hat K(V)$. Then 
we have
$$\hat q_!(x)=[\emptyset, p^o_!R(y)]\ . $$
\end{prop}
\begin{proof}
Let $y=[\cE,\rho]$.
We compute using (\ref{eq300}), Proposition \ref{kop1}, Stokes' theorem, Definition \ref{def313}, and the adiabatic limit $\lambda\to 0$
at the marked equality
\begin{eqnarray*}
\hat q_!(x)&=&[q^\lambda_!\cE_{|V},\int_{V/B}\hA^c(o_q)\wedge \rho +\tilde
\Omega(\lambda,\cE_{|V})+\int_{V/B} \sigma(o_q)\wedge R(x)]\\
&=&[\emptyset, \Omega(p^\lambda_!\cE)+\int_{V/B}\hA^c(o_q)\wedge \rho +\tilde
\Omega(\lambda,\cE_{|V})+\int_{V/B} \sigma(o_q)\wedge R(x)]\\
&\stackrel{!}{=}&[\emptyset,\int_{W/B} \left(\hA^c(o_p)\wedge \Omega(\cE) -
\hA^c(o_p)\wedge d\rho - d\sigma(o_p) \wedge R(y)
\right)]\\
&=&[\emptyset,\int_{W/B} (\hA^c(o_p)-d\sigma(o_p))\wedge R(y)]\\ 
&=&[\emptyset,p^o_! R(y)]
\end{eqnarray*}
\end{proof}

\subsection{$\Z/k\Z$-invariants}

\subsubsection{}

Here we associate to a  family of $\Z/k\Z$-manifolds over $B$ a
class in $\hat K_{flat}(B)$. 
\begin{ddd} A geometric family of  $\Z/k\Z$-manifolds is a
triple $(\cW,\cE,\phi)$, where
$\cW$ is a geometric family with boundary, $\cE$ is a geometric family without boundary, and
$\phi\colon \partial\cW\stackrel{\sim}{\to} k\cE$ is an isomorphism of the boundary of $\cW$ with $k$ copies of $\cE$. \end{ddd}
 We define $u(\cW,\cE,\phi):=[\cE,-\frac{1}{k}\Omega(\cW)]\in\hat K(B)$.
\begin{lem}
We have
$u(\cW,\cE,\phi)\in \hat K_{flat}(B)$. This class
is a $k$-torsion class.
It only depends on the underlying differential-topological data.
\end{lem}
\begin{proof}
We first compute by \ref{kop1}
\begin{eqnarray*}
ku(\cW,\cE,\phi)&=&
k[\cE,-\frac{1}{k}\Omega(\cW)]\\&=&[k\cE,-\Omega(\cW)]\\
&=&[ \emptyset,0]\\
&=&0
\end{eqnarray*}
This implies that $R(u(\cW,\cE,\phi))=0$ so
that $u(\cW,\cE,\phi)\in \hat K_{flat}(B)$.
Independence of the geometric data is now shown by a homotopy argument.
\end{proof}

\subsubsection{}
We now explain the relation of this construction to the
$\Z/k\Z$-index of Freed-Melrose \cite{MR1144425}.
\begin{lem}
Let $B=*$ and $\dim(\cW)$ be even.
Then $u(\cW,\cE,\phi)\in \hat K_{flat}^1(*)\cong \R/\Z$.
Let $i_k\colon \Z/k\Z\rightarrow \R/\Z$ the embedding which sends $1+k\Z$ to
$\frac{1}{k}$.
Then
$$i_k(\ind_a(\bar W))=u(\cW,\cE,\phi)\ , $$
where  $i_k(\ind_a(\bar W))\in\Z/k\Z$ is the index of the
$\Z/k\Z$-manifold $\bar W$ (the notation of \cite{MR1144425}).
\end{lem}
\begin{proof}
We recall the definition of $\ind_a(\bar W)$. In our language is can
be stated as follows. Since $\ind(\cE)=0$ we can choose a taming $\cE_t$.
We let $k$ copies of $\cE_t$  induce the  boundary taming $\cW_{bt}$. We have
 $$\ind_a(\bar W)=\ind(\cW_{bt})+k\Z\ .$$
In fact it is easy to see that a change of the taming $\cE_t$
leads to change of the index $\ind(\cW_{bt})$ by a multiple of $k$.
We can now prove the Lemma using  \cite[Thm. 2.2.18]{math.DG/0201112}.
\begin{eqnarray*}
u(\cW,\cE,\phi)&=&[\cE,-\frac{1}{k}\Omega(\cW)]\\
&{=}&[\emptyset,-\eta(\cE_t)-\frac{1}{k}\Omega(\cW)]\\
&=&[\emptyset,-\frac{1}{k}\ind(\cW_{bt})]\\
&=& a\left(\frac{1}{k}\ind(\cW_{bt})\right)\\
&=&i_k(\ind_a(\bar W)) \in \R/\Z.
\end{eqnarray*}
\end{proof}

\subsection{$Spin^c$-bordism invariants}\label{rhoho}

\subsubsection{}
Let $\pi$ be a finite group.
We construct a transformation
$$ \phi\colon \Omega^{Spin^c}(BU(n)\times  B\pi)\rightarrow \hat K_{flat}(*)\ .$$
Let $f\colon M\rightarrow BU(n)\times B\pi$ represent $[M,f]\in  \Omega^{Spin^c}(BU(n)\times B
\pi)$. This map determines a covering $p\colon \tilde M\rightarrow M$ and an
$n$-dimensional complex vector bundle $V\rightarrow M$.
We choose a Riemannian metric  $g^{TM}$ and a $Spin^c$-extension $\tilde \nabla$ of the Levi-Civita connection $\nabla^{TM}$. 
These structures determine a smooth $K$-orientation of $t\colon M\to *$. We further fix a metric $h^V$ and a connection $\nabla^V$ in order to define a geometric bundle $\bV:=(V,h^V,\nabla^V)$ and the associated geometric family $\cV$ (see \ref{zerofibre}).
The pull-back of $g^{TM}$ and $\tilde \nabla$ via $\tilde M\to M$ fixes a 
smooth $K$-orientation of $\tilde t\colon \tilde M\to *$. 

We define the geometric families $\cM:=t_!\cV$ and $\tilde \cM:=\tilde t_!(p^*\cV)$ over $*$.
Then we set
$$\phi([M,f]):=[\tilde \cM\sqcup_* |\pi|\cM^{op},0]\in \hat K_{flat}(*)\ .$$
By a homotopy argument we see that this class is independent of the
choice of geometry.
We now argue that it only depends on the bordism class
of $[M,f]$.

The construction is additive.
Let now $[M,f]$ be zero-bordant by $[W,F]$.
Then we have a zero bordism $\tilde W$ of $\tilde M$
over $W$. Note that the bundles also extend over the bordism.
The local index form of
$\tilde \cW\sqcup_B |\pi|\cW$ vanishes. We conclude by \ref{kop1}, that
$[\tilde \cM\sqcup_B |\pi|\cdot \cM^{op},0]=0$.

In this construction we can replace $E\pi\rightarrow B\pi$ by any 
finite covering.
\subsubsection{}

This construction allows the following modification. Let
$\rho\in Rep(\pi)_0$ be a virtual zero-dimensional representation of
$\pi$.
It defines a flat vector bundle $F_\rho\rightarrow B\pi$.
To $[M,f]$ we associate the geometric family
$\cM_\rho:=t_!(\cL)$, where $\cL$ is the geometric family associated to the geometric bundle $\bV\otimes (\pr_2\circ f)^*F_\rho$. 
We define
$$\phi_\rho\colon \Omega^{Spin^c}_*(BU(n)\times B\pi)\rightarrow \hat K_{flat}(*)$$
such that $\phi_{\rho}[M,f]:=[\cM_\rho,0]$.
Here we need  not to assume that $\pi$ is finite.
This is the construction of $\rho$-invariants in the smooth $K$-theory
picture.

The first construction is a special case of the second with the
representation $\rho=\C(\pi)\oplus (\C^{|\pi|})^{op}$.
\subsubsection{}
We now discuss a parametrized version. Let $B$ be some compact manifold and $X$ be some topological space.
Then we can define the parametrized bordism group
$\Omega^{Spin^c}_*(X/B)$. Its cycles are pairs $(p\colon W\to B,f\colon W\to X)$ of a proper topologically $K$-oriented submersion
$p$ and a continuous map $f$. The bordism relation is defined correspondingly. 

There is a natural transformation
$$\phi\colon \Omega^{Spin^c}_*((BU(n)\times B\pi)/B)\rightarrow \hat K_{flat}^*(B)\ .$$
It associates to
$x=(p\colon W\to B,f\colon W\to BU(n)\times B\pi)$ the class $[\tilde
\cW\sqcup_B |\pi|\cdot \cW^{op},0]$. In this formula
$p\colon \tilde W\to W$ is again the $\pi$-covering classified by $\pr_2\circ f$. We define the geometric family $\cW$
using some choice of geometric structures and the twisting bundle $V$, where $V$ is classified by the first component of $f$. The family $\tilde \cW$  is obtained from $\tilde W$ and $p^*V$ using the lifted geometric structures. 
Again, the class $\phi(x)$ is flat and independent of the choices of geometry. Using \ref{kop1} one checks that  
 $\phi$ passes through the bordism relation.

Again there is the following modification.
For $\rho\in\Rep(\pi)_0$ we can define
$$\phi_\rho\colon \Omega^{Spin^c}_*((BU(n)\times B\pi)/B)\rightarrow \hat K_{flat}^*(B)\ .$$
It associates to $x=(p\colon W\to B,f\colon W\to BU(n)\times B\pi)$ the class $[\cW_\rho]$ of the geometric manifold $\cW$
with twisting bundle $V\otimes (\pr_2\circ f)^*F_\rho$.
These classes are $K$-theoretic  higher $\rho$-invariants. It seems promising
to use this picture to draw geometric consequences using these invariants.

\subsection{The $e$-invariant}

\subsubsection{}

A framed $n$-manifold $M$ is a manifold with a trivialization $TM\cong M\times \R^n$.
More general, a bundle of framed $n$-manifolds over $B$  is a fibre bundle $\pi\colon E\to B$ with a trivialization
$T^v\pi\cong E\times \R^n$.

\begin{prop}
A bundle of framed $n$-manifolds $\pi\colon E\to B$ has a canonical smooth $K$-orientation which only depends on the homotopy class of the framing.
\end{prop}
\begin{proof}
The framing $T^v\pi\cong E\times \R^n$ induces a vertical Riemannian metric $g^{T^v\pi}$
and an isomorphism $SO(T^v\pi)\cong E\times SO(n)$. Hence we get an induced vertical orientation
and a $Spin$-structure which determines a $Spin^c$-structure, and thus  a $K$-orientation of $\pi$. We choose a horizontal distribution $T^h\pi$ which gives rise to a connection $\nabla^{T^v\pi}$.
Since our $Spin^c$-structure comes from a $Spin$-structure, this connection extends naturally to a $Spin^c$-connection $\tilde \nabla$ of trivial central curvature.

The trivial connection $\nabla^{triv}$ on $T^v\pi$ induced by the framing also lifts naturally to the trivial $Spin^c$-connection $\tilde \nabla^{triv}$.
The quadruple
$$o:=(g^{T^v\pi},T^h\pi,\tilde \nabla,\tilde \hA^c(\tilde \nabla,\tilde \nabla^{triv}))$$
defines a smooth $K$-orientation of $\pi$ which refines the given underlying  topological $K$-orientation.

We claim that this orientation is independent of the choice of the vertical distribution $T^h\pi$.  Indeed, if $T^h\pi$ is a second horizontal distribution with associated $Spin^c$-connection $\tilde \nabla^\prime$, then we set $$o^\prime:=(g^{T^v\pi},T^h\pi^\prime,\tilde \nabla^\prime,\hA^c(\tilde \nabla^\prime,\tilde \nabla^{triv}))\ .$$
Since
$$\tilde \hA^c(\tilde \nabla^\prime,\tilde \nabla^{triv})-\tilde \hA^c(\tilde \nabla,\tilde \nabla^{triv})=\tilde \hA^c(\tilde \nabla^\prime,\tilde \nabla)$$
we have $o\sim o^\prime$ in view of the Definition
\ref{pap201}.

Let us now consider a second  framing of $T^v\pi$ which is homotopic to the first. In induces a second trivial connection
 $\tilde\nabla^{ \prime triv}$ and a metric $g^{\prime T^v\pi}$.
We therefore get a connection $\tilde \nabla^\prime$ and 
 and a second
representative of a smooth $K$-orientation
$o^\prime:=(g^{\prime T^v\pi},T^h\pi,\tilde \nabla^\prime,\tilde \hA^c(\tilde \nabla^\prime,\tilde \nabla^{\prime triv}))$. In fact, the homotopy between the framings provides a connection
$\tilde \nabla^{h,triv}$ on $I\times E$. Since this connection is flat we see that
$\tilde \hA^c(\tilde\nabla^{ \prime triv},\tilde\nabla^{triv})=0$.
From
$$\tilde \hA^c(\tilde \nabla^\prime,\tilde \nabla^{\prime triv})=
\tilde \hA^c(\tilde \nabla^\prime,\tilde \nabla)+ \tilde \hA^c(\tilde \nabla,\tilde \nabla^{triv})+ \tilde \hA^c(\tilde \nabla^{triv},\tilde \nabla^{\prime triv})$$
we get
$$\tilde \hA^c(\tilde \nabla^\prime,\tilde \nabla^{\prime triv})-\tilde \hA^c(\tilde \nabla,\tilde \nabla^{ triv})=\tilde \hA^c(\tilde \nabla^\prime,\tilde \nabla)$$
and thus
$o\sim o^\prime$.
\end{proof}

Since $\tilde \nabla^{triv}$ is flat we have
$$\hA^c(o)-d\sigma(o)=\hA(\tilde \nabla)-d\tilde \hA(\tilde\nabla,\tilde \nabla^{triv})=1\ .$$   
Assume that the fibre dimension $n$ satisfies $n\ge 1$.
According to Lemma  \ref{lem24} the curvature of $\hat \pi_!(1)$ is given by
$$R(\hat\pi_!(1))=\int_{E/B}(\hA^c(o)-d\sigma(o))\wedge 1=\int_{E/B}1\wedge 1=0$$ 

\begin{ddd}
If $\pi\colon E\to B$ is a bundle of framed manifolds of fibre dimension $n\ge 1$, then we define a differential topological invariant
$$e(E\to B):=-\hat\pi_!(1)\in \hat K^{-n}_{flat}(B)\ .$$
\end{ddd}

In the following 
we will explain in some detail that this is a higher generalization of the Adams $e$-invariant. The stable homotopy groups of the sphere
$\pi_n:=\pi_n^s(S^0)$ have a decreasing  filtration $$\dots\subseteq \pi_n^2\subseteq \pi_n^1\subseteq \pi_n^0=\pi_n$$ related to the MSpin-based Adams Novikov spectral sequence.
The $e$-invariant is a homomorphism
$$e\colon \pi_{4n-1}^1/\pi_{4n-1}^2\to \R/\Z\ .$$
A closed framed $4n-1$-dimensional manifold $M$ represents a class
$[M]\in \pi_{4n-1}$ under the Pontrjagin-Thom identification of framed bordism 
with stable homotopy. In the indicated dimension $\pi_{4n-1}=\pi_{4n-1}^1$ so that $[M]$ is actually a boundary of a compact $4n$-dimensional $Spin$-manifold $N$.
As explained in \cite{MR0397797} (see also \cite{MR1660325})
the $e$-invariant $e[M]$ can be calculated as follows. One chooses a connection $\nabla^{TN}$ on $TN$ 
which restricts to the trivial connection $\nabla^{triv}$ on $TM$ given by the framing.
Then
$$e([M])=\left[\int_N \hA(\nabla)\right]_{\R/\Z}\ .$$ 
We now consider $q\colon M\to *$ as a bundle of framed manifolds over the point and identify
$\R/\Z\stackrel{\sim}{\to }\hat K^{-4n+1}_{flat}(*)$ by 
$[u]\mapsto a(u)= [\emptyset,-u]$, $u\in \R$.
\begin {lem}
Under these identifications we have
$e(M\to *)=e([M])$.
\end {lem}
\begin{proof}
We choose a metric $g^{TM}$ on $M$ which induces the representative
$$o:=(g^{TM},0,\tilde \nabla,\tilde \hA^c(\tilde \nabla,\nabla^{triv}))$$ of the smooth $K$-orientation on $q$.  The $Spin$-structure of $N$ induces a $Spin^c$-structure.
We choose a Riemannian metric $g^{TN}$ on $N$ with a product structure near the boundary  which extends 
$g^{TM}$ and induces the $Spin$- and $Spin^c$-connections  $\nabla^N$ and $\tilde \nabla^{N}$.
Note that
$\tilde \hA^c(\tilde \nabla^{N},\tilde \nabla^{TN})$ extends
$\tilde  \hA^c(\tilde \nabla,\tilde \nabla^{triv})$.
Therefore
$o^N:=(g^{TN},0,\tilde \nabla^N,\tilde \hA^c(\tilde \nabla^{N},\tilde \nabla^{TN}))$ represents a smooth
$K$-orientation of $p\colon N\to *$ which extends the orientation $o$ of
$q\colon M\to *$.
We can now apply the bordism formula Proposition \ref{bordin} in the marked step
and get
\begin{eqnarray*}
e(M\to *)&=&-\hat q_!(1)\\&\stackrel{!}{=}&a(p_!(R(1)))\\&=&
\left[\int_{N/*} (\hA^c(o^N)-d\sigma(o^N))\wedge 1\right]_{\R/\Z}\\
&=&\left[\int_{N/*} \hA^c(\tilde \nabla^N)-d\tilde \hA(\tilde \nabla^{N},\tilde \nabla^{TN})\right]_{\R/\Z}\\
&=&\left[\int_{N/*}  \hA^c(\tilde \nabla^{TN}) \right]_{\R/\Z}\\
&=&\left[\int_{N/*}  \hA( \nabla^{TN}) \right]_{\R/\Z}\\
&=&e([M])\ .
\end{eqnarray*}
\end{proof} 
Using the method of  Subsection \ref{shskajd} or the APS index theorem 
it is now easy to reproduce the result of \cite{MR0397797}
$$e([M])=\left[\eta^0(M)-\int_{M} \hA(\tilde \nabla,\tilde \nabla^{triv})\right]_{\R/\Z}\ .$$

\section{The Chern character and a smooth Grothendieck-Riemann-Roch theorem}\label{sqsjhss}

\subsection{Smooth rational cohomology}
%
%
%

\subsubsection{}\label{ratdel}

Let $Z_{k-1}(B)$ be the group of smooth singular cycles on $B$. 
The picture of $\hat H(B,\Q)$ as Cheeger-Simons differential characters
$$\hat H^k(B,\Q)\subset \Hom(Z_{k-1}(B),\R/\Q)$$
is most appropriate to define the integration map. By definition  (see \cite{MR827262})
a homomorphism $\phi \in \Hom(Z_{k-1}(B),\R/\Q)$ is a differential character
if and only if there exists a form $R(\phi)\in \Omega_{d=0}^k(B)$ such that
\begin{equation}\label{uduiwqdqwdwqdqdwqdwdc}
\phi(\partial c)=\left[\int_c R(\phi)\right]_{\R/\Q}
\end{equation}
for all smooth $k$-chains $c\in C_k(B)$.
It is shown in \cite{MR827262}  that $R(\phi)$ is uniquely determined by $\phi$. In fact, the map
$R\colon \hat H^k(B,\Q) \to  \Omega_{d=0}^k(B)$ is the curvature transformation in the sense of Definition \ref{ddd556}.

Assume that $T$ is a closed oriented manifold of dimension $n$ with a triangulation. Then we have a map
$\tau\colon Z^{k-1}(B)\to Z^{k-1+n}(T\times B)$. If
$\sigma\colon \Delta^{k-1}\to B$ is a smooth singular simplex, then the triangulation of 
$T\times \Delta^{k-1}$ gives rise to a $k-1+n$ chain $\tau(\sigma)\colon =\id\times \sigma\colon T\times \Delta\to T\times B$. The integration
$$(\hat\pr_2)_!\colon \hat H(T\times B,\Q)\to \hat H(B,\Q)$$
is  now induced by
$$\tau^*\colon \Hom(Z^{k-1+n}(T\times B),\R/\Q)\to \Hom(Z^{k-1}(B),\R/\Q)\ .$$
Alternative definitions of the integration (for proper oriented submersions)
are given in \cite{MR2192936}, \cite{MR1772521}. Another construction of the
integration has been given in \cite{MR2179587}, where also a projection
formula (the analog of \ref{projcl} for smooth cohomology) is proved. This
picture is used in \cite{Koehler} in particular to establish functoriality.

We will also need the following bordism formula which we prove using yet
another characterization of the push-forward. 
We consider a proper oriented submersion
$q\colon W\to B$ such that $\dim(T^vq)=n$.
Let $x\in \hat H^r(W,\Q)$ and $f\colon \Sigma\to B$ be a smooth map from a closed oriented manifold
of dimension $r-n-1$. We get a pull-back diagram
$$
\begin{CD}
  U @>F>> W\\
  @VVV @VVqV\\
  \Sigma @>f>> B
\end{CD}
.$$
The orientations of $\Sigma$ and $T^vq$ induce an orientation of $U$.
Note that $f^*\hat q_!(x)$ and $F^*x$ are flat classes for dimension reasons.
Therefore
$F^*x\in H^{r-1}(U,\R/\Q)$ and $f^*\hat q_!(x)\in H^{r-n-1}(\Sigma,\R/\Q)$.
The compatibility of the push-forward with cartesian diagrams implies
the following relation in $\R/\Q$:
$$<f^*\hat q_!(x),[\Sigma]>=<F^*x,[U]>\ .$$
If we let $f\colon \Sigma\to B$ vary, then these numbers completely
characterize the push-forward $\hat p_!(x)\in \hat H^{r-n}(B,\Q)$. We will use this
fact in the argument below. 

\subsubsection{}

Let now $p\colon V\to B$ be a proper oriented submersion
from a manifold with boundary such that $\partial V\cong W$ and $p_{|W}=q$.
Assume that $x\in   \hat  H(V,\Q)$.
\begin{lem}\label{bbetwzew}

In $\hat H(B,\Q)$ we have the equality
$$\hat q_!(x_{|W})=-a\left(\int_{V/B} R(x)\right)\ .$$
\end{lem}
\begin{proof}
Assume that $x\in \hat H^r(V,\Q)$.
Let $f\colon \Sigma\to B$ be as above and form the cartesian diagram
$$
\begin{CD}
  Z @>z>> V\\
  @VVV @VVpV\\
  \Sigma @>f>> B.
\end{CD}
$$
The oriented manifold $Z$ has the boundary $\partial Z\cong U$.
Using (\ref{uduiwqdqwdwqdqdwqdwdc}) at the marked equality we calculate
\begin{eqnarray*}
<f^*\hat q_!(x_{|W}),[\Sigma]>&=&<F^*x_{|W},[U]>\\
&=&<(z^*x)_{|U},[U]>\\
&\stackrel{!}{=}&\left[\int_{Z} R(z^*x)\right]_{\R/\Q}\\
&=&\left[\int_\Sigma\int_{Z/\Sigma} R(z^*x)\right]_{\R/\Q}\\
&=&\left[\int_\Sigma f^*\int_{V/B} R(x)\right]_{\R/\Q}\\
&=&-<f^*a\left(\int_{V/B} R(x)\right),[\Sigma]>\ .
\end{eqnarray*}
This implies the assertion. \end{proof}

 \subsection{Construction of the Chern character }

\subsubsection{}\label{cqdef}

We start by recalling the classical smooth characteristic classes of
Cheeger-Simons. 
A complex vector bundle $V\to B$ has Chern classes $c_i\in H^{2i}(B,\Z)$, $i\ge 1$.
If we add the geometric data of a hermitean metric and a metric connection, then we get the geometric bundle
$\bV=(V,h^V,\nabla^V)$. In  \cite{MR827262} the Chern classes have been refined to  smooth integral cohomology-valued Chern classes
$$\hat c_i(\bV)\in \hat H^{2i}(B,\Z)$$
(see \ref{cher} for an introduction to smooth ordinary cohomology).
  In particular, the class
$\hat c_1( \bV)\in \hat H^2(B,\Z)$ classifies
isomorphism classes of hermitean line bundles with connection.  

The embedding
$\Z\hookrightarrow \Q$ induces a natural map
$\hat H(B,\Z)\to \hat H(B,\Q)$, and we let $\hat c_\Q(\bV)\in  \hat H^2(B,\Q)$ denote the image of
$\hat c_1(\bV)\in \hat H^2(B,\Z)$ under this map.


\subsubsection{}

The smooth Chern character $\hat \ch$ which we will construct is a natural transformation 
$$\hat \ch\colon  \hat K(B)\rightarrow \hat H(B,\Q)$$
of smooth cohomology theories. In particular, this means that the following diagrams commute (compare Definition \ref{natdef12})
\begin{equation}\label{eq5001}
\xymatrix{\Omega(B)/\im(d)\ar[r]^a\ar@{=}[d]&\hat K(B)\ar[r]^I\ar[d]^{\hat \ch}&K(B)\ar[d]^{\ch}\\\Omega(B)/\im(d)\ar[r]^a&\hat H(B,\Q)\ar[r]^I&H(B,\Q)}\ ,\quad \xymatrix{\hat K(B)\ar[r]^R\ar[d]^{\hat \ch}&\Omega_{d=0}(B)\ar@{=}[d]\\\hat H(B,\Q)\ar[r]^R&\Omega_{d=0}(B)}\ .
\end{equation}

In addition we require that the even and odd Chern characters are related by suspension, which in the smooth case amounts to
the commutativity of the following diagram
\begin{equation}\label{eq5000}
\xymatrix{\hat K^0(S^1\times B)\ar[d]^{(\hat\pr_2)_!}\ar[r]^{\hat \ch}&\hat H^{ev}(S^1\times B,\Q)\ar[d]^{(\hat\pr_2)_!}\\\hat K^1(B)\ar[r]^{\hat \ch}&\hat H^{odd}(B,\Q)}\ .
\end{equation}
The smooth $K$-orientation of $\pr_2\colon S^1\times B\to B$ is as in \ref{pullpush1}.

\begin{theorem}\label{mmmain}
There exists a unique natural transformation $\hat \ch\colon \hat K(B)\to \hat H(B,\Q)$
such that (\ref{eq5001}) and (\ref{eq5000}) commute.
\end{theorem}
Note that naturality means that $\hat \ch\circ f^*=f^*\circ \hat \ch$ for every smooth map $f\colon B^\prime\to B$.
The proof of this theorem occupies the remainder of the present subsection.

\subsubsection{}

\begin{prop}\label{unique}
If the smooth Chern character $\hat \ch$ exists, then  it is unique.
\end{prop}
\begin{proof}
Assume that
$\hat \ch$ and $\hat \ch^\prime$ are two smooth Chern characters.
Consider the difference $\Delta:=\hat \ch-\hat \ch^\prime$.
It follows from the diagrams above that $\Delta$ factors through an odd
natural transformation $$\bar\Delta \colon K(B)\to H(B,\R/\Q)\ .$$

Indeed, the left   diagram of  (\ref{eq5001}) gives a factorization
$$K(B)\to (\im\colon \Omega(B)/\im(d)\to \hat H(B,\Q))\ ,$$ and the right square in (\ref{eq5001}) refines it to
$\bar \Delta$.

\subsubsection{}
We now use the following topological fact. Let $P$ be a space of the homotopy type of a countable $CW$-complex. It represents a contravariant set-valued functor $W\mapsto P(W):=[W,P]$ on the category of compact manifolds.
We further consider some abelian group $V$.

\begin{lem}\label{hteh}
A natural transformation
of functors $N\colon P(B)\to H^j(B,V)$ on the category of compact manifolds is necessarily induced by a class $N\in H^{j}(P,V)$. 
\end{lem}
\begin{proof}
There exists a countable directed diagram $\cM$ of compact manifolds such that
$\hocolim \cM\cong P$ in the homotopy category.
Hence we have a short exact sequence
$$0\to \lim^1 H(\cM,V)\to H(P,V)\to \lim H(\cM,V)\to 0\ .$$
If $x\in P(P)$ is the tautological class, then the pull-back of $N(x)$ to the system
$\cM$ gives an element in  $\lim H(\cM,V)$. A preimage in $H(P,V)$ induces the
natural transformation.
\end{proof} 

In our application, $P=\Z\times BU$, and  the  relevant cohomology 
$H^{odd}(\Z\times BU,\R/\Q)$ is trivial. Therefore  $\bar \Delta\colon K^0(B)\to 
H^{odd}(B,\R/\Q)$ vanishes 

\subsubsection{}

Next we observe that
$(\hat\pr_2)_!\colon \hat K(S^1\times B)\to \hat K(B)$ is surjective.
In fact, we have
\begin{equation}\label{eq5500}(\hat\pr_2)_!(\pr_1^* x_{S^1}\cup \pr_2^*(x))=x\end{equation} by the projection formula \ref{projcl}
and $\hat p_!(x_{S^1})=1$ for $p\colon S^1\to *$, where $x_S^1\in \hat K(S^1)$ was defined in \ref{dddxs}.
Hence (\ref{eq5000}) implies that
$\bar \Delta\colon K^1(B)\to 
H^{ev}(B,\R/\Q)$
vanishes, too. \end{proof} 

\subsubsection{}\label{einss}

In view of Proposition \ref{unique} it remains to show the existence of the smooth Chern character.
We first construct the even part
$$\hat \ch\colon  \hat K^0(B)\rightarrow \hat H^{ev}(B,\Q)$$
using the splitting principle.
We will define $\hat \ch$ as a natural transformation of functors
such that the following conditions hold.
\begin{enumerate}
\item\label{item:o} $\hat \ch[\cL,0]=e^{\hat c_\Q(\bL)}\in \hat H^{ev}(B,\Q)$,
where $\cL$ is the geometric family given by a hermitean line bundle
with connection $\bL$, and $\hat c_\Q(\bL)\in\hat H^2(B,\Q)$ is derived from the Cheeger-Simons
Chern class which classifies the isomorphism class of $\bL$ (\ref{cqdef}).
\item\label{item:t}  $R\circ \hat \ch= R$
\item\label{item:th} $\hat \ch\circ a=a$
\end{enumerate}

Once this is done, the resulting $\hat \ch$ automatically satisfies 
(\ref{eq5001}). For this it suffices to show that $\ch\circ I= I\circ
\hat\ch$. We consider the following diagram
$$\xymatrix{\hat K(B)\ar@/^0.6cm/[rr]^{R}\ar[r]^{\hat \ch}\ar[d]^I&\hat H(B,\Q)\ar[d]^I\ar[r]^R&\Omega_{d=0}(B)\ar[d] \\K(B)\ar[r]^{\ch}&H(B,\Q)\ar[r]^i&H(B,\R)}
$$
The outer square and the right square commute. It follows from
\ref{item:t}. that the upper triange commutes. Since $i$ is injective
we conclude that the left square commutes, too.


\subsubsection{}

In the construction of the Chern character $\hat \ch$ we will use the splitting principle.
If $x\in \hat K^0(B)$, then there exists a $\Z/2\Z$-graded hermitean vector bundle with connection
$\bV=(V,h^V,\nabla^V)$ such that
$x=[\cV,\rho]$ for some $\rho\in \Omega^{odd}(B)/\im(d)$,  where $\cV$ is the zero-dimensional geometric family with underlying Dirac bundle $\bV$.
We will call $V$ the splitting bundle for $x$. Let
$F(V^\pm)\to B$ be the bundle of full flags on $V^\pm$ and $p\colon F(V):=F(V^+)\times_BF(V^-)\to B$.
Then we have a decomposition
$p^*V^{\pm}\cong \oplus_{L\in I^{\pm}} L$ for some ordered finite sets $I^\pm$ of line bundles
over $F(V)$. For $L\in I^{\pm}$ let $\bL$ denote the bundle with the induced metric and connection, and let 
$\cL$ be the corresponding zero-dimensional geometric family.
Then we have $p^*x=\sum_{L\in I^+} [\cL,0]-\sum_{L\in I^-}[\cL,0]+a(\sigma)$ for
some $\sigma\in \Omega^{odd}(F(V))/\im(d)$. The properties above thus uniquely determine
$p^*\hat \ch(x)$.

\begin{lem}\label{injjj}
The  following  pull-back operations are injective:
\begin{enumerate}
\item
$p^*\colon H^*(B,\Q)\rightarrow H^*(F(V),\Q)$,
\item
$p^*\colon H^*(B,\R)\rightarrow H^*(F(V),\R)$
\item
$p^*\colon H^*(B,\R/\Q)\rightarrow H^*(F(V),\R/\Q)$
\item
$p^*\colon \hat H^*(B,\Q)\rightarrow \hat H^*(F(V),\Q)$
\item
$p^*\colon \Omega(B)\rightarrow \Omega(F(V))$.
\end{enumerate}
\end{lem}
\begin{proof}
The assertion is a classical consequence of the Leray-Hirsch theorem
in the cases 1., 2., and 3. In case 5., it follows from the fact that
$p$ is surjective and a submersion.
It remains to discuss the case 4.  
Let $x\in  \hat H^*(B,\Q)$. Assume that $p^*x=0$. Then in particular
$p^*R(x)=R(p^*x)=0$ so that from 5. also $R(x)=0$.
Thus $x\in H(B,\R/\Q)$. We now apply 3. and see that
$p^*x=0$ implies $x=0$.
\end{proof}
In view of Proposition \ref{unique} we see that a natural transformation
$\hat \ch\colon  \hat K^0(B)\to \hat H^{ev}(B,\Q)$ is uniquely determined by the conditions \ref{item:o}., \ref{item:t}., and \ref{item:th}. formulated in \ref{einss}.

\subsubsection{}

\begin{prop}
There exists a natural transformation
$\hat \ch\colon \hat K^0(B)\to \hat H^{ev}(B,\Q)$ which satisfies the conditions  1. to 3. formulated in \ref{einss}.
\end{prop}
We give the proof of this Proposition in the next couple of subsections.
Let  $x:=[\cE,\rho]\in \hat K^0(B)$, and $V\to B$ be a splitting bundle for $x$ with bundle of flags $p\colon F(V)\to B$.
We choose a geometry $\bV:=(V,h^V,\nabla^V)$ and let $\cV$ denote the associated geometric family\footnote{It was suggested by the referee that one should use the Chern character $\hat \ch(V)\in \hat H^{ev}(B,\Q)$ constructed in \cite{MR827262}. The Ansatz would be
$$\hat \ch(x):=\hat \ch(\bV)+\eta((\cE\sqcup_B\cV^{op})_t)\ .$$
In order to show that this is independent of the choice of $\bV$ one would need to show
an equation like
$$\hat \ch(\bV)-\hat \ch(\bV^\prime)=a(\eta((\cV^{op}\sqcup \cV^{\prime})_t)) \ .$$
Since after all we know that the Chern character exists this equation is true,
but we do not know a {\em simple} direct proof. Therefore we opted for the
variant to give a complete and independent proof.}.
In order to avoid stabilizations we can and will always assume that $\cE$ has a non-zero dimensional component. 
Then we have
$$p^*I(x)=\sum_{\epsilon\in \{\pm 1\},L\in I^{\epsilon}}\epsilon I([\cL,0])\ .$$
We define $\cF:=\bigsqcup_{B,\epsilon\in \{\pm 1\},L\in I^{\epsilon}} \cL^{\epsilon}$.
Then we can find a taming $(p^*\cE\sqcup_{F(V)} \cF^{op})_t$,   and
$$p^*x=\sum_{\epsilon\in \{\pm 1\},L\in I^{\epsilon}}\epsilon ([\cL,0]) -a(p^*\rho-\eta((p^*\cE\sqcup_{F(V)}\cF^{op})_t))\ .$$
We now set
$$p^*\hat \ch(x)=\hat \ch(p^*x):=\sum_{\epsilon\in \{\pm 1\},L\in I^{\epsilon}}\epsilon   \exp(\hat c_\Q(\bL))
 +a(\eta((p^*\cE\sqcup_{F(V)}
\cF^{op})_t))-a(p^*\rho)\ .$$
This construction a priori depends on the choices of the representative of $x$, the splitting bundle $V\to B$, and the taming  $(\cE \sqcup_{F(V)}\cF^{op})_t$.

\subsubsection{}

In this paragraph we show that this construction is independent of the choices.
\begin{prop}\label{zzuzu}
Assume that there exists a class $z\in \hat H^{ev}(B,\Q)$
such that $$p^*z=\sum_{\epsilon\in \{\pm 1\},L\in I^{\epsilon}}\epsilon   \exp(\hat c_\Q(\bL))
 +a(\eta((p^*\cE\sqcup_{F(V)}
\cF^{op})_t))-a(p^*\rho)$$
for one set of choices. Then $z$ is determined by $x\in \hat K^0(B)$.
\end{prop}
\begin{proof}
If $(\cE^\prime,\rho^\prime)$ is another representative of $x$,
then we have $\ind(\cE)=\ind(\cE^\prime)$. Therefore we can take the
same splitting bundle for $\cE^\prime$. The following Lemma (together with Lemma \ref{injjj}) shows that
$z$ does not depend on the choice of the representative of $x$.
\begin{lem}
We have
$$
a(\eta((p^*\cE\sqcup_{F(V)}
\cF^{op})_t)- p^*\rho)=a(\eta((p^*\cE^\prime\sqcup_{F(V)}
\cF^{op})_t)-p^*\rho^\prime)
$$
\end{lem}
\begin{proof}
In fact, by Lemma \ref{lem333} there is a taming $(\cE^\prime\cup \cE^{op})_t$ such that
$\rho^\prime-\rho=\eta\left((\cE^\prime\cup \cE^{op})_t\right)$.
Therefore the assertion is equivalent to
$$a\left[\eta\left((p^*\cE\sqcup_{F(V)}
\cF^{op})_t\right)-\eta\left((p^*\cE^\prime\sqcup_{F(V)}
\cF^{op})_t\right)+p^*  \eta\left((\cE^\prime\sqcup_{F(V)}
  \cE^{op})_t\right)\right]=0\ .$$
But this is true since this sum of $\eta$-forms represents a rational
cohomology class of the form $\ch_{dR}(\xi)$. 
This follows from \ref{pap3} and the fact 
$$p^*\cE\sqcup_{F(V)}\cF^{op}\sqcup_{F(V)} p^*\cE^{\prime op} \sqcup_{F(V)} \cF\sqcup_{F(V)}p^*\cE^{\prime}\sqcup_{F(V)} p^*\cE^{op}$$
admits another taming with vanishing $\eta$-form (as in the proof of Lemma \ref{lem1}).
\end{proof}

\subsubsection{}

Next we discuss what happens if we vary the splitting bundle.
Thus let $V^\prime\to B$ be another $\Z/2\Z$-graded bundle which represents
$\ind(\cE)$. Let $p^\prime \colon F(V^\prime)\rightarrow B$ be the associated splitting bundle.
\begin{lem}
Assume that we have classes $c,c^\prime\in \hat H(B,\Q)$ such that
$$p^*c=\sum_{\epsilon\in \{\pm 1\},L\in I^\epsilon} \epsilon \exp(\hat
c_\Q(\bL)) 
 +a\left(\eta\left((p^*\cE\sqcup_{F(V)}
\cF^{op})_t\right)- p^*\rho\right)$$
and $$p^{\prime *}c^\prime=\sum_{\epsilon\in \{\pm 1\},L\in I^{\prime \epsilon}} \epsilon \exp(\hat c_\Q(\bL^\prime))
 +a\left(\eta\left((p^{\prime *}\cE\sqcup_{F(V^\prime)}
\cF^{\prime op})_t\right)- p^{\prime *}\rho\right)\ .$$
 Then we  have
$c=c^\prime$.
\end{lem}
\begin{proof}
Note that the right-hand sides depend on the geometric bundles $\bV,\bV^\prime$ since they depend on the induced connections on the line bundle summands.
We first discuss a special case, namely that $\bV^\prime$ is obtained
from
$\bV$ by stabilization, i.e.\ $\bV^\prime=\bV\oplus B\times (\C^m\oplus
(\C^m)^{op})$. In this case there is a natural embedding
$i\colon F(\bV)\hookrightarrow F(\bV^\prime)$ which is induced by
extension of the flags in $V$ by the standard flag in $\C^m$.
We can factor $p=p^\prime\circ i$.
Furthermore, there exists subsets $S^{\epsilon}\subset I^{\prime\epsilon}$ of line bundles (the last $m$ line bundles in the natural order)  and a natural  bijection  $I^{\prime\epsilon}\cong I^\epsilon\sqcup S^{\epsilon}$.
If $L\in S^\epsilon$, then $i^*L$ is trivial with the trivial connection.
We thus have
$$
p^{*}(c^\prime-c)=a\left[i^*\eta\left((p^{\prime *}\cE\cup
\cF^{\prime op })_t\right)-\eta\left((p^*\cE\cup
\cF^{op})_t\right)\right]$$
It is again easy to see that this difference of $\eta$-forms
represents a rational cohomology class in the image of $\ch_{dR}$.
Therefore,
$p^*(c^\prime-c)=0$ and hence $c=c^\prime$ by Lemma \ref{injjj}. 

Since the bundle $V$ represents the index of $\cE$, two choices are
always stably isomorphic as hermitean bundles.
Using the special case above we can reduce to the case where $\bV$ and
$\bV^\prime$ only differ by the connection.

We argue as follows.
We have $p^*R(c^\prime-c)=R(p^*(c^\prime-c))=0$ by an explicit
computation. Therefore $c^\prime-c\in H^{odd}(B,\R/\Q)$.
Since any two connections on $V$ can be connected by a family
we conclude that $p^*(c^\prime-c)=0$ by a homotopy argument.
The assertion now follows. 
\end{proof}

This finishes the proof of Proposition \ref{zzuzu}. \end{proof}

\subsubsection{}

In order to finish the construction of the Chern character in the even case
 it remains to verify the existence clause in Proposition  \ref{zzuzu}.
 Let $x:=[\cE,\rho]\in \hat K(B)$  be such that $\cE$ has a non-zero dimensional component. 
Let $V\rightarrow B$ be a splitting bundle and $p\colon F(V)\to B$ be as above.
\begin{lem}
We have
$$z:=\sum_{\epsilon\in \{\pm 1\},L\in I^{ \epsilon}}  \epsilon \exp(\hat c_\Q(\bL))
 +a\left[\eta\left((p^*\cE\cup
\cF^{op})_t\right)-p^*\rho\right]\in\im(p^*)\ .$$
\end{lem}
\begin{proof}
We use a Mayer-Vietoris sequence argument. Let us first recall the Mayer-Vietoris sequence for smooth rational cohomology. Let $B=U\cup V$ be an open  covering of $B$.  Then we have the exact sequence
$$\dots\to H(U\cap V,\R/\Q)\to \hat H(B,\Q)\to \hat H(U,\Q)\oplus \hat H(V,\Q)\to \hat H(U\cap V,\Q)\to H(B,\Q)\to \dots$$
which continues to the left and right by the Mayer-Vietoris sequences of
$H(\dots,\R/\Q)$ and $H(\dots,\Q)$.

We choose a finite covering of $B$ by contractible subsets. Let $U$ be one of these.
Note that $\ind(\cE)_{|U}\in\Z$. Thus
$x_{|U}=[U\times W,\theta]$ for some form $\theta$ and $\Z/2\Z$-graded vector space $W$.
Then we have by 1.~and 3.~that $c_U\colon =\hat \ch(x_{|U})=\dim(W)-a(\theta)$.
This can be seen using the splitting bundle $F(B\times \C^n)$. 
Moreover, $p^*c_U=p^*[\dim(W)-a(\theta)]=z_{|p^{-1}U}$ by Proposition
\ref{zzuzu}.

Assume now that we have already constructed $c_V\in \hat H(V,\Q)$ such
that $p^*c_V=z_{|p^{-1}V}$, where $V$ is
a union $V$ of these subsets. Let $U$ be the next one in the list.

We show that we can extend $c_V$ to $c_{V\cup U}$.  
We have
$(c_U)_{|U\cap V}=(c_V)_{|U\cap V}$ by the injectivity of the pull-back $p^*\colon \hat H(U\cap V,\Q)\to \hat H(p^{-1}(U\cap V),\Q)$, Lemma \ref{injjj}.
The Mayer-Vietoris sequence implies that we can extend $c_V$ by $c_U$ to $U\cup V$.
\end{proof}

\subsubsection{}

We now construct the odd part of the Chern character. In fact, by (\ref{eq5000}) and (\ref{eq5500})
we are forced to define $$\hat \ch\colon \hat K^1(B)\to \hat H^{odd}(B,\Q)$$ by 
$$\hat \ch(x):=(\hat\pr_2)_!(\hat\ch(x_{S^1}\cup x))\ .$$
\begin{lem}
The diagrams (\ref{eq5001}) and (\ref{eq5000}) commute.
\end{lem}
\begin{proof}
The even case of (\ref{eq5001}) has been checked already.
 The diagram (\ref{eq5000}) commutes by construction.
The odd case of (\ref{eq5001}) follows from the Projection formula \ref{projcl} and the even case.
\end{proof} 
This finishes the proof of Theorem \ref{mmmain}

\subsection{The Chern character is a rational isomorphism and multiplicative}

\subsubsection{}

Note that $\hat H(B,\Q)$ is a $\Q$-vector space, and that the sequence (\ref{eq9000}) is an exact sequence of $\Q$-vector spaces. The Chern character extends to  a rational version
$$\hat \ch_\Q\colon \hat K_\Q(B)\to \hat H(B,\Q)\ ,$$ where
$\hat K_\Q(B):=\hat K(B)\otimes_\Z \Q$.
\begin{prop}\label{qiso}
$\hat \ch_\Q\colon \hat K_\Q(B)\to \hat H(B,\Q)$ is an isomorphism.
\end{prop}
\begin{proof}
By (\ref{eq5001}) we have the commutative diagram
$$\xymatrix{K_\Q(B)\ar[d]^{\ch_\Q}\ar[r]^{\ch_{dR}}&\Omega(B)/\im(d)\ar[r]^a\ar@{=}[d]&\hat K_\Q(B)\ar[d]^{\hat \ch_\Q}\ar[r]^I&K_\Q(B)\ar[d]^{\ch_\Q}\ar[r]&0\\
H(B,\Q)\ar[r]&\Omega(B)/\im(d)\ar[r]&\hat H(B,\Q)\ar[r]^I&H(B,\Q)\ar[r]&0}\ ,$$
whose horizontal sequences are exact.
Since $\ch_\Q\colon K_\Q(B)\to H(B,\Q)$ is an isomorphism we conclude that $\hat \ch_\Q$ is an isomorphism by the Five Lemma. \end{proof}

\subsubsection{}

We can extend $\hat K_\Q$ to a smooth cohomology theory if we define the structure maps as follows:
\begin{enumerate}
\item $R\colon \hat K_\Q(B)\to \Omega_{d=0}(B)$ is the rational extension of $R\colon \hat K(B)\to \Omega_{d=0}(B)$.
\item $I\colon \hat K_\Q(B)\stackrel{I\otimes \id_\Q}{\to} K(B)_\Q\stackrel{\ch_\Q}{\to} H(B,\Q)$,
\item $a\colon \Omega(B)/\im(d)\stackrel{a}{\to} \hat K(B)\stackrel{\dots\otimes 1}{\to }\hat K_\Q(B)$.
\end{enumerate}
The commutative diagrams (\ref{eq5001}) now imply:
\begin{kor}
The rational Chern character induces an isomorphism  of smooth cohomology
theories refining the isomorphism 
$\ch_\Q\colon K_\Q\to H\Q$ (in the sense of Definition \ref{natdef12}).
 \end{kor}

\subsubsection{}

\begin{prop}\label{rrr}
The smooth Chern character
$$\hat \ch\colon \hat K(B)\to \hat H(B,\Q)$$ is a ring homomorphism.
\end{prop}
\begin{proof}
Since the target of $\hat \ch$ is a $\Q$-vector space it suffices to show that
$\hat \ch_\Q\colon \hat K_\Q(B)\to \hat H(B,\Q)$ is a ring homomorphism.
Using that  $\hat \ch_\Q$ is an isomorphism of smooth extensions of rational
cohomology we can use the rational Chern character in order to transport the
product on $\hat K_\Q(B)$
to a second product $\cup_K$ on $\hat H(B,\Q)$. It remains to show that $\cup$ and $\cup_K$ coincide.
Hence the following Lemma finishes the proof of Proposition \ref{rrr}.

\subsubsection{}
\begin{lem}\label{lem12000}
There is a unique product on smooth rational cohomology.
\end{lem}
\begin{proof}
Assume that we have two products $\cup_k$, $k=0,1$.
We consider the bilinear transformation
$\bB\colon \hat H(B,\Q)\times \hat H(B,\Q)\to \hat H(B,\Q)$ given by 
$$(x,y)\mapsto \bB(x,y):=x\cup_1 y -x\cup_0 y\ .$$
We first consider the curvature. Since a product is compatible with the curvature (\ref{multdef1}, 2.) we get 
$$R(\bB(x,y))=R(x\cup_1 y)-R(x\cup_0y)=R(x)\wedge R(y)-R(x)\wedge R(y)=0\ .$$
Therefore, by (\ref{eq9000}) the bilinear form factors over an odd transformation
$$\bB\colon \hat H(B,\Q)\times \hat H(B,\Q)\to H(B,\R/\Q)\ .$$ 
Furthermore, for $\omega\in \Omega(B)/\im(d)$ we have by \ref{multdef1}, 2.
$$\bB(a(\omega),y)=a(\omega)\cup_1y-a(\omega)\cup_0y= a(\omega\wedge R(y))- a(\omega\wedge R(y))=0\ .$$
Similarly, $\bB(x,a(\omega))=0$.
Again by (\ref{eq9000}) $B$ has a factorization over a natural bilinear transformation
$$\bar \bB\colon H(B,\Q)\times H(B,\Q)\to H(B,\R/\Q)\ .$$
We consider the restriction
$\bar \bB^{p,q}$ of $\bar \bB$ to $H^p(B,\Q)\times H^q(B,\Q)$.

The functor from finite $CW$-complexes to sets
$$W\to  H^p(W,\Q)\times H^q(W,\Q)$$ is represented by a product of Eilenberg MacLane spaces
$$P^{p,q}:= H\Q^p\times  H\Q^q\ .$$
The spaces $H\Q^p$, and hence $P$ has the homotopy type of countable $CW$-complexes.
Therefore we can apply Lemma \ref{hteh} and conclude that
$\bar \bB^{p,q}$ is induced by a cohomology class $b\in H(P^{p,q},\R/\Q)$.
We finish the proof of Lemma \ref{lem12000} by showing that $b=0$. To this end we analyse the candidates for $b$ and show that they vanish either for degree reasons, or using the fact that $\bar \bB^{p,q}$ is bilinear.

Consider a homomorphism of $\Q$-vector spaces $w\colon \R/\Q\to \Q$. It induces a transformation
$w_*\colon H(B,\R/\Q)\to H(B,\Q)$. In particular we can consider $w_*b\in H(P^{p,q},\Q)$.

\begin{enumerate}\item
First of all if $p,q$ are both even, then  
$w_*b\in H^{odd}(P^{p,q,},\Q)$ vanishes since
$P^{p,q}$ does not have odd-degree rational cohomology at all.
 \item
Assume now that $p,q$ are both odd.
The odd rational cohomology of $P^{p,q}$ is additively generated by the classes
$1\times x_q$ and $x_p\times 1$, where $x_p\in H^p(H\Q^p,\Q)$ and $x_q\in H^q(H\Q^q,\Q)$.
It follows that
$$w_*b=  c\cdot x_p\times 1 +  d \cdot 1\times x_q$$
for some rational constants $c,d$.
Consider odd classes $u_p\in H^p(B,\Q)$ and $v_q\in H^q(B,\Q)$.
 The form of $b$ implies that 
$$w_*\circ \bar \bB^{p,q}(u_p,v_q)=c\cdot  u_p\times 1 + d\cdot  1\times v_q\ .$$
This can only be bilinear if all $c$ and $d$ vanish. Hence $b=0$.
\item
Finally we consider the case that $p$ is even and $q$ is odd (or vice versa, $q$ is even and $p$ is odd). In this case $b$ is an even class.  
The even cohomology of
$P^{p,q}$ is additively generated by the classes
$x_p^n\times 1$, $n\ge 0$.  Therefore $w_*b= \sum_{n\ge 0}c_n  x^n_p\times 1$
for some rational constants $c_n$, $n\ge 0$.
Let $u_p\in H^{p}(B,\Q)$ and $v_q\in H^{q}(B,\Q)$.
Then we have $$w_*\circ \bar \bB^{p,q}(u_p,v_q)= \sum_{n\ge 0} c_n\  u^n_p\ .$$
This is only bilinear if $c_n=0$ for all $n\ge 0$, hence $w_*b=0$.  
\end{enumerate}
Since we can choose $w_*\colon \R/\Q\to \Q$ arbitrary we conclude that $b=0$.
\end{proof}

This also finishes the proof of the Proposition \ref{rrr}.
\end{proof}

\subsection{Riemann Roch theorem}

\subsubsection{}
 
Let $p\colon W\to B$ be a proper submersion with a smooth $K$-orientation $o$.
The Riemann Roch theorem asserts the commutativity of a diagram
$$
\begin{CD}
  \hat K(W) @>{\hat\ch}>> \hat H(W,\Q)\\
  @VV{p_!}V @VV{\hat p_!^A}V\\
  \hat K(B) @>{\hat\ch}>> \hat H(B,\Q)
\end{CD}
.$$
Here $\hat p_!^A$ is the composition  of the cup product with a  smooth
rational cohomology class $\hat \hA^c(o)$ and the push-forward in smooth
rational cohomology.
The Riemann Roch theorem refines the characteristic class version of the ordinary index theorem for families.

We will first give the details of the definition of the push-forward $\hat
p^A_!$.
In order to show the Riemann Roch theorem we then show that the difference
$$\Delta:=\hat \ch\circ \hat p_!-\hat p_!^A\circ \hat \ch$$
vanishes.

This is proved in several steps. First we use the compatibilites of the push-forward
with the transformations $a,I,R$ in order to show that
$\Delta$ factors over a map
$$\bar \Delta\colon  K(W)\to H(B,\R/\Q)\ .$$
In the next step we  show that
$\Delta$ is natural with respect to the pull-back of fibre bundles, and that it does neither depend on the smooth nor on the topological $K$-orientations of $p$.

We then show that $\Delta$ vanishes in the special case that $B=*$.
The argument is based on the bordism invariance Proposition  \ref{bordin} and some calculation of 
rational $Spin^c$-bordism groups.

Finally we use the functoriality of the push-forward Proposition \ref{funktt}
in order to reduce the case of a general $B$ to the special case of a point.

\subsubsection{}\label{pap88771}

We consider a proper submersion   $p\colon W\rightarrow B$  with closed fibres with a smooth $K$-orientation  represented by $o=(g^{T^vp},T^hp,\tilde \nabla,\sigma)$. In the following we define a refinement $\hat \hA(o)\in \hat H^{ev}(W,\Q)$ of the form $\hA^c(o)\in \Omega^{ev}(W)$.
The geometric data of $o$ determines a connection $\nabla^{T^vp}$ (see \ref{uzu1}, \ref{ldefr}) and hence a geometric bundle $\mathbf{T^vp}:=(T^vp,g^{T^vp},\nabla^{T^vp})$.  According to \cite{MR827262} we can define Pontrjagin classes $$\hat p_{i}(\mathbf{T^vp})\in \hat H^{4i}(W,\Z)\ ,\quad i\ge 1\ .$$
The $Spin^c$-structure gives rise to a hermitean line bundle $L^2\to W$ with connection $ \nabla^{L^2}$ (see \ref{l2def}).
A choice of a local spin structure amounts to a choice of a local square root $L$ of $L^2$ (this bundle was considered already in  \ref{ldefr}) such that $S^c(T^vp)\cong S(T^vp)\otimes L$ as hermitean bundles with connections.
We set $\bL^2:=(L^2,h^{L^2},\nabla^{L^2})$.
In particular, we have
$$\frac{1}{2\pi i }R^{\tilde \nabla^{L^2}}=2c_1(\tilde \nabla)\ .$$
Again using \cite{MR827262}  we get a class
$$\hat c_1(\bL^2)\in \hat H^2(W,\Z)$$
with curvature
$R(\hat c_1(\bL^2))=2c_1(\tilde \nabla)$.

\subsubsection{}

Inserting the  classes $\hat p_{i}(\mathbf{T^vp})$ into that $\hA$-series $\hA(p_1,p_2,\dots)\in \Q[[p_1,p_2\dots]]$ we can define
\begin{equation}\label{hahaha}\hat\hA(\mathbf{T^vp}):=\hA(\hat p_1(\mathbf{T^vp}),\hat p_2(\mathbf{T^vp}),\dots)\in \hat H^{ev}(W,\Q)\ .\end{equation}
Let $\hat c_\Q(\bL^2)\in \hat  H^2(W,\Q)$ denote the image of $\hat c_1(\bL^2)$ under the natural map $\hat  H^2(W,\Z)\to \hat  H^2(W,\Q)$.
\begin{ddd}
We define
$$\hat \hA^c(o):=\hat\hA(\mathbf{T^vp})\wedge e^{\frac{1}{2}\hat c_\Q(\bL^2)}\in \hat H^{ev}(W,\Q)\ .$$
\end{ddd}
Note that
$R(\hat \hA^c(o))=\hA^c(o)$.

\begin{lem}\label{weldioe}
The class\footnote{This class  is denoted by $A(p)$ in the abstract and \ref{uiefwefwef}.}
$$\hat \hA^c(o)-a(\sigma(o))\in \hat H^{ev}(W,\Q)$$
only depends on the smooth $K$-orientation represented by $o$.
\end{lem}
\begin{proof}
This is a consequence of the homotopy formula Lemma \ref{lem23}. 
Given two representatives $o_0,o_1$ of a smooth $K$-orientation we can choose
a representative $\tilde o$ of a smooth $K$-orientation on $\id_\R\times
p\colon \R\times W\to \R\times B$ which restricts to $o_k$ on $\{k\}\times B$,
$k=0,1$. The construction of the class $\hat \hA^c(o)$ is compatible with
pull-back. Therefore by the definition of the transgression form \ref{uzu4} we have
$$\hat \hA^c(o_1)- \hat \hA^c(o_0)=i_1^*\hat \hA^c(\tilde o)-i_0^*\hat \hA^c(\tilde o)=
a\left[\int_{[0,1]\times W/W} R(\hat \hA^c(\tilde o))\right]=a\left[\tilde \hA^c(\tilde \nabla_1,\tilde \nabla_0)\right]\ .$$ 
By the definition of equivalence of representatives of smooth $K$-orientations we have
$$\sigma(o_1)-\sigma(o_0)=\tilde \hA^c(\tilde \nabla_1,\tilde \nabla_0)\ .$$
Therefore
$$\hat \hA^c(o_1)-a(\sigma(o_1))=\hat \hA^c(o_0)-a(\sigma(o_0))\ .$$
\end{proof}

\subsubsection{}

We use the class $\hat \hA^c(o)\in \hat H^{ev}(W,\Q)$ in order to define the
push-forward 
\begin{equation}\label{eq9811}
\hat p^A_!:=\hat p_!([\hat \hA^c(o)-a(\sigma(o))]\cup\dots)\colon \hat H(W,\Q)\to \hat H(B,\Q)\ ,
\end{equation}
where $\hat p_!\colon \hat H(W,\Q)\to \hat H(B,\Q)$ is the  push-forward in smooth rational cohomology (see \ref{ratdel})
fixed by the underlying ordinary orientation of $p$.
By Lemma  \ref{weldioe} also $\hat p_!^A$ only depends to the smooth
$K$-orientation of $p$ and not on the choice 
of the representative.

If $f\colon B^\prime\to B$ is a smooth map then
we consider the pull-back diagram 
$$\xymatrix{W^\prime\ar[d]^{p^\prime}\ar[r]^F&W\ar[d]^p\\B^\prime\ar[r]^f&B}\ .$$
The smooth $K$-orientation $o$ of $p$ induces (see \ref{pap101}) a smooth $K$-orientation $o^\prime$ of $p^\prime$. We have
$\hat \hA(o^\prime)=F^*\hat \hA(o)$ and
$\hat p^{\prime A}_!\circ F^*=f^*\circ\hat p^A_!$.


\subsubsection{}
As in \ref{tz123}
we consider the composition of proper smoothly $K$-oriented submersions
$$\xymatrix{W\ar@/_0.5cm/[rr]_q\ar[r]^p&B\ar[r]^r&A}\ .$$
The composition $q:=r\circ p$ has an induced smooth $K$-orientation
(Definition \ref{def100} and Lemma \ref{lem19123}). In this situation we have
push-forwards $\hat p^A_!$, $\hat r^A_!$ and $\hat q^A_!$ in smooth rational
cohomology given
by  (\ref{eq9811}). 
\begin{lem}\label{functk}
We have the  equality $$\hat r_!^A\circ\hat p^A_!=\hat q^A_!$$ of maps $\hat
H(W,\Q)\to \hat H(B,\Q)$.
\end{lem}
\begin{proof} We choose representatives of smooth $K$-orientations $o_p$ of $p$ and $o_r$ of
$r$, and  we let $o^\lambda_q:=o_p\circ_\lambda o_r$ be the composition.
We consider the class (see Definition \ref{def100})
\begin{eqnarray*}\lefteqn{
\hat \hA^c(o_q^\lambda)-a(\sigma(o^\lambda_q))}&&\\&=&\hat \hA^c(o_q^\lambda)-a\left(
\sigma(o_p)\wedge p^*\hA^c(o_r)+\hA^c(o_p)\wedge
p^*\sigma(o_r)- \tilde\hA^c (\tilde \nabla^{adia},\tilde \nabla^\lambda_q)-d\sigma(o_p)\wedge p^*\sigma(o_r)\right)\ .
\end{eqnarray*}

By Lemma \ref{weldioe} and Lemma \ref{lem19123} this class is independent of
$\lambda$. 
If we let $\lambda\to 0$, then the connection $\nabla^{T^vq}$ tends to the
direct sum  connection $\nabla^{T^vp}\oplus p^*\nabla^{T^vr}$.
Furthermore, the transgression  $\tilde\hA^c (\tilde \nabla^{adia},\tilde \nabla^\lambda_q)$ tends to zero.  Therefore
\begin{eqnarray*}\lefteqn{
\lim_{\lambda\to 0}[\hat \hA^c(o_q^\lambda)-a(\sigma(o^\lambda_q))]}&&\\&=&\hat \hA^c(o_p)\cup p^*\hat \hA^c(o_r)-a\left(\sigma(o_p)\wedge p^*\hA^c(o_r)+\hA^c(o_p)\wedge
p^*\sigma(o_r)-  d\sigma(o_p)\wedge p^*\sigma(o_r)\right)\\
&=&(\hat \hA^c(o_p)-a(\sigma(o_p)))\cup p^*(\hat \hA^c(o_r)-a(\sigma(o_r)))\ .
\end{eqnarray*}
For $x\in \hat H(W,\Q)$ we get using the  projection formula and the
functorialty $\hat q_!=\hat r_!\circ \hat p_!$ for the push-forward in smooth
rational cohomology
\begin{eqnarray*}
\hat r_!^A\circ \hat p^A_!(x)&=&\hat r_!\left(\left[\hat
    \hA^c(o_r)-a(\sigma(o_r))\right]\cup \hat p_!\left(\left[\hat
      \hA^c(o_p)-a(\sigma(o_p))\right]\cup x\right)\right)\\
&=&\hat q_! \left(p^*\left[\hat \hA^c(o_r)-a(\sigma(o_r))\right]\cup
  \left[\hat \hA^c(o_p)-a(\sigma(o_p))\right]\cup x\right)\\
&=&
\hat q_!\left((\hat \hA^c(o_q^a)-a(\sigma(o^a_q)))\cup x\right)\\
&=&\hat q_!^A(x)\ .
\end{eqnarray*}
\end{proof}

\subsubsection{}

Recall Definition \ref{ddd1} that 
the smooth $K$-orientation determines a push-down
$$\hat p_!\colon \hat K(W)\to \hat K(B)\ .$$
We can now formulate the index theorem.
\begin{theorem}\label{main}
The following square commutes
$$
\begin{CD}
  \hat K(W) @>{\hat\ch}>> \hat H(W,\Q)\\
  @VV{\hat p_!}V @VV{\hat p_!^A}V\\
  \hat K(B) @>{\hat\ch}>> \hat H(B,\Q)
\end{CD}
.$$
\end{theorem}
\begin{proof}
We consider the difference
$$\Delta:=\hat \ch\circ \hat p_!-\hat p_!^A\circ \hat \ch\ .$$ 
It suffices to show that $\Delta=0$.

\subsubsection{}
Let $x\in \hat K(W)$.
\begin{lem}\label{lem666}
We have $R(\Delta(x))=0$.
\end{lem}
\begin{proof}
This Lemma is essentially equivalent to the local index theorem.
We have by Definition \ref{def313} and Lemma \ref{lem24}
$$R(\hat \ch\circ \hat p_!(x))=R(\hat p_!(x))=p_!(R(x))=\int_{W/B}
\left(\hA^c(o)-d\sigma(o)\right)\wedge R(x)\ .$$ 
On the other hand, since
$R\left(\hat \hA^c(o)-a(\sigma(o))\right)=\hA^c(o)-d\sigma(o)$ we get
$$R\left(\hat p^A_!\circ \hat \ch(x)\right)=\int_{W/B}
\left(\hA^c(o)-d\sigma(o)\right)\wedge R(\hat \ch(x))=\int_{W/B}
\left(\hA^c(o)-d\sigma(o)\right)\wedge R(x)\ .$$ 
Therefore $R(\Delta(x))=0$. \end{proof} 

\subsubsection{}

\begin{lem}
We have $I(\Delta(x))=0$
\end{lem}
\begin{proof}
This is the usual index theorem.
Indeed, $$I(\hat \ch\circ \hat p_!(x))=\ch\circ I(\hat p_!(x))=\int_{W/B}
\hA^c(T^vp)\cup \ch(I(x))$$ and 
$$I\left(\hat p^A_!\circ \hat \ch(x)\right)=\int_{W/B}\hA^c(T^vp)\cup I(\hat \ch(x))=\int_{W/B} \hA^c(T^vp)\cup \ch(I(x))\ .$$
The equality of the right-hand sides proves the Lemma.
Alternatively one could observe that the Lemma is a consequence of Lemma \ref{lem666}.
\end{proof}

\subsubsection{}

Let $\omega\in \Omega(W)/\im(d)$. 
\begin{lem}\label{lemuuu7}
We have $\Delta(a(\omega))=0$.
\end{lem}
\begin{proof}
We have by Proposition \ref{mainprop}
$$\hat \ch\circ \hat p_!(a(\omega))=\hat \ch\circ a
(p_!(\omega))=a\left(\int_{W/B}\left(\hA^c(o)-d\sigma(o)\right)\wedge
  \omega\right)\ .$$  
On the other hand, by (\ref{eq5001}) and 
$$\left[\hat \hA^c(o)-a(\sigma(o))\right]\cup a(\omega)=
a\left(R\left(\hat \hA(o)-a(\sigma(o))\right)\wedge \omega\right)=a\left(\left(\hA^c(o)-d\sigma(o)\right)\wedge \omega\right),$$

$$\hat p^A_!\circ \hat \ch(a(\omega))=\hat p^A_!(a(\omega))= 
a\left(\int_{W/B}\left(\hA^c(o)-d\sigma(o)\right)\wedge \omega\right)\ .$$
\end{proof} 

\subsubsection{}

Let $o_0,o_1$ represents two smooth refinements of the same topological $K$-orientation of $p$. Assume that $\Delta_k$ is defined with
the choice $o_k$, $k=0,1$.

\begin{lem}
We have $\Delta_0=\Delta_1$. 
\end{lem}
\begin{proof}
We can assume that
$o_k=(g^{T^vp},T^hp,\tilde\nabla,\sigma_k)$ for $\sigma_k\in
\Omega^{odd}(W)/\im(d)$.
  
Then we have for $x\in \hat K(W)$
\begin{eqnarray*}
\Delta_1(x)-\Delta_0(x)&=&-a\left(\int_{W/B} (\sigma_1-\sigma_0)\wedge R(x)
\right) 
+\int_{W/B}a(\sigma_1-\sigma_0)\cup
\hat \ch(x)\\
&=&-a\left(\int_{W/B} (\sigma_1-\sigma_0)\wedge R(x)\right) 
+\int_{W/B}a\left[(\sigma_1-\sigma_0)\wedge
R\circ \hat \ch(x)\right]\\
&=&0
\end{eqnarray*}
since $R\circ \hat \ch(x)=R(x)$ and $a\circ \int_{W/B}=\int_{W/B}\circ a$.
\end{proof}

\subsubsection{}

It follows from Lemma \ref{lem666} and (\ref{eq9000}) that
$\Delta$ factorizes through a transformation
$$\Delta\colon \hat K(W)\to H(B,\R/\Q)\ .$$
By Lemma \ref{lemuuu7} and \ref{prop1}
the map $\Delta$ factors over a map
$$\bar \Delta\colon  K(W)\to H(B,\R/\Q)\ .$$
This map only depends on the topological $K$-orientation of $p$.
It is our goal to show that $\bar \Delta=0$.

 \subsubsection{}

Next we want to show that the transformation $\bar \Delta$ is natural.
For the moment we write $\Delta_p:=\bar \Delta$. 
Let  $f\colon B^\prime\rightarrow B$ be a smooth map and
form the cartesian diagram
$$\xymatrix{W^\prime\ar[d]^{p^\prime}\ar[r]^F&W\ar[d]^p\\B^\prime\ar[r]^f&B}\ .$$
The map $p^\prime$ is a proper submersion with closed fibres which has an induced topological $K$-orientation.

\begin{lem} \label{lemnat33}
We have the equality of maps $K(W)\to H(B^\prime,\R/\Q)$
$$\Delta_{p^\prime}\circ F^*=f^*\circ \Delta_p\ .$$
\end{lem}
\begin{proof}
This follows from the naturality of $\hat\ch$, $\hat p_!$, and $\hat p_!^A$
with respect to the base $B$.
\end{proof}

\subsubsection{}

\begin{lem}\label{s1}
If $\pr_2\colon S^1\times B\rightarrow B$ is the trivial bundle with the topological $K$-orientation given by the
bounding spin structure, then
$\Delta_{\pr_2}\colon K^0(S^1\times B)\rightarrow H^{odd}(B,\R/\Q)$ vanishes.
\end{lem}
\begin{proof}
The odd Chern character is defined such that for $x\in K^0(S^1\times B)$ we have
$\hat \ch_1((\hat\pr_2)_!x)=(\hat\pr_2)_!\hat\ch_0(x)$  (see (\ref{eq5000})).
With the choice of the smooth $K$-orientation of $\pr_2$ given in
\ref{pullpush1} we have $\hat \hA(o)-a(\sigma(o))=1$ so that $\hat p_!^A=\hat
p_!$.
This implies the Lemma.
\end{proof} 
 
\subsubsection{}\label{pap8877}

The group $H^2(W,\Z)$ acts simply transitive on the set of $Spin^c$-structures
of $T^vp$. Let $Q\to W$ be a unitary line bundle classified by $c_1(Q)\in H^2(W,\Z)$.
We choose a hermitean connection $\nabla^Q$ and form the geometric line bundle $\bQ:=(Q,h^Q,\nabla^Q)$. Let $o:=(T^vp,T^hp,\tilde \nabla,\rho)$ represent a smooth $K$-orientation refining the given  topological $K$-orientation of $p$. Note that $\tilde \nabla$ is completely determined by the Clifford connection on the Spinor bundle $S^c(T^vp)$. The spinor bundle of the shift of the topological $K$-orientation by $c_1(Q)$ is given by $S^c(T^vp)^\prime=S^c(T^vp)\otimes Q$.
We  construct a corresponding smooth $K$-orientation $o^\prime=(T^vp,T^hp,\tilde \nabla\otimes\nabla^Q,\rho)$.
We let $\hat p_!$ and $\hat p_!^\prime$ denote the corresponding push-forwards
in smooth $K$-theory.
Let $\cQ$ be the geometric family over $W$ with zero-dimensional fibre given by the bundle $\bQ$ (see \ref{zerofibre}).
The push-forwards $\hat p_!$ and $\hat p_!^\prime$  are now related as
follows:
\begin{lem}
$$\hat p_!^\prime(x)=\hat p_!([\cQ,0]\cup x),\qquad \forall x\in \hat K(W).$$
\end{lem}
\begin{proof}
Let $x=[\cE,\rho]$.
By an inspection of the constructions leading to Definition \ref{ddd7771} we
see that 
$$p^{\prime \lambda}_!\cE=p^\lambda_!(\cQ\times_W \cE)\ .$$ 
Furthermore we have
$c_1(\tilde \nabla\otimes \nabla^Q)=c_1(\tilde \nabla)+c_1(\nabla^Q)$ so that
$$\hA^c(o^\prime)=\hA^c(o)\wedge e^{c_1(\nabla^Q)}\ .$$
On the other hand, since $\Omega(\cQ)= e^{c_1(\nabla^Q)}$ we have $$[\cQ,0]\cup [\cE,\rho]=[\cQ\times_W\cE,e^{c_1(\nabla^Q)}\wedge \rho]$$
Using the explicit formula (\ref{eq300}) we get
$$\hat p_!^\prime([\cE,\rho])-\hat p_!([\cQ,0]\cup [\cE,\rho])=[\emptyset,
\tilde \Omega^\prime(\lambda,\cE)-\tilde \Omega(\lambda,\cE)]$$ for all small
$\lambda>0$. Since both transgression forms vanish in the limit $\lambda=0$
we get the desired result.
\end{proof}

In the notation of \ref{pap88771} we have $\bL^\prime=\bL\otimes \bQ$.
Therefore
$$\hat c_\Q(\bL^{\prime 2})=\hat c_\Q(\bL^2)+ 2\hat c_\Q(\bQ)$$
and hence we can express $\hat p_!^{\prime,A}$ according to (\ref{eq9811}) as 
$$\hat p^{\prime A}_!(x)=\hat p_!\left[\left(\hat \hA^c(o)\cup e^{\hat
      c_\Q(\bQ)}-a(\sigma(o))\right)\cup x\right]\ .
$$
 
\subsubsection{}

As before, let $p\colon W\to B$ be a proper oriented submersion which admits
topological $K$-orientations.
\begin{lem}\label{ddedwd}
If $\Delta_p=0$ for some topological $K$-orientation of $p$, then it vanishes  for every topological
$K$-orientation of $p$.
\end{lem}
\begin{proof}
We fix the $K$-orientation of $p$ such that $\Delta_p=0$ and let $p^\prime$ denote the same map with the topological $K$-orientation shifted by $c_1(Q)\in H^2(W,\Z)$. We continue to use the notation of \ref{pap8877}.
We choose a representative $o$ of a smooth $K$-orientation of $p$ refining the topological $K$-orientation.
For simplicity we take $\sigma(o)=0$. Furthermore, we take $o^\prime$ as above.
Using $\hat \ch([\cQ,0])=e^{\hat c_\Q(\bQ)}$ and the multiplicativity of the Chern character we get
\begin{eqnarray*}
\hat p^{\prime A}_!\circ \hat \ch(x)- \hat \ch\circ \hat p^\prime_!(x)
&=&\hat p_!\left[\hat \hA^c(o)\cup e^{\hat c_\Q(\bQ)} \cup \hat
  \ch(x)\right]-\hat \ch\circ \hat p_!\left([\cQ,0]\cup x\right)\\
&=&\hat p_!\left[\hat \hA^c(o)\cup\hat \ch([\cQ,0])  \cup \hat
  \ch(x)\right]-\hat p_!^A\circ \hat \ch\left([\cQ,0]\cup x\right)\\ 
&=&\hat p^A_! \circ \hat \ch([\cQ,0]  \cup x)-\hat p_!^A\circ \hat
\ch([\cQ,0]\cup x)\\
&=&0\ .
\end{eqnarray*}
\end{proof}

\subsubsection{}

We now consider the special case that  $B={*}$ and $W$ is an odd-dimensional $Spin^c$-manifold.
Since $H(*,\R/\Q)\cong \R/\Q$  we get a homomorphism
$$\Delta_p\colon K(W)\rightarrow \R/\Q\ .$$
\begin{prop}\label{point}
If $B\cong *$, then $\Delta_p=0$
\end{prop}
\begin{proof}
First note that $\Delta_p$ is trivial on $K^1(W)$ for degree reasons.
It therefore suffices to study $\Delta_p\colon K^0(W)\to \R/\Q$.
Let $x\in K^0(W)$ be classified by $\xi\colon W\to \Z\times BU$.
It gives rise to an element $[\xi]\in \Omega_{\dim(W)}^{Spin^c}(\Z\times BU)$ of the
$Spin^c$-bordism group of $\Z\times BU$.

\begin{lem}\label{rtrtrt}
If $[\xi]=0$, then $\Delta_p=0$.
\end{lem}
\begin{proof}
Assume that $[\xi]=0$. In this case there exists a compact $Spin^c$-manifold
$V$ with boundary $\partial V\cong W$ (as $Spin^c$-manifolds), and a map
$\nu\colon V\to \Z\times BU$ such that $\nu_{|\partial V}=\xi$.

We can choose a $\Z/2\Z$-graded vector bundle $E\to V$ which represents the class
of $\nu$ in $K^0(V)$. We refine $E$ to a geometric bundle
$\bE:=(E,h^E,\nabla^E)$ and form the associated geometric family $\cE$ with zero-dimensional fibre.

We choose a representative $\tilde o$ of a smooth $K$-orientation of the map
$q\colon V\to *$ which refines the topological $K$-orientation given by the
$Spin^c$-structure and which has a product structure near the boundary. For simplicity we assume that $\sigma(\tilde o)=0$. The restriction of $\tilde o$ to the boundary $\partial V$ defines a smooth $K$-orientation of $p$. 

We let $\hat y:=[\cE,0]\in \hat K(V)$, and we define $\hat x:=\hat y_{|\partial V}$ such that $I(\hat x)=x$.  By Proposition \ref{bordin} we have
$$\hat \ch\circ \hat p_!(\hat x)=\hat \ch\circ \hat p_!(\hat y_{|W})=\hat
\ch([\emptyset,q_!(R(\hat y))])=-a\left(\int_{V}\hA^c(\tilde o)\wedge R(\hat
  y)\right)\ .$$
On the other hand, the bordism formula for the push-forward in smooth rational cohomology, Lemma \ref{bbetwzew}, gives
$$\hat p^A_!\circ \hat \ch(\hat x)=\hat p_!\left(\hat \hA^c(o) \cup \hat
  \ch(\hat x)\right) =\hat p_!\left(\hat \hA^c(\tilde o)_{|W}\cup \hat
  \ch(\hat y)_{|W}\right) 
=-a\left(\int_{V}  \hA^c(\tilde o)\wedge R(\hat y)\right)\ .$$
These two formulas imply that $\Delta_p=0$.
\end{proof}

\subsubsection{}

We now finish the proof of Proposition \ref{point}.
We claim that there exists $c\in \nat$   such that $c[\xi]=0$.
In view of Lemma \ref{rtrtrt} we then have
 $$0=\Delta_{cp} =c\Delta_p\ ,$$ and this implies the Proposition since the target $\R/\Q$ of $\Delta_p$
is a $\Q$-vector space.

Note that the graded ring  $\Omega^{Spin^c}_*\otimes \Q$ is concentrated in even degrees.
Using that $\Omega^{SO}_*\otimes \Q$ is concentrated in even degrees, one can
see this as follows.
In \cite[p.~352]{MR0248858} it is shown that the homomorphism
$Spin^c\to U(1)\times SO$ induces an injection $\Omega^{Spin^c}_*\to \Omega^{SO}_*(BU(1))$.
Since $H_*(BU(1),\Z)\cong \Z[z]$ with $\deg(z)=2$ lives in even degrees, we see using the Atiyah-Hirzebruch spectral sequence that $\Omega^{SO}(BU(1))\otimes \Q$ lives in even degrees, too. This implies that
$\Omega^{Spin^c}_*\otimes \Q$ is concentrated in even degrees.

Since $H_*(\Z\times BU,\Z)$ is also concentrated in even degrees it follows again from the Atiyah-Hirzebruch spectral sequence that $\Omega^{Spin^c}_*(\Z\times BU)\otimes \Q$ is concentrated in even degrees. 

Since $[\xi]$ is of odd degree we conclude the claim that $c[\xi]=0$ for an appropriate $c\in \nat$.
\end{proof}
This finishes the proof of Proposition \ref{point}. \end{proof}

\subsubsection{}

We now consider the general case. Let $p\colon W\to B$ be a proper submersion with closed fibres
with a topological $K$-orientation.
\begin{prop}\label{prop:DeltaNull}
We have $\Delta_p=0$.
\end{prop}
We give the proof in the next couple of subsections.
\subsubsection{}

For a closed oriented manifold $Z$ let $\PD\colon H^*(Z,\Q)\stackrel{\sim}{\to} H_*(Z,\Q)$ denote the Poincar{\'e} duality isomorphism. 
\begin{lem}\label{lem13000}
The group
$H_*(B,\Q)$ is generated by classes of the
form $f_*\left(\PD(\hA^c(TZ))\right)$, where $Z$ is a closed $Spin^c$-manifold
and $f\colon Z\rightarrow B$. 
\end{lem}
\begin{proof}
We consider the sequence of transformations of homology theories
$$\Omega^{Spin^c}_*(B)\stackrel{\alpha}{\to} K_*(B) \stackrel{\beta}{\to} H_*(B,\Q)\ .$$
The transformation $\alpha$ is the $K$-orientation of the $Spin^c$-cobordism theory, and 
$\beta$ is the homological Chern character. We consider all groups as $\Z/2\Z$-graded.
The homological Chern character is a rational isomorphism.
Furthermore one knows by \cite{MR658506}, \cite{math.KT/0701484} that $\Omega^{Spin^c}_*(B)\stackrel{\alpha}{\to} K_*(B)$
is surjective.
It follows that the composition
$$\beta\circ \alpha\colon \Omega^{Spin^c}(B)\otimes\Q\to H^*(B,\Q)$$
is surjective. An explicit description of $\beta\circ \alpha$ is  given as follows.
Let $x\in \Omega^{Spin^c}(B)$ be represented by a map $f\colon Z\to B$ from
a closed $Spin^c$-manifold $Z$ to $B$. Let
$\PD\colon H^*(Z,\Q)\stackrel{\sim}{\to} H_*(Z,\Q)$ denote the Poincar{\'e} duality isomorphism.
Then we have
$$\beta\circ \alpha(x)=f_*\left(\PD(\hA^c(TZ))\right)\ .$$\end{proof}

\subsubsection{}

For the proof of Proposition \ref{prop:DeltaNull} we first consider the case that $p$ has even-dimensional fibres, and that
$x\in K^0(W)$.
By Lemma \ref{lem13000}, in order to show that $\Delta_p(x)=0$, it suffices to show that
all evaluations 
$\Delta_p(x)\left(f_*(\PD(\hA^c(TZ)))\right)$ vanish.
In the following, if $x$ denotes a $K$-theory class, then $\hat x$ denotes a smooth $K$-theory class
such that $I(\hat x)=x$.

  We choose a representative $o_q$ of a  smooth 
$K$-orientation which refines the topological $K$-orientation 
of the map $q\colon Z\rightarrow *$ induced by the $Spin^c$-structure on $TZ$.
Furthermore, we consider  the diagram with a cartesian square
$$\xymatrix{V\ar@/_1cm/[dd]^s\ar[d]^r\ar[r]^F&W\ar[d]^p\\Z\ar[d]^q\ar[r]^f&B\\ {*}&}\ .$$
In the present case
$\Delta_p(x)\in H^{odd}(B,\R/\Q)$, and  we can assume that $Z$ is odd-dimensional. We calculate 
\begin{eqnarray*}
\Delta_p(x)\left(f_*(\PD(\hA^c(TZ)))\right)
&=&f^*\Delta_p(x)\left(\PD(\hA^c(TZ))\right)\\
&\stackrel{Lemma \:\ref{lemnat33}}{=}&
\Delta_r(F^*x)\left(\PD(\hA^c(TZ))\right)\\
&=&
(\hA^c(\nabla^{TZ})\cup \Delta_r(F^*x))[Z]\\
&=&\int_{Z}\hA^c(o)\wedge \Delta_r(F^*x)\\
&=&\hat q_!\left(\hat \hA^c(o_q)\cup  \Delta_r(F^*x)\right)\\
&=&\hat q_!^A\left( \Delta_r(F^*\hat x)\right)\\
&=&\hat q_!^A\left[\hat \ch\circ \hat r_!(F^* \hat x)-\hat r^A_!\circ \hat \ch(F^*\hat x)\right]\\
&=&\hat q_!^A\circ\hat \ch\circ \hat r_!(F^*\hat x) -\hat s^A_!\circ \hat
\ch(F^*\hat x)\\  
 &\stackrel{Proposition \: \ref{point}}{=}&
\hat \ch\circ \hat q_!\circ \hat r_!(F^*\hat x)-\hat s^A_!\circ \hat
\ch(F^*\hat x)\\
&=&\hat \ch\circ \hat s_!(F^*\hat x)-\hat s^A_!\circ \hat \ch(F^*\hat x)\\
&=&\Delta_s(F^*x)\\
&\stackrel{Proposition \: \ref{point}}{=}&0\ .
\end{eqnarray*}
We thus have shown that $$0=\Delta_p\colon K^0(W)\to H^{odd}(B,\R/\Q)$$ if $p$ has even-dimensional fibres.

\subsubsection{}

If $p$ has odd-dimensional fibres and $x\in K^1(W)$, then we can choose $y\in K^0(S^1\times W)$ such that
$(\hat\pr_2)_!(y)=x$. Since $p\circ \pr_2$ has even-dimensional fibres we get
using the Lemmas \ref{functk} and \ref{funktt} 
\begin{eqnarray*}
\Delta_p(x)&=&\hat \ch\circ \hat p_!\circ (\hat \pr_2)_!(\hat y)-\hat p^A_!\circ \hat \ch\circ (\hat\pr_2)_!(\hat y)\\
&\stackrel{Lemma \:\ref{s1}}{=}&\hat \ch\circ (\widehat{p\circ \pr_2})_!(\hat
y)-\hat p^A_!\circ (\hat\pr_2)^A_!\circ \hat \ch(\hat y)\\&=&\hat \ch\circ
(\widehat{p\circ \pr_2})_!(\hat y)-(\widehat{p\circ \pr_2})^A_!\circ \hat
\ch(\hat y)\\
&=&\Delta_{p\circ \pr_2}(y)\\
&=&0\ .
\end{eqnarray*}
Therefore $$0=\Delta_p\colon K^1(W)\to H^{odd}(B,\R/\Q)$$ 
if $p$ has odd-dimensional fibres.

\subsubsection{}

Let us now consider the case that $p$ has even-dimensional fibres, and that $x\in K^1(W)$. In this case we consider the diagram
$$
\begin{CD}
  S^1\times W @>{\Pr_2}>> W\\
  @VV{t:=\id_{S^1}\times
   p}V @VV{p}V\\
  S^1\times B @>{\pr_2}>> B
\end{CD}
\ .$$
We choose a class $y\in K^0(S^1\times W)$ such that $(\Pr_2)_!(y)=x$. We further  choose a smooth refinement
$\hat y\in \hat K^0(S^1\times W)$ of $y$ and  set $\hat x:=(\hat\Pr_2)_!(\hat
y)$. 
Then we calculate using the Lemmas \ref{functk} and \ref{funktt}
\begin{eqnarray*}
\Delta_p(x)&=&\hat \ch\circ \hat p_!(\hat x)-\hat p^A_!\circ \hat \ch(\hat
x)\\ 
&=&\hat \ch\circ \hat p_!\circ (\hat\Pr_2)_!(\hat y)-\hat p_!^A\circ \hat \ch\circ (\hat\Pr_2)_!(\hat y)\\
&\stackrel{Lemma \:\ref{s1}}{=}&\hat \ch\circ \hat p_!\circ (\hat
\Pr_2)_!(\hat y)-\hat p_!^A\circ (\hat\Pr_2)^A_!\circ \hat \ch\circ (\hat y)\\
&=&\hat \ch\circ (\widehat{p\circ \Pr_2})_!(\hat y)-(\widehat{p\circ \Pr_2})^A_!\circ \hat \ch(\hat y)\\
&=&\hat \ch\circ (\widehat{\pr_2\circ t})_!(\hat y)-(\widehat{\pr_2\circ
  t})^A_!\circ \hat \ch(\hat y)\\
&=&\hat \ch\circ \hat\pr_{2!}\circ \hat t_!(\hat y)-\hat\pr_{2!}^A\circ \hat
t^A_!\circ \hat \ch(\hat y)\\ 
&\stackrel{Lemma \:\ref{s1}}{=}&(\hat\pr_2)^A_!\left[\hat \ch\circ \hat
  t_!(\hat 
y)-\hat t^A_!\circ \hat \ch(\hat y)\right]\\
&=&(\hat\pr_2)^A_!\circ \Delta_t(y)=0\ .
\end{eqnarray*}
Therefore
$$0=\Delta_p\colon K^1(W)\to H^{ev}(B,\R/\Q)$$
if $p$ has even-dimensional fibres.

\subsubsection{}

In the final case $p$ has odd-dimensional fibres and $x\in K^0(W)$.
In this case we consider the sequence of projections
$$S^1\times S^1\times W\stackrel{\pr_{23}}{\to} S^1\times W\stackrel{\pr_2}{\to} W\ .$$
We choose a class $y\in K^0(S^1\times S^1\times W)$ such that
$(\pr_2\circ \pr_{23})_!(y)=x$. We further choose a smooth refinement $\hat
y\in \hat  K^0(S^1\times S^1\times W)$ of $y$ and set $\hat x:=(\widehat{\pr_2\circ \pr_{23}})_!(\hat y)$.
Then we calculate using the already known cases and the Lemmas \ref{functk} and \ref{funktt},
\begin{eqnarray*}
\Delta_p(x)&=&\hat \ch\circ \hat p_!(\hat x)-\hat p_!^A\circ \hat \ch(\hat
x)\\ 
&=& \hat \ch\circ \hat p_!\circ (\hat\pr_2)_!\circ (\hat\pr_{23})_!(\hat
y)-\hat p_!^A\circ \hat \ch\circ (\hat\pr_2)_!\circ (\hat\pr_{23})_!(\hat
y)\\&=& 
 \hat \ch\circ (\widehat{p\circ \pr_2})_!\circ (\hat\pr_{23})_!(\hat y)- \hat p_!^A\circ \hat \ch\circ (\widehat{\pr_2\circ
\pr_{23}})_!(\hat y)\\
&=&(\widehat{p\circ \pr_2})_!^A\circ \hat \ch\circ (\hat\pr_{23})_!(\hat
y)-\hat p_!^A\circ (\widehat{\pr_2\circ
\pr_{23}})_!^A\circ \hat \ch(\hat y)\\
&=&(\widehat{p\circ \pr_2})_!^A\circ \Delta_{\pr_{23}}(\hat y)\\
&\stackrel{Lemma \:\ref{s1}}{=}&0\ .
\end{eqnarray*}
This finishes the proof of Theorem \ref{main}.
\hB

\section{Conclusion}
\label{sec:backm-bibl-bibl}

We have now constructed a geometric model for smooth K-theory, built out of
geometric families of Dirac-type operators. We equipped it with a compatible
multiplicative structure, and we have given an explicit construction of a
push-down map for fibre bundles with all the expected properties. For the
verification of these properties we heavily used local index theory.

We presented a collection of natural examples of smooth K-theory classes and
showed in particular that several known secondary analytic-geometric
invariants can be understood in this framework very naturally. This involved
also the consideration of bordisms in this framework.

Finally, we constructed a smooth lift of the Chern character and proved a
smooth version of the Grothendieck-Riemann-Roch theorem. This also involved
certain considerations from homotopy theory which are special to K-theory.

Important open questions concern the construction of equivariant versions of
this theory, or even better versions which work for orbifolds or similar
singular spaces.

In a different direction, we have  addressed the construction of geometric
models of smooth bordism theories along similar lines in \cite{BSSW07}; using singular bordism 
this has also been achieved for smooth ordinary cohomology in \cite{BKS09}.

\backmatter

\bibliographystyle{smfalpha}



\providecommand{\bysame}{\leavevmode ---\ }
\providecommand{\og}{``}
\providecommand{\fg}{''}
\providecommand{\smfandname}{\&}
\providecommand{\smfedsname}{{\'e}ds.}
\providecommand{\smfedname}{{\'e}d.}
\providecommand{\smfmastersthesisname}{M{\'e}moire}
\providecommand{\smfphdthesisname}{Th{\`e}se}

\end{document}